\documentclass[a4paper, 11pt]{article}

\usepackage[T1]{fontenc}
\usepackage[latin1]{inputenc}
\usepackage{url}

\usepackage[backref=page,colorlinks,bookmarksopen=true]{hyperref}
\usepackage{amsfonts,amsmath,amssymb}
\usepackage{amsthm}
\usepackage{enumerate}
\usepackage{amsrefs}

\usepackage{mathrsfs, xspace}
\usepackage{stmaryrd}
\usepackage{a4wide}
\usepackage[all]{xy}

\usepackage{pgf,tikz}
\usetikzlibrary{arrows}

\theoremstyle{plain}
\newtheorem{proposition}{Proposition}[section]
\newtheorem{theorem}[proposition]{Theorem}
\newtheorem{conj}[proposition]{Conjecture}

\newtheorem{lemma}[proposition]{Lemma}
\newtheorem{corollary}[proposition]{Corollary}

\newtheorem*{claim*}{Claim}

\theoremstyle{definition}
\newtheorem{definition}[proposition]{Definition}

\theoremstyle{remark}
\newtheorem{remark}[proposition]{Remark}

\DeclareMathOperator{\Fix}{Fix}

\DeclareMathOperator{\Aut}{Aut}
\DeclareMathOperator{\Out}{Out}
\DeclareMathOperator{\Inn}{Inn}
\DeclareMathOperator{\Stab}{Stab}
\DeclareMathOperator{\Ker}{Ker}
\DeclareMathOperator{\proj}{proj}

\DeclareMathOperator{\Haar}{Haar}

\DeclareMathOperator{\Ch}{Ch}
\DeclareMathOperator{\con}{con}
\DeclareMathOperator{\Res}{Res}
\DeclareMathOperator{\Sym}{Sym}

\newcommand{\tdlc}{t.d.l.c.\xspace}

\renewcommand{\d}{\mathrm{d}}

\DeclareMathOperator{\GL}{GL}
\DeclareMathOperator{\SL}{SL}
\DeclareMathOperator{\PSL}{PSL}
\DeclareMathOperator{\PGL}{PGL}

\DeclareMathOperator{\Isom}{Isom}

\DeclareMathOperator{\pr}{pr}

\newcommand{\eps}{\varepsilon}

\newcommand{\RR}{\mathbf{R}}                          
\newcommand{\NN}{\mathbf{N}}                          
\newcommand{\ZZ}{\mathbf{Z}}                          
\newcommand{\FF}{\mathbf{F}}

\newcommand{\scrS}{\mathscr W}
\newcommand{\scrF}{\mathscr F}
\newcommand{\op}{\mathrm{op}}

\newcommand{\bfH}{\mathbf H}
\newcommand{\bfG}{\mathbf G}
\newcommand{\bfV}{\mathbf V}

\newcommand{\LL}{\operatorname{L}^0}

\newcommand{\D}{\Delta}

\begin{document}

\title{On the linearity of lattices in affine buildings\\ and ergodicity of the singular Cartan flow}
%
\author{Uri Bader, Pierre-Emmanuel Caprace and  Jean L\'ecureux}

\maketitle

\begin{abstract}
Let $X$ be a locally finite irreducible affine building of dimension~$\geq 2$ and $\Gamma \leq \Aut(X)$ be a discrete group acting cocompactly. The goal of this paper is to address the following question: When is $\Gamma$ linear? More generally, when does $\Gamma$ admit a finite-dimensional representation with infinite image over a commutative unital ring? If $X$ is the Bruhat--Tits building of a simple algebraic group over a local field and if $\Gamma$ is an arithmetic lattice, then $\Gamma$ is clearly linear. We prove that if $X$ is of type $\widetilde A_2$, then the converse holds. In particular,  cocompact lattices in exotic $\widetilde A_2$-buildings are   non-linear. As an application, we obtain the first infinite family of lattices in exotic $\widetilde A_2$-buildings of arbitrarily large thickness, providing   a partial answer to a question of W.~Kantor from 1986. We also show that if $X$ is Bruhat--Tits of arbitrary type, then the linearity of $\Gamma$ implies that $\Gamma$ is virtually contained in the linear part of the automorphism group of $X$; in particular $\Gamma$ is an arithmetic lattice. The proofs are based on the machinery of algebraic representations of ergodic systems recently developed by U.~Bader and A.~Furman. The implementation of that tool in the present context requires the geometric construction of a suitable ergodic $\Gamma$-space attached to the the building $X$, which we call the \emph{singular Cartan flow}.
\end{abstract}

\tableofcontents

\section{Introduction}

\subsection{Are lattices in affine buildings linear?}

Let $X$ be a locally finite affine building with a cocompact automorphism group, and $\Gamma < \Aut(X)$ be a lattice. The goal of this work is to address the following question: When does $\Gamma$ admit a  faithful finite-dimensional linear representations  over a field, or more generally over a commutative ring with unity? A group admitting  such a representation is called \textbf{linear}.

Let us review some known facts concerning that question.

The most important locally finite affine buildings are the so-called \textbf{Bruhat--Tits buildings}. Such a building $X$ is associated with a semi-simple algebraic group $\mathbf G$ over a local field $k$. The dimension of $X$ coincides with the $k$-rank of $\mathbf G$. When $\mathbf G$ is $k$-simple, the building $X$ is irreducible. In that case,   if we assume moreover that $\dim(X)\geq 2$,  celebrated theorems by Margulis \cite{Margulis} and Venkataramana \cite{Venka} ensure  that lattices in $\mathbf G(k)$ are all arithmetic. It is tempting to deduce that lattices in Bruhat--Tits buildings of dimension~$\geq 2$ are classified, but this is actually not the case. Indeed, the full automorphism group $\Aut(X)$ of the Bruhat--Tits building $X$ of $\mathbf G(k)$ is isomorphic to  $\Aut(\mathbf G(k))$ (assuming that the $k$-rank of $\mathbf G$ is~$\geq 2$), which is an extension of the group $\Aut(\mathbf G)(k)$ of algebraic automorphisms over $k$ by a subgroup $\Aut(k)$ (see \cite{BoTi} and Proposition~\ref{prop:Aut} below). As long as $\Aut(k)$ is infinite --- which happens if and only if the local field $k$ is of positive characteristic --- a lattice in $\Aut(X)$ can potentially have infinite image in $\Aut(k)$ and, hence, fail to satisfy the hypotheses of the arithmeticity theorem. A lattice in $\Aut(X)$ with infinite image in $\Aut(k)$ will be called a \textbf{Galois lattice}. Galois lattices exist in rank~$1$ (see Proposition~\ref{prop:RankOneGaloisLattice} below), while their existence in higher rank is an open question at the time of writing.

By the seminal work of Tits (see the last corollary in \cite{Tits_cours84}), irreducible locally finite affine building of dimension~$\geq 3$ are all  Bruhat--Tits buildings of  simple algebraic groups over local fields. This is not the case in low dimensions.

If $\dim(X)=1$, namely if $X$ is a tree  with vertex degrees~$\geq 2$ ,  any uniform lattice in $\Aut(X)$ is linear,   indeed virtually free. However, a non-uniform   lattice  can be non-linear.

If $\dim(X)=2$ and $X$ is reducible, then   $X$ is the product of two trees. In that case  Burger and Mozes \cite{BurgerMozes_lattices} constructed celebrated examples of uniform lattices $\Gamma < \Aut(X)$ that are simple groups, and thus cannot be linear, as they fail to be residually finite.    It is still unknown whether an  irreducible lattice in a product of more than two topologically simple groups, each acting cocompactly  on a tree, can be non-linear (in this context, the condition of \textit{irreducibility} means that the projection of the lattice to each simple factor is faithful with dense image).

If $\dim(X)=2$ and $X$ is irreducible, but not Bruhat--Tits, then we say that $X$ is \textbf{exotic}. Some of the  known constructions of exotic affine buildings and lattices acting on them  are referenced in \S\ref{sec:ReviewExotic} below.

The simplest and most studied  exotic affine buildings are those  of type $\widetilde{A}_2$. They will be of core interest in this paper. Let us recall that (locally finite) $\widetilde{A}_2$-buildings may be characterized among simply connected $2$-dimensional simplicial complexes by the property that the link of every vertex is the incidence graph of a (finite) projective plane; this follows from \cite[Th.~7.3]{CharneyLytchak}. %
%
From the classification, it follows that the locally finite Bruhat--Tits buildings of type $\widetilde{A}_2$ are precisely those associated with $\mathrm{PGL}_3(D)$, where $D$ is a finite dimensional division algebra over a local field $k$ (see \cite[Chapter~28]{WeissBook}). Thus a locally finite $\widetilde A_2$-building is exotic if and only if it is not isomorphic to the Bruhat--Tits building of such a group $\mathrm{PGL}_3(D)$.

Our first main result, whose proof occupies the major part of this paper, is the following.

\begin{theorem}\label{nonlinear}
Let $X$ be a locally finite $\widetilde A_2$-building and $\Gamma < \Aut(X)$ be a discrete group acting cocompactly. Assume that $X$ is not isomorphic to the building associated to $\mathrm{PGL}_3(D)$, with $D$ a finite dimensional division algebra over a local field.

Then, for any commutative unital ring $R$ and any $n \geq 1$, any homomorphism $\Gamma \to \GL_n(R)$ has a finite image.
\end{theorem}

We shall also establish a similar result for Galois lattices in arbitrary Bruhat--Tits buildings of  dimension~$\geq 2$.

\begin{theorem}\label{thm:nonlinearGalois}
Let $\mathbf G$ be a $k$-simple algebraic group defined over a local field $k$, of $k$-rank~$\geq 2$. Let $\Gamma$ be a lattice in $\Aut(\mathbf G(k))$.

If $\Gamma$ has infinite image in $\Out(\mathbf G(k))$, then for any commutative unital ring $R$ and any   $n\geq 1$, any homomorphism $\Gamma \to \GL_n(R)$ has a finite image.
\end{theorem}

Combining the latter theorem with the aforementioned Arithmeticity Theorem, we draw the following consequence.

\begin{corollary}\label{cor:Galois}
	Let $X$ be the Bruhat--Tits building associated with a $k$-simple algebraic group $\mathbf G$ defined over a local field $k$, of $k$-rank~$\geq 2$.
	
	A lattice $\Gamma \leq \Aut(X)$ admits a finite-dimensional linear representation with infinite image if and only if $\Gamma$ is virtually contained (necessarily as an arithmetic subgroup) in the image of $\mathbf G(k)$ in $\Aut(X)$.
\end{corollary}

In the case of $\widetilde A_2$-buildings, we obtain a similar statement without the hypothesis that the ambient building is Bruhat--Tits:

\begin{corollary}\label{cor:equiv}
Let $X$ be a locally finite $\widetilde A_2$-building and $\Gamma \leq \Aut(X)$ be a discrete group acting cocompactly.

Then $\Gamma$ admits a representation with infinite image of dimension~$n\geq 1$ over a commutative unital ring  if and only if $X$ is Bruhat--Tits and $\Gamma$ is arithmetic.
\end{corollary}

We mention that by an unpublished work of Yehuda Shalom and Tim Steger we know that every non-trivial normal subgroup of a lattice in an $\widetilde{A}_2$-building is of finite index. Clearly if the lattice is linear then finite index normal subgroups are in abundance. The question of finding   finite quotients of lattices in possibly  exotic affine buildings was already asked by W. Kantor \cite[Problem~C.6.2]{Kantor84}. Some computer experiments we held indicate that some non-linear lattices might have no normal subgroups at all. We suggest the following.

\begin{conj}\label{conj:simplicity}
Let $X$ be a locally finite affine building of type $\widetilde A_2$. Assume that $X$ is not isomorphic to the building associated to $\mathrm{PGL}_3(D)$, where $D$ is a finite dimensional division algebra over a local field. Let $\Gamma$ be a lattice in $\Aut(X)$.
Then $\Gamma$ is virtually simple.
\end{conj}

It should be emphasized that, while there are several known constructions of locally finite $\widetilde A_2$-buildings with lattices, it is generally a delicate problem to determine whether or not a given such building arising from one of those constructions is Bruhat--Tits. Actually, until very recently, the only examples that could be proved to be exotic are buildings of thickness~$\leq 10$, and in most cases the proof of exoticity relied on computer calculations.  As an   application of Theorem~\ref{nonlinear}, presented in Section~\ref{app:ExoticLargeOrder}, we will establish an exoticity criterion for a large concrete family of $\widetilde A_2$-lattices which was first defined by Ronan~\cite[\S3]{Ronan_triangle} and  Kantor~\cite[\S C.3]{Kantor84} and more recently investigated by Essert \cite{Essert}. In particular, we obtain the following. 

\begin{corollary}\label{cor:ExistenceExotic}
	Given $q$ a power of $2$ which is not a power of  $8$, there exists an $\widetilde A_2$-building $X$ of order $q$ that is not Bruhat--Tits and admits an automorphism group $\Gamma \leq \Aut(X)$ acting freely and transitively on the set of panels of each type.
\end{corollary}

That result provides the first infinite family of $\widetilde A_2$-lattices in exotic buildings of arbitrarily large thickness, and may be viewed as a partial answer to a question of W.~Kantor, who wrote in the remarks following Corollary~C.3.2 in \cite{Kantor84}: ``\textit{It would be very interesting to know which, if any, of these buildings are ``classical'' ones obtained from $\PSL_3(K)$ for complete local (skew) fields $K$}.'' After a first version of this paper was circulated, two more sources of infinite families of $\widetilde A_2$-lattices in exotic buildings have been found: see Theorem~C and Corollary~E in \cite{Radu_GeomDed}, and Proposition~7 and Remark~8 in \cite{CapHypT}.

\subsection{Proof ingredients}

The proofs of both Theorems~\ref{nonlinear} and~\ref{thm:nonlinearGalois} establish the contrapositive assertion, and thus start by assuming  that the lattice $\Gamma$ has a linear representation with infinite image over a ring.

The first step of the proof is a reduction showing that the given representation can be assumed to have its image contained in a $k$-simple algebraic group $\mathbf G$ over a local field $k$, with an unbounded and Zariski dense image. That step is actually a consequence of an abstract statement of independent interest, valid at a high level of generality, and presented in Appendix~\ref{app:RingLinear}.

At that stage, the proof of Theorem~\ref{nonlinear} reduces to that of the following.

\begin{theorem}\label{thm:nonlinear-local}
Let $X$ be a locally finite $\widetilde A_2$-building and $\Gamma \leq  \Aut(X)$ a discrete group consisting of type preserving automorphisms which acts cocompactly. Assume that  there exist   a local field $k$, a connected adjoint $k$-simple $k$-algebraic group $\mathbf{G}$ and a group homomorphism $\Gamma\to \mathbf{G}(k)$ with an unbounded and Zariski dense image. Then $X$ is a Bruhat--Tits building.
\end{theorem}

\paragraph*{Analogy with Margulis' Superrigidity Theorem.}

Our strategy to approach Theorem~\ref{thm:nonlinear-local} is inspired by the recent proof of the Margulis Superrigidity  Theorem by the first author and A.~Furman, developed in \cite{BF-Superrigid}. The  tools developed in that work are designed to study a representation $\rho \colon\Gamma\to  \mathbf J(k)$ of an arbitrary countable group $\Gamma$ to a $k$-simple algebraic group over a local field $k$, with Zariski dense image.  A key notion from \cite{BF-Superrigid}  is that of a \textbf{gate}  (see \S\ref{sub:BF} below for more details). This gate is an algebraic variety canonically associated with any standard Borel space $Y$ equipped with a measure-class preserving ergodic action of $\Gamma$. The gate moreover provides a natural representation from $\Aut_\Gamma(Y)$ to a subquotient of the algebraic group $\mathbf J(k)$, which is non-trivial  under suitable assumptions on $\Gamma$ and $Y$.

Therefore, in order to prove Theorem~\ref{thm:nonlinear-local}, we need  to provide an ergodic $\Gamma$-space $Y$   such that $\Aut_{\Gamma}(Y)$ does not admit any non-trivial representation. The space $Y$ we shall construct will be called the \textbf{singular Cartan flow}.
In order to give a better understanding of its construction, let us first explain what this space is  in the `classical' case when $\Gamma$ is a lattice in $\SL_3(k)$, where $k$ is a local field.

Let $A$ be the subgroup of $G = \SL_3(k)$ consisting of diagonal matrices, and $S$ the singular torus
$$S=\left\{ \begin{pmatrix}
a & 0& 0\\
0&a&0\\
0&0&a^{-2}
\end{pmatrix}\mid a\in k^*\right\}$$
We consider the $\Gamma$-spaces $G/A$ and $G/S$ (or the actions of $A$ and $S$ on $\Gamma\backslash G$). Ergodicity is then provided by the Howe--Moore Theorem.  The gate of the space $Y=G/S$ gives a representation of $N_G(S)/S\simeq \PGL_2(k)$.

\paragraph*{Geometric analogues in the exotic case.}

Now let us go back to the setting of Theorem~\ref{thm:nonlinear-local}, where $\Gamma$ is a lattice of an   $\widetilde A_2$-building $X$ that is possibly exotic. In particular the whole group $\Aut(X)$ could be discrete, and the $\Gamma$-spaces we need to construct can thus hardly be homogeneous spaces of a non-discrete locally compact enveloping group of $\Gamma$. The spaces we shall construct are of a geometric nature; the measures will be constructed explicitly by hand, and a crucial step in our proof will be to establish their ergodicity.

In order to construct a geometric analogue of the space $G/A$, we first interpret this space geometrically. The group $A$ acts 
as a group of translations (with compact kernel) on some model apartment $\Sigma$ of the Bruhat--Tits building $X$ associated to $G$. It can be viewed as the stabilizer of this oriented apartment. Hence, $G/A$ is the $G$-orbit of this oriented apartment in $X$, and since $G$ acts transitively, $G/A$ is the space 
of all such apartments. We can also go up one step  and consider $G$, equipped with the action of $G$ (left multiplication) and $A$ (right multiplication). As a dynamical system, this is analogous to the set of simplicial embeddings of the Coxeter complex $\Sigma$ into $X$, with the action of $G$ by post-composition and the action of $A$ by pre-composition.

Therefore, in our setting, we are led consider the set $\scrF$ consisting of all simplicial embeddings of the Coxeter complex $\Sigma$ of type $\widetilde A_2$ to $X$. 
The group $\ZZ^2$ acts on the set $\scrF$ via its action on $\Sigma$, while $\Gamma$ acts on $\scrF$ via its action on $X$. In particular these two actions commute, and we get a $\ZZ^2$-action on the compact space $\scrF/\Gamma$. The latter action is called the \textbf{Cartan Flow}.  A measure on $\scrF$ is constructed in \S\ref{subsec:d2} and \S\ref{subsec:measgen}. Its construction will be natural enough so that this measure is invariant under both the action of $\Gamma$ and the action of $\ZZ^2$.

A geometric analogue of the  space $G/S$ from the classical picture is more technical to realize. Let us again describe it first in the classical setting. We want to understand $G/S$ both as a $G$-space and a $N_G(S)$-space, with both actions commuting. The group $S$ acts by translations (with compact kernel) on some singular line $\ell \subset X$. The set of all lines parallel to $\ell$ form a regular tree $T$, which is a  so-called \emph{panel tree} of $X$. This tree is the Bruhat--Tits tree of the Levi  subgroup $N_G(S)$. Hence, an analogue of the $G\times N_G(S)$-dynamical system $G$ is given by a set of embeddings of $T\times\ell$ into $X$. However, since we want the action of $G$ to be transitive, we must consider only a $G$-orbit  of embeddings. This amounts to ensure that we embed $T$ ``as a Bruhat--Tits tree'' of $N_G(S)$. 

In the geometric setting, we therefore introduce the space $\widetilde \scrS$ of all embeddings of $T\times \ell$ into our building. It is endowed with an action of $\Gamma$ (by postcomposition), and $\ZZ$ (by precomposition with translates of the line $\ell$). However, we want to restrict to the analogue of the $G$-orbit.
 In order to do so, we use   the \textbf{projectivity group}  of the projective plane at infinity of $X$. That group naturally acts  by automorphisms on $T$, and we denote by $P$ its closure in $\Aut(T)$ (see below for more about this group). We introduce the \emph{singular Cartan flow} $\scrS$ as the orbit of an element by the (commuting) actions of $\Gamma$, $P$ and $\ZZ$.  We again define a natural measure on this space, which takes into account both the geometric measure on $\scrF$ and the Haar measure on $P$ (see \S\ref{subsec:dpm} for more details).

\paragraph*{Ergodicity results}

A fundamental step in our proof of Theorem~\ref{thm:nonlinear-local} is  the following (see Theorem \ref{CartanFlow} and Theorem \ref{ergodic} for more precise statements).

 \begin{theorem}
 The Cartan flow and the singular Cartan flow are ergodic.
 \end{theorem}

As mentioned above, those flows are a not necessarily homogeneous spaces of some locally compact groups; in particular the ergodicity cannot be deduced from   the Howe--Moore Theorem as it is done in the classical picture.  The relevant tool happens to be the one originally used in the study of the geodesic flow on manifolds with variable negative curvature, namely Hopf's  argument    \cite{Hopf} (see also \cite{Coudene} for a nice introduction). In Appendix~\ref{app-hopf} we give an abstract version of the Hopf argument which is suited for our purpose. The Cartan flow is not needed for the end of the proof, but it is an intermediate step in the proof of the ergodicity of the singular Cartan flow.

If $Y$ denotes the singular Cartan flow, the gate theory provides us with a linear representation of $\Aut_{\Gamma}(Y)$ (the latter group is $\PGL_2(k)$ in the classical case). It follows from the construction of the singular Cartan flow that the group $\Aut_{\Gamma}(Y)$ naturally contains the group $P$ introduced above, which is the closure of the natural image in $\Aut(T)$ of the projectivity group  of the projective plane at infinity of $X$. This projectivity group has a lot more structure than $\Gamma$: it acts on a locally finite tree $T$ (a so-called \emph{panel tree} of $X$) with a $3$-transitive action on its boundary. The final step in the proof  of Theorem~\ref{thm:nonlinear-local} will be deduced from the following result of independent interest (see Theorem~\ref{thm:charBT}):

\begin{theorem}\label{thm:ProjIntro}
Let  $X$ be a locally finite $\widetilde A_2$-building and  $P$ be the closure of the projectivity group in the locally compact group $\Aut(T)$. If an open subgroup of $P$ admits a faithful continuous representation $P\to \GL_n(k)$, where $k$ is a local field, then $X$ is a Bruhat--Tits building.
\end{theorem}

It should be underlined that, while all the other steps of the proof generalize with only small modifications to    Euclidean buildings of  types other than  $\widetilde A_2$, we do not know how to extend Theorem~\ref{thm:ProjIntro}  to $\widetilde C_2$ or $\widetilde  G_2$ buildings. Indeed the proof of the latter relies on a result on projective planes due to A.~Schleiermarcher (Theorem~\ref{thm:Schleiermacher}), whose analogue for generalized quadrangles and generalized hexagons is still unknown. 
In this regard, we make the following remark.

\begin{remark}
Let  $X$ be a locally finite 2-dimensional affine building and $\Gamma$ a discrete cocompact group of isometries of $X$. Let $P$ be the closure of the projectivity group of $X$ in the locally compact group $\Aut(T)$.
By the methods of our proof of Theorem~\ref{nonlinear} below, one could show that if $P$ is 
non-linear then also $\Gamma$ is non-linear. 
While we do not know in general that $P$ is non-linear for exotic locally finite buildings of type $\widetilde C_2$ or $\widetilde G_2$, given a concrete example of such a building it might be possible to check the non-linearity of $P$ directly. For example, in the panel-regular exotic buildings of type $\widetilde C_2$ considered in  \cite{Kantor81} and \cite{Essert}, some of the residues are non-classical finite generalized quadrangles of order $(q-1, q+1)$ for a prime power $q$. In the case where $q \geq 7$, it would  suffice to check that the projectivity groups of those quadrangles contain the full alternating group to ensure that the closure of the projectivity group in the automorphism group of the panel tree is locally symmetric or alternating, hence non-linear as a consequence of the main results from \cite{Radu_classif}.
\end{remark}

\medskip
Let us now briefly discuss the proof of Theorem~\ref{thm:nonlinearGalois}, which is presented in Section~\ref{app:Galois}. It  also relies on the Bader--Furman machinery from \cite{BF-Superrigid}. The starting point of that proof is the observation that the argument for Theorem~\ref{nonlinear} sketched above also provides relevant information when $X$ is the Bruhat--Tits building of $\PGL_3(D)$. Indeed, while in that case the projectivity group is isomorphic to $\PGL_2(D)$, and is thus linear by hypothesis, it turns out that the group $\Aut_\Gamma(Y)$ of automorphisms of the Cartan flow commuting with $\Gamma$ considered above can be strictly larger than the projectivity group. More precisely, if $\Gamma$ is a Galois lattice, then $\Gamma$ has an infinite (hence non-virtually abelian by property (T)) image in the quotient $\Aut(X)/  \PGL_3(D)$, which is virtually the group of field automorphisms of $X$. It then turns out that $\Aut_\Gamma(Y)/\PGL_2(D)$ is also a non-abelian group of field automorphisms of $\PGL_2(D)$. On the other hand, the gate theory provides us with a continuous faithful representation of the locally compact group $\Aut_\Gamma(Y)$, and by the general structure theory of linear locally compact groups from \cite{CapStu}, this implies that $\Aut_\Gamma(Y)/\PGL_2(D)$ must in fact be virtually abelian. This shows that if $\Gamma$ is linear, then it cannot be a Galois lattice. This part of the argument does not require $X$ to be of type $\widetilde A_2$, and is actually valid for arbitrary irreducible Bruhat--Tits buildings of dimension~$\geq 2$. Indeed  the relevant $\Gamma$-spaces are   homogeneous spaces of the non-discrete group $\Aut(X)$, and their ergodicity can be established by a suitable application of the Howe--Moore Theorem.

\subsection{A review of exotic $2$-dimensional buildings with lattices}\label{sec:ReviewExotic}

The aforementioned classification of irreducible affine buildings of dimension~$\geq 3$ was completed in the early days of the 1980's. The existence of exotic  $2$-dimensional affine buildings came soon after. Some of them have no non-trivial automorphism whatsoever (see e.g. \cite{Ronan_free}); this is in particular the case of \emph{generic} $\widetilde A_2$-buildings (in a suitably defined Baire categorical sense), see \cite[Th.~5]{BarrePichot}. Some have a cocompact automorphism group, with a global fixed vertex at infinity, see \cite[\S7]{HVM}. The existence of such a fixed point prevents that group from being unimodular (see \cite[Th.~M]{CaMo}), and hence cannot contain any lattice. The known exotic $\widetilde A_2$-buildings relevant to Theorem~\ref{nonlinear}  are those endowed with a proper cocompact action of a discrete group. They are constructed as universal covers of finite simple complexes having finite projective planes as residues, or of suitably defined finite complexes of groups. The constructions that we are aware of are the following (the \textbf{order} of a locally finite $\widetilde A_2$-building is the common order of its \textbf{residue planes}, i.e. the finite projective planes appearing as its rank~$2$ residues).

\begin{itemize}
\item In \cite[\S3.1]{Tits_Andrews} and \cite[Th.~2.5]{Ronan_triangle}, Tits and Ronan construct four $\widetilde A_2$-buildings of order~$2$ with cocompact lattices. The lattices act \textbf{regularly} (i.e. sharply transitively) on the set of chambers of their building. It is shown  in  \cite[\S3.1]{Tits_cours85} and \cite{KMW}  that two of those are Bruhat--Tits buildings with arithmetic lattices, while the other two are exotic (see    also  \cite[\S3.6]{Kantor84}).

\item In \cite[\S3.1]{Tits_Andrews}  and \cite[\S3.2]{Tits_cours85}, Tits constructs 44 buildings of type $\widetilde A_2$, of order $8$, with cocompact lattices acting regularly on the set of chambers of their building. It is asserted that 43 of them are exotic.

\item In \cite[\S3]{Ronan_triangle} and \cite[\S C.3]{Kantor84}, Ronan and Kantor describe a framework providing   in particular cocompact lattices  of $\widetilde A_2$-buildings  acting regularly on the panels of each given type. Such lattices are called \textbf{panel-regular}.

\item In \cite{CMSZ2}, Cartwright--Mantero--Steger--Zappa construct $65$ exotic $\widetilde A_2$-buildings of order $3$  with cocompact lattices. In those examples, the lattices act regularly on the set of vertices of their building; in particular the types of vertices are permuted cyclically.

\item In \cite[\S3]{Barre_exotic}, Barr\'e construct an exotic $\widetilde A_2$-building of order $2$  with a cocompact lattice. The full automorphism group of that building is discrete, and has exactly two orbits of vertices.

\item In \cite{Essert}, Essert  provides  presentations of arbitrary  panel-regular lattices in $\widetilde A_2$-buildings. When the building is of order $2$, there are exactly two inequivalent presentations of such groups. One of them gives rise to an arithmetic group, the other is exotic (see Remark~\ref{rem:residually-2} and Lemma~\ref{lem:G_0} below). The $\widetilde A_2$-buildings of order~$3$ and $5$ admitting  a panel-regular lattice have recently been classified by S.~Witzel (see \cite{Witzel}), who determined which are exotic and described their full automorphism group.  We show in Section~\ref{app:ExoticLargeOrder} how to use Theorem~\ref{nonlinear} to establish  Corollary~\ref{cor:ExistenceExotic} and, more generally, to construct, from a given  panel-regular lattice in an exotic $\widetilde A_2$ of order $p$, an infinite family of panel-regular lattices in an exotic $\widetilde A_2$-building of order $p^n$. 

\item In all of the previous examples, the residue planes are always Desarguesian. Recently, N.~Radu \cite{Radu} constructed an example of an exotic $\widetilde A_2$-building of order~$9$, with a vertex-transitive, virtually torsion-free, discrete automorphism group and with non-Desarguesian residues (isomorphic to the Hughes plane of order~$9$).

\end{itemize}

After the first version of the present paper was circulated, two more sources of exotic $\widetilde A_2$-buildings with lattices have been identified:
\begin{itemize}

\item In Proposition~7 and Remark~8 in \cite{CapHypT}, it is shown that for every non-Desarguesian finite projection plane $\mathscr P$, there is a locally finite $\widetilde A_2$-building $X$ with a discrete group $\Gamma \leq \Aut(X)$ acting cocompactly, such that $X$ has a residue plane isomorphic to $\mathscr P$. In particular $X$ must be exotic. Since there exist non-Desarguesian finite projective planes of arbitrarily large order, this provides another infinite family of lattices in exotic   $\widetilde A_2$-buildings of unbounded order. It is not known whether the lattices constructed in this way are virtually torsion-free. 

\item In Theorem~C and Corollary~E in \cite{Radu_GeomDed}, it is proved that, among the panel-regular lattices in $\widetilde A_2$-buildings studied by Essert and mentioned above, the proportion of the exotic ones tends (exponentially fast) to $1$ as the order of the building tends to infinity. In particular, it follows from Theorem~\ref{nonlinear}  that the ``{vast majority}'' of panel-regular lattices in $\widetilde A_2$-buildings considered by Essert are non-linear. 
\end{itemize}

Remark that in all those examples, the lattices are cocompact: no exotic example with a non-uniform lattice is known.

Exotic buildings of type $\widetilde C_2$ and $\widetilde G_2$ have been less investigated, but some constructions are known, see \cite{Kantor81} and \cite{Essert}. In fact,   the most fascinating example of an exotic $2$-dimensional affine building with a cocompact lattice is probably the building of type $\widetilde G_2$ constructed by Kantor \cite{Kantor81} as the universal cover of a simplicial complex associated with the finite sporadic simple group of Lyons (see also \cite[\S3.5]{Tits_cours85}). The corresponding building is exotic (an assertion  attributed to Tits, without proof, in \cite{Kantor81}). This can be deduced from the fact that the full automorphism group of that building is discrete, a fact established by Gr\"uninger \cite[Cor.~3.3]{Gruninger}.

\subsection{Structure of the paper}

The next two sections are of introductory nature and review known material.
The next section, \S\ref{sec:buildings}, serves as a concise introduction to $\widetilde A_2$-buildings.
In particular we recall the notions of {\em wall trees} and {\em panel trees} which are of great importance in this work.
In \S\ref{sec:projgroup} we recall the notion of the {\em projectivity group} and explain how it acts on panel trees.
In subsection \S\ref{subsec:treesandalggroups} we discuss linear representations and algebraicity of groups acting on trees.
In \S\ref{sec:repofproj} we use the results of the previous section in order to study linear representations of the projectivity group.
The main result of this section is Theorem~\ref{thm:charBT} which roughly states that the projectivity group of a building admits a linear representation if and only if the building is Bruhat--Tits.
This result may be of independent interest. The proof of Theorem~\ref{thm:nonlinear-local} will be carried eventually by a reduction to this theorem.

The next three sections, \S\ref{sec:topspaces}, \S\ref{sec:meas} and \S\ref{sec:ergodicity}
are devoted to the construction and the study of some new spaces, the most important of which are the {\em the space of marked flats} and {\em space of (restricted) marked wall trees}.
We define these as topological spaces in \S\ref{sec:topspaces} and introduce their measured structure in \S\ref{sec:meas}. Finally, in \S\ref{sec:ergodicity}, we prove ergodicity results for these spaces.
The ergodicity of the singular Cartan flow, Theorem~\ref{ergodic}, is the technical heart of this paper.

In \S\ref{sec:nonlinear} we combine Theorem~\ref{ergodic} with {\em the theory of algebraic gates} in order to reduce the proof of Theorem~\ref{thm:nonlinear-local} to Theorem~\ref{thm:charBT}.
We then proceed and prove Theorem~\ref{nonlinear}.

The last two sections contain supplementary results.
In \S\ref{app:Galois} we consider Galois lattices.
The goal of that section  is to establish Theorem~\ref{thm:nonlinearGalois}. 
Finally, in  \S\ref{app:ExoticLargeOrder}, we use the results of the paper to construct a concrete infinite family of exotic $\widetilde A_2$-lattices, a result recorded in  Corollary~\ref{cor:ExistenceExotic}.

The paper has two appendices.
They contain   essential ingredients in the proof of the main theorems of the paper,
which are presented as appendices due to their independent characters.
The first is \S\ref{app:RingLinear} which is devoted to the proof of Theorem~\ref{thm:ringlinear}.
This theorem is the main ingredient in the reduction of Theorem~\ref{nonlinear} to Theorem~\ref{thm:nonlinear-local}.
In \S\ref{app-hopf} we review the classical Hopf argument in an abstract setting. In particular we prove Theorem~\ref{hopf}, which is essential in the proofs given in \S\ref{sec:ergodicity}.

\subsection{Acknowledgements}

The contribution of Alex Furman to the developments of the methods we use in this paper is indispensable.
We thank him  for many hours of discussions. We are also grateful to both referees for their comments on an earlier version of the manuscript.
UB acknowledges the support of ERC grant \#306706.
PEC ackowledges the support of F.R.S.-FNRS and of ERC grant \#278469.  JL is supported in part by ANR grant ANR-14-CE25-0004 GAMME and ANR-16-CE40-0022-01 AGIRA.

\section{$\widetilde A_2$-buildings and their boundaries} \label{sec:buildings}

In this section we recall some basic facts on $\widetilde A_2$-buildings and their boundaries.
In \S\ref{subsec:buildingsintro} we give a general introduction to $\widetilde A_2$-buildings and in \S\ref{subsec:buildingsboundaries} we discuss their boundaries.
In \S\ref{subsec:paneltree} we discuss the important notions of {\em wall trees} and {\em panel trees}.
Most of the results we will state here are well known and we will not provide proofs in these cases.
Standard references on the subject are \cite{AbramenkoBrown}, \cite{Garrett} and \cite{Ronan}.

\subsection{Buildings of type $\widetilde A_2$} \label{subsec:buildingsintro}

Let us start by recalling the definition of a building of type $\widetilde A_2$  (or $\widetilde A_2$-building for short).
We denote by $\Sigma$ a \textbf{model apartment} : it is the $2$-dimensional simplicial complex afforded by the tessellation of the Euclidean plane $\RR^2$ by equilateral triangles. A \textbf{wall} in $\Sigma$ is a line in $\Sigma$ which is a union of edges. We fix a vertex $0\in \Sigma$.

\begin{definition}
A \textbf{building} of type $\widetilde A_2$ is a triangular complex satisfying the two following conditions, in which we call \textbf{apartment} a subcomplex of $X$ isomorphic to $\Sigma$:

\begin{itemize}
\item Any two simplices of $X$ are contained in an apartment,
\item If $F$ and $F'$ are two apartments, there exists a simplicial isomorphism $F\to F'$ which fixes pointwise $F\cap F'$.
\end{itemize}
\end{definition}

The $2$-dimensional simplices of an $\widetilde A_2$-building are called \textbf{chambers}. The space $X$ is endowed with a natural metric, defined as follows: the distance between two points $x$ and $y$ is the distance between $x$ and $y$ calculated in any apartment containing $x$ and $y$. This is well-defined, and furthermore gives a CAT(0) metric on $X$.
Apartments are also called \textbf{flats}, because in the CAT($0$)-metric on $X$, the apartments are maximal subspaces isometric to a Euclidean space.

An $\widetilde A_2$-building $X$ is equipped with a coloring $X^{(0)} \to \{0, 1, 2\}$ of its vertex set $X^{(0)}$ such that each chamber has exactly one vertex of each color. The color of a vertex is called its \textbf{type}.

We assume throughout that the building $X$ is  \textbf{thick}, i.e.   there is more than two chambers adjacent to each edge.

It is worth noting that the local geometry is forced by the geometry of the building. For a start, the building is \textbf{regular}: the cardinality of the set of chambers adjacent to an edge is constant (see for example \cite[Theorem 2.4]{ParkinsonHecke}). When that number is finite, we define the \textbf{order} of the $\widetilde A_2$-building $X$ as the number $q$ such that the number of chambers adjacent to each edge equals $q+1$. Since we assumed $X$ to be thick we have $q>1$.

Furthermore, the link at a vertex is always a finite projective plane, and the boundary of the building is a compact (profinite) projective plane. For the convenience of the reader, we recall the definition\footnote{Strictly speaking, the definition we give is the definition of the incidence graph of a projective plane; the identification of a projective plane and its incidence graph implicit throughout this paper should not cause any confusion.}:

\begin{definition}\label{def:ProjectivePlane}
A \textbf{projective plane} is a bipartite graph, with two types of vertices (called \textbf{points} and \textbf{lines}), such that
\begin{itemize}
\item Any two lines are adjacent to a unique point,
\item Any two points are adjacent to a unique line.
\item There exist four points, no three of which are adjacent to a common line.
\end{itemize}
\end{definition}

Projective planes are also known as generalized triangles, or thick buildings of type $A_2$.

Another striking feature of $\widetilde A_2$- (or more general affine) buildings is the following fact:

\begin{theorem}\label{aptfull}
Let $\Omega$ be a subset of $X$ which is either convex or of non-empty interior. If $\Omega$ is isometric to a subset of $\RR^2$, then there exists an apartment $F$ such that $\Omega\subset F$.
\end{theorem}

An important class of subsets of $X$ to which we will apply this theorem is the following:

\begin{definition}
A \textbf{sector} based at $0$ in $\Sigma$ is a connected component of the complement in $\Sigma$ of the union of the walls passing through $0$.

A sector of $X$ is a subset of $X$ which is isometric to a sector in $\Sigma$.
\end{definition}

\subsection{Boundaries of buildings} \label{subsec:buildingsboundaries}

Let $X$ be a building of type $\widetilde A_2$.  Although we shall primarily view $X$ as a $2$-dimensional simplicial complex, we will occasionally refer to $X$ as a CAT($0$) metric space. Such a metric realization is obtained by giving each each of $X$ length~$1$, and endowing each $2$-simplex with the  metric of a Euclidean  equilateral triangle with side length~$1$. Since $X$ is a CAT($0$) space, its ideal boundary $\partial X$ is naturally endowed with the structure of a CAT($1$) space. The combinatorial nature of $X$ is such that $\partial X$ inherits a much finer structure, namely that of a compact projective plane, equipped with a   family of natural measures.

\subsubsection{The projective plane at infinity}

A boundary point $v \in \partial X$ is called a \textbf{vertex at infinity} if it is the endpoint of a geodesic ray parallel to a wall in $X$. Two sectors  $S$ and $S'$ in $X$ are \textbf{equivalent} if their intersection contains a sector; equivalently their ideal boundaries $\partial S$ and $\partial S'$ coincide. An equivalence class of sectors (or, equivalently, the ideal boundary of a sector) is called a \textbf{Weyl chamber}.

Two Weyl chambers are called \textbf{adjacent} if they contain a common vertex at infinity. This happens if and only if they are represented by two sectors whose intersection contains a geodesic ray. We let $\Delta$ denote the graph whose vertices are the vertices at infinity, such that two vertices define an edge if the corresponding vertices at infinity lie in a common Weyl chamber. The set of edges of $\Delta$ is denoted by $\Ch(\Delta)$. The ideal boundary $\partial X$ may be viewed as a CAT($1$) metric realization of the graph $\Delta$.

%

We can finally state the fact alluded to above.

\begin{theorem}
The graph $\Delta$  is a projective plane.
\end{theorem}

A key point in the proof of this theorem is to prove that two chambers are always in some apartment at infinity. An intermediate step is the following lemma, which is also useful for us:

 \begin{lemma}
 Let $x$ be a vertex in $X$, and $C\in \Ch(\Delta)$. There exists a unique sector based at $x$ in the equivalence class of $C$. We denote this sector by $Q(x,C)$.
 \end{lemma}

Two chambers in $\Delta$ are called \textbf{opposite} if they are at maximal distance (\textit{i.e.} at distance three). Similarly, two vertices are called opposite if they are at maximal distance.

By convexity of apartments, the convex hull of two chambers is contained in any apartment containing both of them. This explains the following.

 \begin{lemma}\label{opp-apt}
 Two opposite chambers are contained in a unique apartment.
 \end{lemma}

 We note for further use the following standard fact \cite[Proposition 29.50]{WeissSphBook}:

 \begin{lemma}\label{lem:twoop}
 For every pair of chambers $x,y\in \Ch(\Delta)$, there exists a chamber $z\in \Ch(\Delta)$ which is opposite to both $x$ and $y$.
 \end{lemma}

\subsubsection{Topology on $\Delta$}\label{sec:TopologyAtInfinity}

The projective plane $\Delta$, seen as the boundary of $X$, is endowed with a natural topology: indeed, $\Delta$ is naturally endowed with the structure of a compact projective plane.  One way to topologize $\Delta$ consists in observing that the ideal boundary $\partial X$ with its natural $\mathrm{CAT}(1)$ metric is a metric realization of the graph $\Delta$, so that a vertex of $\Delta$ is a point in $\partial X$ and an edge of $\Delta$ is a geodesic segment in $\partial X$. Moreover, the vertex-set of $\Delta$ is a closed subset of the ideal boundary $\partial X$ endowed with the cone topology. Thus the vertex-set of $\Delta$ inherits a compact topology from the cone topology on $\partial X$. Similarly, the edge-set of $\Delta$, denoted by $\Ch(\Delta)$, can be topologized via the map the sends each edge of $\Delta$ to its midpoint in $\partial X$.  In that way the vertex-set and edge-set of $\Delta$ are naturally endowed with a topology that is compact and totally disconnected (see \cite[Prop.~3.5]{CaMo_coctamen}). A basis of open sets of  the topology on $\Ch(\Delta)$  is provided by the sets of the form $\Omega_x(y):=\{C \in \Ch(\Delta) \mid y\in Q(x,C)\}$, where $x$ and $y$ are vertices of $X$.

%
%

We note the following fact for further use.

\begin{lemma}\label{lem:oppopen}
Let $C$ be a chamber. The set of chambers which are opposite to $C$ is an open subset of $\Ch(\Delta)$.
\end{lemma}

\begin{proof}
Fix an apartment $A$ containing $C$, and let $C'$ be the chamber in $A$ opposite to $C$. Let $O$ be the set of chambers opposite to $C$. Then $O$ is the union of sets of the form $\Omega_x(y)$, where $x$ is a vertex in $A$ and $y$ is a vertex in the sector $Q(x,C')$.
\end{proof}

\subsection{Wall trees and panel trees}\label{subsec:paneltree}

An important tool in our study of triangular buildings are \textbf{wall trees} and \textbf{panel trees}. These closely related notions were introduced by Tits in his study of the structure and classification of affine buildings. A good reference is \cite[Chapter 10 and 11]{WeissBook}.

Let us start by defining wall trees. These trees are defined as equivalence classes of parallel lines. A geodesic line $l$ is said \textbf{singular} if it parallel to a wall in some (hence any) apartment containing it, otherwise it is called \textbf{regular}.

\begin{definition}
Two geodesic lines $\ell$ and $\ell'$ are called \textbf{parallel} if they are contained in a common apartment, and are parallel in this apartment. The \textbf{distance} $d(\ell,\ell')$ between parallel lines is the Euclidean distance between them in any apartment containing them.
\end{definition}

Parallelism is an equivalence relation. Furthermore, the distance between parallel lines turns each class into a metric space.

For regular lines, this space is just isometric to $\RR$. However, these spaces turn out to be interesting for singular lines \cite[Proposition 10.14 and Corollary 10.16]{WeissBook}:

\begin{proposition}\label{prop:walltree}
Let $\ell$ be a singular line, and $T_\ell$ be the set of lines parallel to $\ell$. Then $T_\ell$ is a  thick tree. If $X$ is locally finite of order~$q$, then $T_\ell$ is regular of degree~$q+1$.
\end{proposition}

The tree $T_\ell$ defined above is called a \textbf{wall tree}. Note that every geodesic line parallel to $\ell$ has the same two endpoints at infinity as $\ell$. These two endpoints are by definition vertices (or panels) of the building at infinity $\Delta$.

Now we turn to the definition of a panel tree. Let $v$ be a vertex of $\Delta$. Then $v$ is represented by an equivalence class of geodesic rays, two of them being equivalent if they contain  subrays that are contained in a common apartment and are parallel in that apartment. 

\begin{definition}
Two  geodesic rays are said to be \textbf{asymptotic} if their intersection contains a ray.
\end{definition}

If two rays $r,r':[0,+\infty)\to X$ belong to the same equivalence class, their \textbf{distance} is defined as
$$\inf_s\lim\limits_{t\to+\infty} d(r(t+s),r'(t)).$$
It is clear that this distance only depends on the asymptotic classes of $r$ and $r'$. It is possible to check that we have indeed a distance on the set of asymptotic classes of rays in the class of $v$.

\begin{proposition}\label{prop:paneltree}
The metric space, denoted $T_v$, of asymptotic classes of rays in the class of $v$, is a thick tree. If $X$ is locally finite of order~$q$, then $T_v$ is regular of degree~$q+1$.
\end{proposition}

This proposition can be found in \cite[Corollary 11.18]{WeissBook}. The tree $T_v$ is called the \textbf{panel tree} associated to $v$.

Panel trees are related  to wall trees via canonical maps \cite[Proposition 11.16]{WeissBook}:

\begin{proposition}\label{isompaneltree}
Let $\ell$ be a singular geodesic line, and $v$ be an endpoint of $\ell$. Then the map $T_\ell\to T_v$ which associates to a line parallel to $\ell$ its asymptotic class is an isometry.
\end{proposition}

The set of ends of a panel tree is also well understood. Any ray in $T_v$ can be lifted to a sequence of adjacent parallel geodesic rays in some apartment $F$ containing $v$. To such a sequence one can associate the chamber, in the boundary of $F$, which contains $v$ and goes in the direction of the sequence of rays. This explains the construction of the bijection in the following proposition \cite[Proposition 11.22]{WeissBook}:

\begin{proposition}\label{bndpaneltree}
Let $v$ be a vertex in $\Delta$. There is a canonical $\Aut(X)_v$-equivariant bijection between the set of ends of the panel tree $T_v$  and the set $\Ch(v)$ of chambers of $\Delta$ containing $v$.
\end{proposition}

\section{Structure of the projectivity group} \label{sec:projgroup}

The projectivity group is an important invariant of the building at infinity $\Delta$, and plays a crucial role in our proofs.
In this chapter we  define it and discuss its properties.
In the next section, \S\ref{sec:repofproj}, we will study its linear representations.
First,  in \S\ref{sub:defproj} we  define the projectivity group in the context of a general projective plane, and then, in \S\ref{sec:projtree},
in the restricted setting of an affine building.
In the latter setting, we will explain how the projectivity group acts on a panel tree.
The last subsection, \S\ref{subsec:treesandalggroups},  is devoted to the general study of linear representations of groups acting on trees.
This is a preparation for \S\ref{sec:repofproj}.

\subsection{Projectivity groups of projective planes}\label{sub:defproj}

In this section, we recall some basic terminology on projectivity groups, and record important known results that will be needed in the sequel.

Let $\Delta$ be a thick building of type $A_2$ (i.e. $\Delta$ is a projective plane). If $v$ is a vertex of $\Delta$, we denote by $\Ch(v)$ the set of chambers of $\Delta$ adjacent to $v$.

Let $v$ and $v'$ be opposite vertices of $\Delta$. The \textbf{combinatorial projection} to $v$ is the map from $\Delta$ to $\Ch(v)$ which associates to $C\in \Delta$ the unique chamber of $\Ch(v)$ which is at minimal distance from $C$. The fact that we chose $v$ and $v'$ opposite implies that this projection, restricted to $\Ch(v')$, is an involutory bijection.

Given a sequence $(v_0, \dots, v_k)$ of vertices such that $v_i$ is opposite $v_{i-1}$ and $v_{i+1}$ for all $i \in \{1, \dots, k-1\}$, we obtain a bijection of $\Ch(v_0)$ to $\Ch(v_k)$ by composing the successive projection maps from $\Ch(v_{i-1})$ to $\Ch(v_i)$.  That bijection is called a \textbf{perspectivity}; it is denoted by $[v_0; v_1 ; \dots ; v_k]$. In the special case when  $v_0 =  v_k$, it is called a \textbf{projectivity} at $v_0$. The collection of all projectivities at a vertex $v_0$ is denoted by $\Pi(v_0)$; it is  a permutation group of the set $\Ch(v_0)$, called the \textbf{projectivity group of $\Delta$ at $v_0$} and denoted by $(\Pi(v_0), \Ch(v_0))$. The isomorphism type of the permutation group $(\Pi(v_0), \Ch(v_0))$ does not depend on the choice of the vertex $v_0$. Indeed, given another vertex $v$, there exists a perspectivity from $\Ch(v_0)$ to $\Ch(v)$ (because $\Delta$ is thick, see Lemma~\ref{lem:twoop}), which conjugates $(\Pi(v_0), \Ch(v_0))$ onto $(\Pi(v), \Ch(v))$. For that reason, it makes sense to define \textbf{the projectivity group} of the building $\Delta$ as any representative $(\Pi, \Omega)$ of the isomorphism class of the projectivity group $(\Pi(v), \Ch(v))$ at a vertex $v$.

\begin{proposition}\label{prop:3trans}
The projectivity group  $(\Pi, \Omega)$  of a  thick building of type $A_2$  is $3$-transitive.
\end{proposition}
\begin{proof}
See Proposition~2.4 in \cite{CameronBook}.
\end{proof}

We remark that projectivity groups can be defined for arbitrary generalized polygons; a result of N.~Knarr~\cite[Lemma~1.2]{Knarr} ensures that they are always $2$-transitive. 

The projectivity group $(\Pi, \Omega)$ is an important invariant of a projective plane $\Delta$. For example, it is known that $(\Pi, \Omega)$ is sharply $3$-transitive if and only if $\Delta$ is a projective plane over a commutative field, see Theorem~2.5 in \cite{CameronBook}. We will need an  important variant of that fact. In order to state it, we recall that a $2$-transitive permutation  group $(\Pi, \Omega)$ is called a \textbf{Moufang set} if there exists a set $\{U_x\}_{x \in \Omega}$ of subgroups of $\Pi$ such that for each $x \in \Omega$, the group $U_x$  acts sharply transitively on $\Omega \setminus \{x\}$ and $g U_x g^{-1} = U_{g(x)}$ for all $g \in \Pi$. The subgroups   $U_x$ for $x \in \Omega$ are called the \textbf{root groups} of the Moufang set.

\begin{theorem}\label{thm:Schleiermacher}
Let $\Delta$ be a thick building of type $A_2$. Then $\Delta$ is Moufang if and only if the projectivity group $(\Pi, \Omega)$ is a Moufang set.
\end{theorem}
\begin{proof}
See the main result from \cite{Schleiermacher}.
\end{proof}

We will also need the following elementary criterion, ensuring that a projectivity of a very specific form is non-trivial.

\begin{lemma}\label{lem:NonTriv}
Let  $\Delta$ be a thick building of type $A_2$. Let $C, C_0, C_1, C_2, C_3$ be chambers such that $\Sigma_{i, j} = C_i \cup C \cup C_j$    is a half-apartment  of $\Delta$ for all $(i, j) \in \{0, 2\} \times \{1, 3\}$.

For all $i \mod 4$, let $v_i$ be the vertex of $C_i$ that is not contained in $C$. 
If $C_0 \neq C_2$ and $C_1 \neq C_3$, then the projectivity $[v_0; v_1; v_2; v_3; v_0]$ is non-trivial: indeed, its only fixed point in $\Ch(v_0)$ is  $C_0$.
\end{lemma}

\begin{proof}
Observe first that the projectivity 	$g = [v_0; v_1; v_2; v_3; v_0]$  fixes $C_0$. 
Let $C'_0 \in \Ch(v_0)$ be any chamber different from $C_0$. We claim that  $g = [v_0; v_1; v_2; v_3; v_0]$ does not fix $C'_0$. Indeed, we define inductively, for each $i=0, 1, 2, 3$, the chamber $C'_{i+1}$ as the projection of $C'_i$ to $\Ch(v_{i+1})$, and  $D_i$ as the unique chamber which is adjacent to both $C'_i$ and $C'_{i+1}$. It then follows that for all $i = 0, 1, 2$, the chamber $D_{i}$ is adjacent to $D_{i+1}$. Moreover $D_3$ is adjacent to $C'_4$.

By definition, we have $g(C'_0) = C'_4$. Therefore, if $g(C'_0) = C'_0$, then $D_3$ is adjacent to $D_0$, so that $(D_0, D_1, D_2, D_3, D_0)$ forms a closed gallery of length~$4$. Therefore, we must have $D_0= D_1$ or $D_1 = D_2$, which forces respectively $C_0 = C_2$ or $C_1 = C_3$.
\end{proof}

\subsection{Projectivity groups of Euclidean buildings}\label{sec:projtree}

Now, let us go back to the situation when $X$ is a thick  Euclidean building of type $\widetilde A_2$, and $\Delta$ is the set of chambers of its building at infinity.  In that particular case, the projectivity group $(\Pi, \Omega)$ of $\Delta$ is more than just a permutation group of a set $\Omega$. Indeed, to each vertex $v \in \Delta$, one associates the panel tree $T_v$, as in Proposition \ref{prop:paneltree}.

Let $v$ and $v'$ be two opposite vertices of different types in $\Delta$, and let $\ell$ be a geodesic line in $X$ joining these two vertices.
By Proposition \ref{isompaneltree}, the panel trees $T_v$ and $T_{v'}$ are both canonically isomorphic to the wall tree $T_\ell$.
Hence we have a  canonical involutory isomorphism between the panel trees $T_v$ and $T_{v'}$.

From that observation, it follows that any perspectivity of $\Ch(v)$ to $\Ch(v')$ defines an isomorphism of panel trees from $T_v$ to $T_{v'}$, and that the projectivity group at $v$ acts by automorphisms on the panel tree $T_v$.
We define the \textbf{projectivity group of $X$ at $v$} as $(\Pi(v), T_v)$ by viewing the usual projectivity group $\Pi(v)$ as a subgroup of $\Aut(T_v)$.
Thus the projectivity group of $\Delta$ at $v$ is   $(\Pi(v), \partial T_v)$, since the set $\Ch(v)$ is canonically isomorphic to the set of ends of the panel tree $T_v$ by Proposition \ref{bndpaneltree}.
Note that using isomorphisms as described above, all wall trees and all panel trees are mutually isomorphic (though not canonically).

\begin{definition} \label{def:modeltree}
We fix once and for all a tree $T$ which is isomorphic to all panel and wall trees in $X$ and call it the \textbf{model   tree} of $X$.
By choosing an actual identification of $T$ with a panel tree $T_v$ and pulling back the projectivity group $\Pi(v)$ we define $\Pi$, the \textbf{model projectivity group} of $X$.
\end{definition}

\begin{lemma} \label{lem:modeltree}
The projectivity group $\Pi$ acts $3$-transitively on $\partial T$ and vertex-transitively on $T$. In particular $T$ is regular. If moreover $X$ is locally finite of order $q$, then $T$   is locally finite of degree~$q+1$).
\end{lemma}

\begin{proof}
The $3$-transitivity of $\Pi$ on $\partial T$ follows from Proposition~\ref{prop:3trans}.
$\Pi$ is vertex-transitive on $T$, since every vertex can be uniquely determined by a triple of distinct ends.
The regularity of $T$ follows. The last assertion follows from Propositions~\ref{prop:walltree} and~\ref{isompaneltree}.
\end{proof}

\subsection{Algebraic groups and locally finite trees} \label{subsec:treesandalggroups}

In this section we recall some general results on groups acting on trees that we need in our analysis of the group of projectivities of an $\widetilde A_2$-building.

\subsubsection{Linear simple locally compact groups are algebraic}

A locally compact group is called \textbf{linear} if it has a continuous, faithful, finite-dimensional representation over a locally compact field.

\begin{theorem}\label{thm:LocallyLinear}
Let $G$ be a   \tdlc{} group such that the intersection $G^+$ of all non-trivial closed normal subgroups of $G$ is compactly generated, topologically simple and non-discrete, and that the quotient $G/G^+$ is compact. If $G$ has a linear open subgroup,  then the following assertions hold.
\begin{enumerate}[(i)]
\item $G^+$ is algebraic: indeed $G^+$ is isomorphic to $H(k)/Z$, where $k$ is a local field, $H$ is an absolutely simple, simply connected algebraic group defined and isotropic over $k$, and $Z$ is the center of $H(k)$.

\item $G/G^+$ is virtually abelian.
\end{enumerate}
\end{theorem}
\begin{proof}
Assertion (i) follows from Corollary~1.4 from \cite{CapStu}. Assertion (ii) follows from Theorem~1.1 in  \cite{CapStu} and the fact that the compact open subgroups of $G^+$ are not solvable (indeed they are Zariski-dense in a simple algebraic group).
\end{proof}

\subsubsection{Boundary-transitive automorphism groups of trees}

The following result, due to Burger--Mozes  \cite{BurgerMozes}, provides a natural set-up where all hypotheses of Theorem~\ref{thm:LocallyLinear} other than the linearity are satisfied.

\begin{theorem}\label{thm:BuMo}
Let $T$ be a thick locally finite tree and $G \leq \Aut(T)$ be a closed subgroup acting $2$-transitively on the set of ends $\partial T$.

Then the intersection $G^+$ of all non-trivial closed normal subgroups of $G$ is compactly generated, topologically simple and non-discrete. It acts edge-transitively on $T$ and $2$-transitively on $\partial T$, and the quotient $G/G^+$ is compact. In particular, $G$ is unimodular.
\end{theorem}

\begin{proof}
See Proposition 3.1.2 from \cite{BurgerMozes} and Theorem~2.2 from \cite{CaDM2}.
For the unimodularity of $G$, observe that $G^+$ must be in the kernel of the modular character of $G$. Thus the modular character of $G$ factors through the quotient $G/G^+$, which is compact. Hence $G$ is indeed unimodular.
\end{proof}

An automorphism of a tree $T$ is called \textbf{horocyclic} if it fixes pointwise a horosphere of $T$.
Combining the previous two statements, we obtain a subsidiary result which is a key tool in the proof of our characterization.

\begin{proposition}\label{prop:aux}
Let $T$ be a thick locally finite tree and $G \leq \Aut(T)$ be a closed subgroup acting $2$-transitively on the set of ends $\partial T$. If $G$ has a linear open subgroup, then the following assertions hold.

\begin{enumerate}[(i)]
\item The group $G^+$, defined as in Theorem~\ref{thm:BuMo}, is isomorphic to $H(k)/Z$, where $k$ is a local field, $H$ is an absolutely simple, simply connected algebraic $k$-group of $k$-rank~$1$, and $Z$ is the center of $H(k)$.

\item $(G, \partial T)$ is a Moufang set, whose root groups are the unipotent radicals of proper $k$-parabolic subgroups of $G^+$.

\item Each element $u \in G$ which is a horocyclic automorphism of $T$ belongs to some root group.

\item The quotient $G/G^+$ is virtually abelian.
\end{enumerate}

\end{proposition}

\begin{proof}
By Theorems~\ref{thm:LocallyLinear} and~\ref{thm:BuMo}, we know that $G^+$, defined as in Theorem~\ref{thm:BuMo}, is isomorphic to $H(k)/Z$, where $k$ is a local field, $H$ is an absolutely simple, simply connected algebraic $k$-group of $k$-rank~$>0$. Since $G^+$ acts properly and cocompactly on a thick locally finite tree, its $k$-rank must be~$1$ (see for example \cite[Corollary~4]{Tits_LieKolchin}). This proves (i).

It is well known, and easy to see using maximal compact subgroups, that $G^+$ has, up to isomorphism, a unique continuous proper cocompact action by automorphisms on a thick locally finite tree (see e.g. Lemma~2.6(viii) from \cite{CaDM1} for a proof). Thus $T$ is isomorphic to the Bruhat--Tits tree of $G^+$. It follows that $(G^+, \partial T)$ is a Moufang set, whose root groups are the unipotent radicals of proper $k$-parabolic subgroups of $G^+$. Since $G^+$ is normal in $G$, it follows that the root groups are permuted by conjugation under $G$. Therefore, the permutation group $(G, \partial T)$ is also a Moufang set. This proves (ii).

In order to establish (iii), we recall that the \textbf{contraction group} of an element $h \in G$ is defined as
$$\con(h) = \{ g \in G \; | \; h^n g h^{-n} \to 1\}.$$
Let now $u \in G$ be a horocyclic element, and let $e \in \partial T$ be the center of the horosphere of $T$ fixed by $u$. Given any hyperbolic element $h \in G$ whose repelling fixed point is $e$, we have $u \in \con(h)$. We can choose such a hyperbolic $h$ in $G^+$ (this is possible because $G^+$ is $2$-transitive on $\partial T$).  The image of $\con(h)$ in the compact quotient $G/G^+$ must be trivial, so that $\con(h) \leq G^+$. We now invoke   \cite[Lemma~2.4]{Prasad}, ensuring that $\con(h)$ is a root group of the Moufang set  $(G, \partial T)$.   This proves (iii). The assertion (iv) holds by Theorems~\ref{thm:LocallyLinear}(ii) and~\ref{thm:BuMo}.
\end{proof}

\section{Structure and representations of the group of projectivities} \label{sec:repofproj}

The goal of this section is to prove Theorem~\ref{thm:charBT} which roughly states that the projectivity group is linear if and only if the building $X$ is a Bruhat--Tits building.
The proof of this theorem will be carried in \S\ref{subsec:projgrplinear}.
It relies  on Proposition~\ref{prop:Unipotent} below, whose proof is our main concern in \S\ref{subsec:horocyclic}.

\subsection{Existence of horocyclic elements} \label{subsec:horocyclic}

This subsection is devoted to the proof of the following technical result.

\begin{proposition}\label{prop:Unipotent}
Let $X$ be a  thick Euclidean building of type $\widetilde A_2$. Then the projectivity group   contains a non-trivial horocyclic automorphism of  the panel tree defined in Definition~\ref{def:modeltree}.
\end{proposition}

We need to collect a number of subsidiary facts, assuming throughout that $X$ is a  thick Euclidean building of type $\widetilde A_2$ whose projective plane at infinity is denoted by $\Delta$. The proof of Proposition~\ref{prop:Unipotent} is deferred to the end of the section.

\begin{lemma}\label{lem:FixedPoint}
Let $e_0, e_1, e_2, e_3$ be vertices of $\Delta$, and $D_0, D_1, D_2, D_3 \subset X$ be geodesic lines such that the two endpoints of the line $D_i$ are $e_i$ and  $e_{i+1}$ for all $i$, where the indices are written modulo~$4$.

If  $D_i \cap D_{i+1}$ is non-empty for all $i$ modulo~$4$, then the projectivity $[e_0; e_1; e_2; e_3; e_0]$ fixes the point of the panel tree $T_{e_0}$ represented by $D_0$.
\end{lemma}
\begin{proof}
Let $x \in T_{e_0}$ be the point represented by $D_0$. Its image under  the projectivity $g =[e_0; e_1; e_2; e_3; e_0]$  is the point of $T_{e_0}$ represented by $D_3$. Since the intersection $D_0 \cap D_3$ is non-empty, it contains a geodesic ray tending to $e_0$. Thus $g$ fixes $x$ as claimed.
\end{proof}

\begin{lemma}\label{lem:TechnicalParallelLines}
Let $e_0, e_1, e_2$ be vertices of $\Delta$, and $D_0, D_1 \subset X$ be geodesic lines with $D_0 \cap D_1 \neq \varnothing$, such that  the two endpoints of the line $D_i$ are $e_i$ and $e_{i+1}$ for   $i=0,1$.

Let $x \in T_{e_0}$ be the point represented by $D_0$. Suppose that $x$ is a vertex of $T_{e_0}$ and let $x' \in T_{e_0}$ be a vertex adjacent to $x$. Let $D'_0$ be the geodesic line parallel to $D_0$ and representing $x'$. Let also $D'_1$ be the geodesic line parallel to $D_1$ and representing the same point of $T_{e_1}$ as $D'_{0}$.

Let $c$ be a chamber of $X$ with vertices $u, v, w$. If $u$ and $v$ are both contained in $D_0 \cap D_1$ and if $w$ is contained in $D'_0$, then $w \in D'_1$.
\end{lemma}

\begin{proof}
We refer to Figure~\ref{fig:FlatStrip}.
\begin{figure}[h]
\begin{center}
\includegraphics[width=\textwidth]{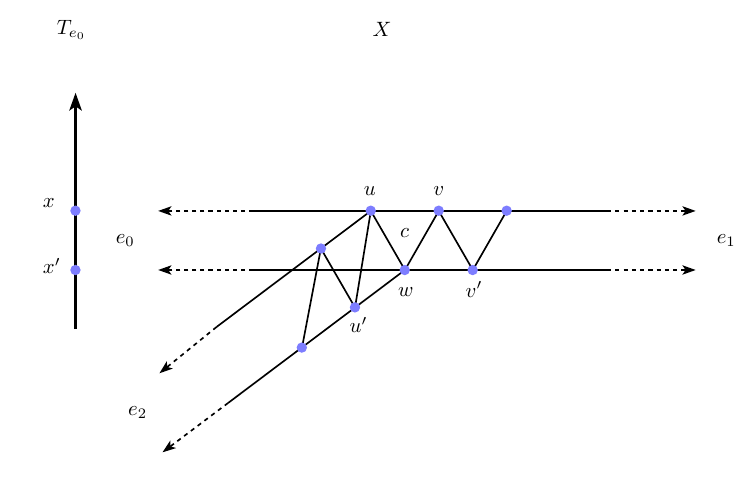}
\caption{Illustration of Lemma~\ref{lem:TechnicalParallelLines}} \label{fig:FlatStrip}
\end{center}
\end{figure}
Upon exchanging $u$ and $v$, we can assume that $u$ belongs to the geodesic ray $[v, e_0)$. Therefore $u$ also belongs to the geodesic ray $[v, e_2)$ since $u, v \in D_1$.

Let $u'$ (resp. $v'$) be the first vertex different from $w$ on the geodesic ray $[w, e_2)$ (resp. $[w, e_1)$). We claim that $u'$ is opposite $v'$ in the residue $\Res(w)$ of the vertex $w$.

To establish the claim, we first observe that $u'$ is adjacent to $u$. Indeed, by hypothesis the first vertex different from $v$ on the geodesic segment $[v, e_2)$ is $u$. Since the line $D_1$ is singular, it is a wall in any apartment that contains it. Since the chamber $c \in \Ch(X)$ has a panel on $D_1$, we may find an apartment containing both the line $D_1$ and the chamber $c$. In particular, that apartment contains the geodesic rays $[v, e_2)$ and $[w, e_2)$, which are parallel. It follows that $u$ is adjacent to $u'$ as desired. Similarly $v$ is adjacent to $v'$. 

This observation implies that the vertices $u'$ and $v$ have the same type. Hence $u'$ and $v'$ have different types. Therefore, if $u'$ and $v'$ are not opposite in $\Res(w)$, then they are adjacent. But the unique vertex in $\Res(w)$ which is adjacent to both $u$ and $v'$ is $v$. Thus, if the claim fails, then  we have $u'=v$. This implies that the first vertex different from $w$ on $[w, e_2)$ is $v$. Therefore the first vertex different from $v$ on $[v, e_2)$ is opposite $w$ in $\Res(v)$. On the other hand, as seen above, the first vertex different from $v$ on $[v, e_2)$ is $u$. Thus $u$ is opposite $w$ in $\Res(v)$. This is absurd since $u$ is adjacent to $w$. The claim is established.

The claim implies that the union of the geodesic rays $(e_1, w] \cup [w, e_2)$ is a geodesic line intersecting $D'_0$ in the  geodesic ray $[w, e_1)$. Therefore that line must coincide with $D'_1$, which confirms that $w \in D'_1$.
\end{proof}

\begin{lemma}\label{lem:AdjacentPoints}
Let $e_0, e_1, e_2, e_3$ be vertices of $\Delta$, and $D_0, D_1, D_2, D_3 \subset X$ be geodesic lines such that  two endpoints of the line $D_i$ are $e_i$ and $e_{i+1}$ for all $i \mod 4$.

Suppose that the point $x \in T_{e_0}$  represented by $D_0$ is a vertex and let $x' \in T_{e_0}$ be a vertex adjacent to $x$. Let $D'_0$ be the geodesic line parallel to $D_0$ and representing $x'$. For $i>0$, define inductively $D'_i$ as the geodesic line parallel to $D_i$ and representing the same point of $T_{e_i}$ as $D'_{i-1}$.

For each $n>0$, if $D_0 \cap D_1 \cap D_2 \cap D_3$ is a geodesic segment of length~$\geq n$, then $D'_0 \cap D'_1 \cap D'_2 \cap D'_3$ is a geodesic segment of length~$\geq n-1$.
\end{lemma}

\begin{proof}
Since the line $D_0$ represents a vertex of $T_{e_0}$, it lies on the boundary of a half-apartment, and thus entirely consists of vertices and edges of $X$. Let $(v_0, v_1, \dots, v_m)$ be the geodesic segment (where $v_i$ is a vertex) which is the intersection $D_0 \cap D_1 \cap D_2 \cap D_3$, so that $m \geq n$ by hypothesis.

For all $i \in \{0, 1, \dots,  m-1\}$, let $v'_i$ be the unique vertex belonging to $D'_0$ which is adjacent to both $v_i$ and $v_{i+1}$. By invoking Lemma~\ref{lem:TechnicalParallelLines} three times successively, we see that $v'_i$ is contained in $D'_1$, $D'_2$ and $D'_3$.
This proves that the segment $[v'_0, v'_{m-1}]$, which is of length $m-1 \geq n-1$, is contained in $D'_0 \cap D'_1 \cap D'_2 \cap D'_3$.
\end{proof}

\begin{lemma}\label{lem:FixedBall}
Let $e_0, e_1, e_2, e_3 \in \Delta$ be vertices, and $D_0, D_1, D_2, D_3 \subset X$ be geodesic lines such that  two endpoints of the line $D_i$ are $e_i$ and $e_{i+1}$ for all $i \mod 4$. Suppose that the point $x \in T_{e_0}$  represented by $D_0$ is a vertex.

If $D_0 \cap D_1 \cap D_2 \cap D_3$ is a geodesic segment of length~$\geq n$, then the projectivity $[e_0; e_1; e_2; e_3; e_0]$ fixes pointwise the ball of radius $n$ around   $x$.
\end{lemma}

\begin{proof}
Let $x'$ be a vertex at distance~$r \leq n$ from $x$. Let $D'_0$ be the geodesic line parallel to $D_0$ and representing $x'$. For $i>0$, define inductively $D'_i$ as the geodesic line parallel to $D_i$ and representing the same point of $T_{e_i}$ as $D'_{i-1}$.

An immediate induction on $r = d(x, x')$, using Lemma~\ref{lem:AdjacentPoints}, shows that $D'_0 \cap D'_1 \cap D'_2 \cap D'_3$ is a geodesic segment of length~$\geq n-r$. By Lemma~\ref{lem:FixedPoint}, this implies that $x'$ is fixed by $[e_0; e_1; e_2; e_3; e_0]$.
\end{proof}

\begin{lemma}\label{lem:horocyclic}
Let $v \in X$ be a vertex. Let $Q, Q_0, Q_1, Q_2, Q_3$ be sectors based at $v$, such that $A_{i, j} = Q_i \cup Q \cup Q_j$ is a half-apartment  of $X$ for all $(i, j) \in \{0, 2\} \times \{1, 3\}$. For all $i \mod 4$, let $D_i$ be the boundary line of $A_{i, i+1}$, and let $e_i$ be the common endpoint of $D_{i-1}$ and $D_i$.

Then the projectivity $[e_0; e_1; e_2; e_3; e_0]$ is a horocyclic automorphism of the tree $T_{e_0}$.
\end{lemma}

\begin{proof}
We refer to Figure~\ref{fig}.
\begin{figure}[h]
\begin{center}
\includegraphics[width=\textwidth]{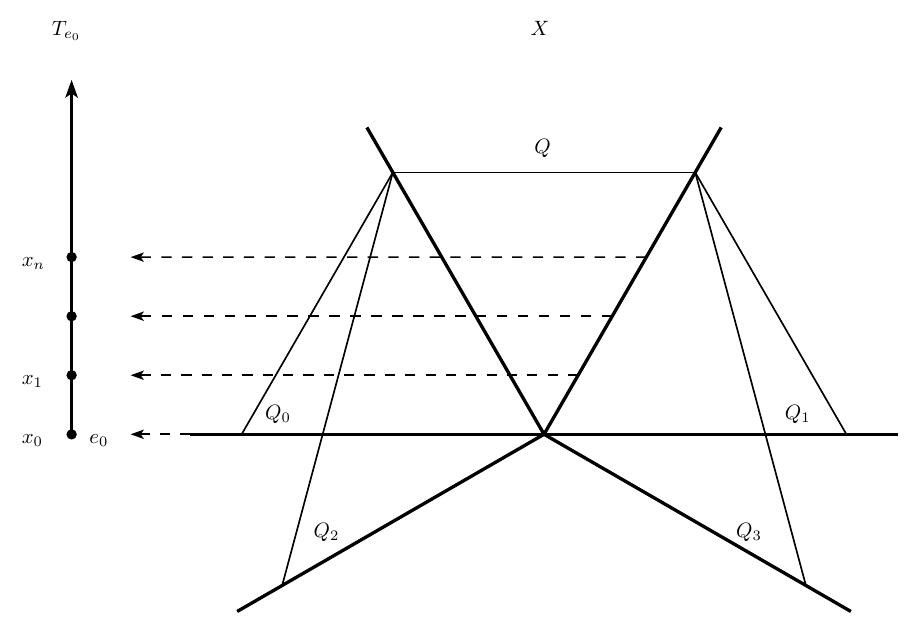}
\caption{Illustration of Lemma~\ref{lem:horocyclic}} \label{fig}
\end{center}
\end{figure}

Let $x_0 \in T_{e_0}$ be the vertex represented by $D_0$, and let $(x_0, x_1, \dots)$ be the geodesic ray of $T_{e_0}$ represented by the sector $Q_0$.

We claim that the projectivity $u = [e_0; e_1; e_2; e_3; e_0]$ fixes pointwise the ball of radius $n$ around $x_n$, so that $u$ is indeed horocyclic.

To see this, let $D'_0$ be the geodesic line parallel to $D_0$ and representing $x_n$. For $i>0$, define inductively $D'_i$ as the geodesic line parallel to $D_i$ and representing the same point of $T_{e_i}$ as $D'_{i-1}$.

By construction, the line $D'_i$ is contained in the half-apartment $A_{i, i+1}$. Therefore, it intersects the sector $Q$ in a segment of length $n$. In particular the intersection $D'_0 \cap D'_1 \cap D'_2 \cap D'_3$ is a geodesic segment of length~$\geq n$. The claim then follows from Lemma~\ref{lem:FixedBall}.
\end{proof}

\begin{proof}[Proof of Proposition~\ref{prop:Unipotent}]
Let $v \in X$ be a vertex. Let $Q, Q_0, Q_1, Q_2, Q_3$ be sectors based at $v$, such that $A_{i, j} = Q_i \cup Q \cup Q_j$ is a half-apartment  of $X$ for all $(i, j) \in \{0, 2\} \times \{1, 3\}$. Since $X$ is thick, we may choose those sectors such that $Q_i \cap Q_{i+2}$ is a geodesic ray, for all $i =0, 1$. In particular, denoting the boundary chamber of $Q_i$ by $C_i \in \Ch(\Delta)$, we have $C_0 \neq C_2$ and $C_1 \neq C_3$.

For all $i \mod 4$, let $D_i$ be the boundary line of $A_{i, i+1}$, and let $e_i$ be the common endpoint of $D_{i-1}$ and $D_i$.

By Lemma~\ref{lem:horocyclic}, the projectivity $u = [e_0; e_1; e_2; e_3; e_0]$ is a horocyclic automorphism of the tree $T_{e_0}$. By Lemma~\ref{lem:NonTriv}, the projectivity $u$ is a non-trivial permutation of $\partial T_{e_0}$, so that $u$ is indeed a non-trivial automorphism of $T_{e_0}$.
\end{proof}

\subsection{$\widetilde A_2$-buildings with a linear projectivity group are Bruhat--Tits} \label{subsec:projgrplinear}

\begin{theorem}\label{thm:charBT}
Let $X$ be a thick, locally finite, Euclidean building of type $\widetilde A_2$. Let  $\Pi$  be its projectivity group, viewed as a subgroup of $\Aut(T)$, where $T$ is the model tree of $X$. Let moreover  $P = \overline{\Pi} \leq \Aut(T)$ denote the closure of $\Pi$ in the group $\Aut(T)$ endowed with the compact-open topology.

If $P$ has an open subgroup admitting a continuous, faithful, finite-dimensional representation over a local field, then $X$ is a Bruhat--Tits building.
\end{theorem}

By the classification of Bruhat--Tits buildings, one knows that $X$ is then the Euclidean building associated with $\mathrm{PGL}_3(D)$, where $D$ is a finite-dimensional division algebra over a local field (see \cite[Chapter~28]{WeissBook}). Notice moreover that the converse to Theorem~\ref{thm:charBT} holds: indeed, if $X$ is the Euclidean building associated with $\mathrm{PGL}_3(D)$, then its projectivity group is $\mathrm{PGL}_2(D)$ acting on its Bruhat--Tits tree (this follows from Proposition~2.3 from \cite{Knarr}).

\begin{proof}[Proof of Theorem~\ref{thm:charBT}]
By Proposition~\ref{prop:3trans}, the projectivity group $(\Pi, \partial T)$ is $3$-transitive. In particular $P$ is $3$-transitive on $\partial T$. We may therefore invoke Proposition~\ref{prop:aux}, which ensures that $(P, \partial T)$ is a Moufang set.

By Proposition~\ref{prop:Unipotent}, the group $\Pi$ contains a non-trivial element $u$ which is a horocyclic automorphism of $T$. By Proposition~\ref{prop:aux}(iii), that element $u$ belongs to a root group $U_e$ of $P$, for some end $e \in \partial T$. Since $(P, \partial T)$ is Moufang, it follows that $U_e$ is normal in $P_e$. In particular $U_e \cap \Pi_e$ is normal in $\Pi_e$. Since $\Pi$ is $3$-transitive on $\partial T$, it follows that $\Pi_e$ is $2$-transitive on $\partial T \setminus \{e\}$, and hence any non-trivial normal subgroup of $\Pi_e$ is transitive on $\partial T \setminus \{e\}$. We have seen that $U_e \cap \Pi_e$ is a non-trivial normal subgroup of $\Pi_e$, and is thus transitive on $\partial T \setminus \{e\}$. Since the root group $U_e$ is sharply transitive on $\partial T \setminus \{e\}$, we must have $U_e \leq \Pi_e$. It follows that $(\Pi, \partial T)$ is a Moufang set. Therefore the building at infinity $\Delta$ is Moufang by Theorem~\ref{thm:Schleiermacher}. Equivalently  $X$ is Bruhat--Tits.
\end{proof}

\section{Other topological spaces associated with $\widetilde{A}_2$-buildings} \label{sec:topspaces}

Throughout this section, we let $X$ be a locally finite $\widetilde A_2$-building and $\Gamma \leq \Aut(X)$ be a discrete subgroup consisting of type preserving automorphisms which acts cocompactly on $X$. The goal of this section is to define   auxiliary topological spaces associated with the pair $(X, \Gamma)$ and needed for subsequent developments. In particular we define in \S\ref{subsec:cartan-setup} the space $\scrF$ consisting of \emph{marked flats}
and in \S\ref{susec:walltrees} the space $\widetilde\scrS$ consisting of \emph{marked wall trees}.
It turns out that the space $\widetilde\scrS$ is too big for our purposes.
We therefore define the notion of \emph{restricted marked wall trees} and study the space of such objects in \S\ref{subsec:rmwt}.
Unfortunately, this makes our discussion a bit technical.
With the hope to help the reader to overcome those technicalities, we  provide in   \S\ref{subsec:broadpic}  a broad overview of the spaces we consider and the maps connecting them.

\subsection{The space of marked flats} \label{subsec:cartan-setup}

We fix a model flat $\Sigma$ for $X$, a base vertex $0\in \Sigma$ and a sector $Q^+$ based at $0$.
The direction of $Q^+$ is called the \textbf{positive Weyl chamber}. The sector of $\Sigma$ based at $0$ and opposite $Q^+$ is denoted by $Q^-$. We let $A$ be the group of simplicial translations of $\Sigma$ (not necessarily type-preserving). The group
$A$ acts sharply transitively on the vertices of $\Sigma$.
We let $W$ be the group of type-preserving automorphisms of $\Sigma$ fixing $0$. It is  called the \textbf{(spherical) Weyl group} of $X$ and is isomorphic to the symmetric group $S_3$.

\begin{definition} \label{def:flats}
A \textbf{marked flat} in $X$ is a simplicial (not necessarily type-preserving) embedding $\Sigma\to X$.
We denote by $\scrF$ the set of marked flats in $X$ and endow it with the topology of pointwise convergence.
\end{definition}

\begin{remark}
By Theorem~\ref{aptfull} the image of a marked flat in $X$ is indeed a flat in $X$.
\end{remark}

Note that $\scrF$ is locally compact as $X$ is locally finite.
In fact it is easy to see how $\scrF$ is decomposed into a disjoint union of compact open subspaces.
We define a map
$$\scrF\to X^{(0)} :
 \phi \mapsto \phi(0),$$
where as before $X^{(0)}$ denotes the vertex-set of $X$. For each $x\in X^{(0)}$ we let $\scrF_x$ denote the preimage of $x$ under this map.
Clearly the subsets $\scrF_x$ are compact open and $\scrF=\sqcup \scrF_x$.

The group $W\ltimes A$ acts on $\scrF$ by precomposition and $\Gamma$ acts on $\scrF$ by postcomposition.
Thus these two actions commute.
We endow $\scrF/\Gamma$ and $\scrF/A$ with the quotient topologies.
Note that $\scrF/\Gamma$ is a $W\ltimes A$-space and $\scrF/A$ is a $\Gamma\times W$-space.

\begin{lemma}\label{lem:FmodGammaCompact}
The space $\scrF/\Gamma$ is a compact $W\ltimes A$-space.
\end{lemma}

\begin{proof}
The union of the spaces $\scrF_x$ taken over a finite set of representatives of the $\Gamma$-orbits in $X^{(0)}$ is compact, and maps surjectively onto $\scrF/\Gamma$. The latter is thus indeed compact.
\end{proof}

Recall that two chambers in $\Delta$ are called \textbf{opposite} if they are at maximal distance in the graph $\Delta$ (i.e. at distance 3). We denote by $\Delta^{2,\op}\subset \Ch(\Delta) \times \Ch(\Delta)$ the set of ordered pairs of opposite chambers, and endow it with the topology induced by the product topology.
This is a locally compact topology.
We consider the map
$$\theta \colon \scrF\to \D^{2,\op}
:  \phi \mapsto ([\phi(Q^+)],[\phi(Q^-)]) , $$
where $[\phi(Q^+)]$ denotes the equivalence class of the sector $\phi(Q^+)\subset X$ which is a chamber in $\Ch(\Delta)$ and similarly for $[\phi(Q^-)]$ which is an opposite chamber.

\begin{lemma} \label{lem:op=F}
The map $\theta \colon \scrF\to \D^{2,\op}$ induces a $\Gamma\times W$-equivariant homeomorphism between $\scrF/A$ and $\Delta^{2,\op}$.
\end{lemma}

\begin{proof}
We first check that $\theta$ is continuous.
Assume that $\phi_n\in\scrF$ are embeddings of $\Sigma$ converging to $\phi$. Then there exists $x$ such that for $n$ large enough we have $\phi_n(0)=x$, and for even larger $n$, $\phi_n$ is constant on the ball of radius $r$ around $x$. In particular $\phi_n(Q^+)$ and $\phi(Q^+)$ are in $\Omega_x(y)$ for some $y$ at distance $r$ from $x$ (see Section~\ref{sec:TopologyAtInfinity} for the definition of the set $\Omega_x(y)$). Hence $\phi_n(Q^+)$ converges to $\phi(Q^+)$, and similarly for $\phi_n(Q^-)$.

By Lemma \ref{opp-apt}, $\theta$ is surjective and its fibers are exactly the $A$-orbits on $\scrF$ since $A$ is vertex-transitive on $\Sigma$.
By the definition of the quotient topology we thus obtain a continuous bijection $\theta' \colon \scrF/A\to \D^{2,\op}$.
We need to show that $\theta'$ is an open map.
For $x\in X^{(0)}$ we denote by $\scrF'_x$ the image of $\scrF_x$ in $\scrF/A$.
Note that $\{\scrF'_x\}$ is an open cover of $\scrF/A$.
Thus it is enough to show that $\theta'|_{\scrF'_x}$ is open.
By compactness, $\theta'|_{\scrF'_x}$ is a homeomorphism of $\scrF'_x$ onto its image,
thus it is enough to show that $\theta'|_{\scrF'_x}$ has an open image.
Note that $\theta'(\scrF'_x)=\theta(\scrF_x)$  is the union of all $(\Omega_x(y)\times\Omega_x(y'))\cap\Delta^{2,\op}$, where the vertices $y$ and $y'$ have the following properties:   the geodesic segment $[y, y']$ contains $x$, and  $[y, y']$ is not contained in any wall of any apartment in $X$.
Hence this set is indeed open in $\Delta^{2,\op}$.
\end{proof}

\subsection{The space of marked wall trees} \label{susec:walltrees}

A flat in $X$ may be viewed as the union of all geodesic lines that are parallel to a given \emph{regular} line in $X$ (see Section~\ref{subsec:paneltree} for the definition of regular and singular lines). We now develop analogous notions for the case of a \emph{singular} line. Let $\ell$ be a singular geodesic line in $X$, so that the two endpoints of $\ell$ are vertices of $\Delta$. The union of all geodesic lines in $X$ that are parallel to $\ell$ is a closed CAT($0$)-convex subset of $X$. As a CAT($0$) space, it is isometric to the product $T_\ell \times \mathbf R$, where $T_\ell$ is the wall tree associated to $\ell$ (see Proposition~\ref{prop:walltree}). In particular $T_\ell$ is isomorphic to the model   tree $T$ (see Definition~\ref{def:modeltree}). Conversely, using Theorem~\ref{aptfull} it is not hard to see that any isometric embedding of the CAT($0$) space $T\times \RR$ in $X$ arises in that way.

The subspace $T_\ell \times \mathbf R \subset X$ also inherits a simplicial structure from $X$. We shall now describe that simplicial structure abstractly. The \textbf{model wall tree} of $X$ is the simplicial complex denoted by $\Upsilon$ and defined as follows. Fix a base vertex $o \in VT$ in the model tree $T$ of $X$. The vertex set $\Upsilon^{(0)}$ of $\Upsilon$ is the collection of pairs $(v, s)$ where $v$ is a vertex of $T$ and $s$ is a real number that satisfies the following conditions: if the distance $d(o, v)$ is even, then $s \in \mathbf Z$ while if $d(o, v)$ is odd, then $s- \frac 1 2 \in \mathbf Z$. Two distinct vertices $(v, s)$  and $(w, t)$ form an edge of $\Upsilon$ if and only if $d(v, w) \leq 1$ and $|s-t| \leq 1$. The $2$-simplices of $\Upsilon$ are defined so that every triple of distinct vertices that are pairwise adjacent are contained in  a unique $2$-simplex. This defines the model wall tree $\Upsilon$ as a simplicial complex. 

Let us mention that a CAT($0$) metric realization of $\Upsilon$ is obtained by giving each edge length $1$ and by endowing each $2$-simplex with the metric of a Euclidean  equilateral triangle. In that way $\Upsilon$ carries a complete CAT($0$) metric that is isometric to $T \times \mathbf R$, where $T$ is viewed as a metric tree in which edges have length $\frac{\sqrt 3} 2$. That isometry $\Upsilon \to T \times \mathbf R$ is the identity on the vertex set $\Upsilon^{(0)}$, which has been defined as a subset of $VT \times \mathbf R$. Although we shall not need that fact, let us mention moreover that every simplicial embedding $\Upsilon \to X$ is also an isometric embedding with respect to the respective CAT($0$) metric realizations of $\Upsilon$ and $X$.

There are two types of vertices in the spherical building $\Delta$, which we called \emph{points} and \emph{lines} in Definition~\ref{def:ProjectivePlane}. From now on, we adopt the convention to denote these two types by $+$ and $-$ respectively. We let
$\Delta_+$ (resp. $\Delta_-$) be the set of vertices of type $+$ (resp. of type $-$) in  $\Delta$. The natural maps $\Ch(\Delta) \to\Delta_+$ and $\Ch(\Delta) \to \Delta_-$, associating to each chamber its two boundary vertices, are $\Aut(X)$-equivariant, continuous  and surjective (see \cite[Prop.~3.5]{CaMo_coctamen}). We denote by $\Omega_x^+(y)$ and $\Omega_x^-(y)$ the image in $\Delta_+$ and $\Delta_-$ of the basic open sets $\Omega_x(y)$.

\begin{definition} \label{def:walltrees}
A \emph{marked wall tree} in $X$ is a simplicial embedding $\phi \colon \Upsilon \to X$ of the model wall tree $\Upsilon$ to $X$, such that the limit of $\phi(x,s)$ when $s$ tends to $+\infty$ is a vertex of type $+$.
We denote by $\widetilde\scrS$ the set of marked wall trees in $X$ and endow it with the topology of pointwise convergence.
\end{definition}

Note that $\widetilde\scrS$ is locally compact as $X$ is locally finite.
In fact it is easy to see how it decomposes into a disjoint union of compact open subspaces.
As before, let  $o\in VT$ be the base vertex and define a map
\begin{align} \label{eq:tldS}
\widetilde\scrS\to X^{(0)} :   \phi \mapsto \phi(o,0).
\end{align}
For each $x\in X^{(0)}$ we let $\widetilde\scrS_x$ denote the preimage of $x$ under this map.
Clearly the subsets $\widetilde\scrS_x$ are compact open and $\widetilde\scrS=\sqcup \widetilde\scrS_x$.

Let us now describe the full automorphism group $\Aut(\Upsilon)$ of the model wall tree. It is a totally disconnected locally compact group with respect to the topology of pointwise convergence. It embeds continuously as a closed subgroup in $\Aut(T) \times \Isom(\RR)$ via the metric realization $\Upsilon \to T \times \RR$ explained above. In particular there are canonical projections 
$$p_T \colon \Aut(\Upsilon) \to  \Aut(T) \hspace{1cm} \text{and} \hspace{1cm} p_\RR \colon \Aut(\Upsilon) \to  \Isom(\RR)$$
that are both continuous homomorphisms. Moreover $p_T$ is surjective, whereas the image of $p_\RR$ is isomorphic to the infinite dihedral group. However, we emphasize that the continuous surjection  $p_T \colon \Aut(\Upsilon) \to  \Aut(T)$ has no section: indeed, no involution that swaps two adjacent vertices in $T$ can be lifted to an automorphism of $\Upsilon$. Let 
$$\Aut(\Upsilon)^0 = \{ \alpha \in \Aut(\Upsilon) \mid p_\RR(\alpha) \text{ is a translation}\}.$$
Since the translation subgroup of $\Isom(\RR)$ is an open subgroup of index~$2$, we see that $\Aut(\Upsilon)^0 $ is an open subgroup of index~$2$ in $\Aut(\Upsilon)$. We shall also need to consider the subgroup 
$$S = \Ker(p_T) \cap \Aut(\Upsilon)^0 \lhd \Aut(\Upsilon).$$
Observe that $S \cong (\ZZ,+)$.

For each $\alpha \in \Aut(\Upsilon)$ and $\phi \in \widetilde\scrS$, the map $\phi \circ \alpha$ belongs to $\widetilde\scrS$ if and only if $\alpha \in \Aut(\Upsilon)^0 $. Similarly, for each $\gamma \in \Aut(X)$ and $\phi \in \widetilde\scrS$, the map $\gamma \circ \phi$ belongs to $\widetilde\scrS$ if and only if the $\gamma$-action on the spherical building $\Delta$ is type-preserving. 
Recall that $\Gamma<\Aut(X)$ is type-preserving,
thus the $\Gamma$ action on $\Delta$ is also type-preserving.
Thus we get a continuous action of $\Aut(\Upsilon)^0$ on $\widetilde\scrS$ by precompositions, 
and an action of $\Gamma$ on $\widetilde\scrS$ by postcompositions.  Clearly these actions  commute. 
We endow the quotient spaces $\widetilde\scrS/\Gamma$ and $\widetilde\scrS/\Aut(\Upsilon)^0$ with the quotient topologies.
Note that $\widetilde\scrS/\Gamma$ is an $\Aut(\Upsilon)^0$-space and $\widetilde\scrS/\Aut(\Upsilon)^0$ is a $\Gamma$-space.

\begin{lemma}\label{lem:tldcocompact}
The $\Gamma$-action on $\widetilde\scrS$ is proper. Moreover  $\widetilde\scrS/\Gamma$ is a compact $\Aut(\Upsilon)^0$-space.
\end{lemma}

\begin{proof}
The properness is clear from the definitions. The compactness of   $\widetilde\scrS/\Gamma$    follows   from the fact that $\Gamma$ has a finite fundamental domain in $X^{(0)}$ and the compactness of the spaces $\widetilde\scrS_x$.
\end{proof}

Let $\D_\pm^\op$  be the set of pairs of opposite  vertices $(u,v)$ with $u\in\D_+$ and $v\in \D_-$.
This is an open subset of the compact space $\D_+\times \D_-$, hence a locally compact space.
We define a map
\begin{align} \label{eq:scrsdpm}
\sigma \colon \widetilde\scrS\to \D_\pm^\op :
 \phi \mapsto (\phi^+, \phi^-) = \left(\lim_{s\to +\infty} \phi(o,s),\lim_{s\to -\infty} \phi(o,s)\right),
\end{align}
where $\lim_{s\to +\infty} \phi(o,s)$ represents the vertex at infinity of the ray $(\phi(o, s))_{s \geq 0}$  and similarly for $\lim_{s\to -\infty} \phi(o,s)\in \D_-$.

\begin{lemma} \label{lem:S/A=D}
The $\Aut(\Upsilon)^0$-action on $\widetilde\scrS$ is proper and free and 
the map $\sigma \colon \widetilde\scrS\to \D_\pm^\op$ induces a $\Gamma$-equivariant homeomorphism
$$\sigma':\widetilde\scrS/\Aut(\Upsilon)^0\overset{\sim}{\longrightarrow} \D_\pm^\op.$$
\end{lemma}

\begin{proof}
The $\Aut(\Upsilon)^0$-action on $\widetilde\scrS$ is free by construction. To see that it is proper
it is   enough to show that for all $x,y\in X^{(0)}$, the collection
\[ K=\{\alpha\in \Aut(\Upsilon)^0\mid \alpha(\widetilde\scrS_x) \cap \widetilde\scrS_y \neq \emptyset \} \]
is a compact subset of $\Aut(\Upsilon)^0$ and this is indeed the case,
as if $\alpha\in K$ then $d(\alpha(o,0),(o,0))=d(x,y)$.

To prove that $\sigma'$ is bijective, observe that for all $(u, v) \in \D_\pm^\op$, there exists a marked wall tree $\phi \in \widetilde\scrS$ with $\sigma(\phi) = (u, v)$. Moreover, for any other   marked wall tree  $\phi' \in \widetilde\scrS$, we have $\sigma(\phi') = (u, v)$ if and only if there exists $\alpha \in \Aut(\Upsilon)^0$ such that $\phi' = \phi \circ \alpha$. This proves that the map in the Lemma is indeed bijective.  

Let us now check that the map $\sigma$ is continuous. Let $\phi_n$ are elements of $\widetilde \scrS$, with $\phi_n\to \phi$. Let $u_n=\lim_{s\to +\infty} \phi_n(o,s)$ and $u=\lim_{s\to +\infty} \phi(o,s)$. Then for every $r>0$, for $n$ large enough, $\phi_n$ is constant on a ball of radius $r$ around $(o,0)$. If $y$ is a point in this ball, it follows that $u_n\in \Omega_x^+(y)$ for $n$ large enough. Similarly, $u\in \Omega_x^+(y)$, so that $(u_n)$ converges to $u$.

We are left to show that $\sigma$ is an open map.
As in the proof of Lemma~\ref{lem:op=F}, it suffices  by compactness to prove that $\sigma(\widetilde\scrS_x)$ is open, for every $x\in X^{(0)}$. Indeed, $\sigma(\widetilde\scrS_x)$ is a union of all $\Omega_x^+(y)\times \Omega_x^-(y')\cap \D_\pm^\op$, where $y$ and $y'$ are vertices such that $x$ belongs to the geodesic segment $[y, y']$, hence open.
\end{proof}

\begin{definition} \label{def:bard}
We let $\overline\Delta_+$ be the set of pairs $(v,f)$, where $v\in\Delta_+$ and $f$ is a simplicial isomorphism from $T$ to the panel tree $T_v$.
We let $\pi_+$ be the map $\widetilde\scrS\to \overline\Delta_+$ given by $\pi_+(\phi)=(v,f)$ where $v$ is the common limit of $\phi(x,t)$ when $t$ tends to $+\infty$, and $f(x)$ is the class of the geodesic line $\phi(x,\mathbf R)$ in $T_v$. We define similarly $\overline\Delta_-$ and $\pi_-$.
Observe that the maps $\pi_+$ and $\pi_-$ are surjective. We endow $\overline\D_+$ and $\overline \D_-$ with the quotient topologies obtained under $\pi_+$ and $\pi_-$.
\end{definition}

\begin{lemma}
Fix $u\in\D_+$. Let $F_u$ be the set of simplicial isomorphisms from $T$ to $T_u$, endowed with the topology of pointwise convergence. Then the map $F_u\to \overline\D_+$ defined by $f\mapsto (u,f)$ is a homeomorphism on its image.
\end{lemma}

\begin{proof}
Let $f_n\in F_u$ be such that $(f_n)$ converges to $f$. Fix a vertex $v\in\D_-$ opposite   $u$. Let $\phi_n\in\widetilde\scrS$ be such that for every $n$, $\lim_{s\to -\infty}\phi_n(x,s)=v$ and $\pi_+(\phi_n)=(u,f_n)$. Let $x\in T$. Then the map $t \mapsto \phi_n(x,t)$ is a parametrization of a geodesic from $v$ to $u$ which, for $n$ large enough, corresponds to $f(x)$. Upon replacing $\phi_n(x, t)$ by $\phi_n(x, t-t_0)$ for some suitable $t_0 \in \RR$, we can assume that the $\phi_n$ are chosen so that $\phi_n(x,0)$ is eventually constant. Then $\phi_n$ converges to some $\phi\in\widetilde\scrS$ with $\pi_+(\phi)=(u,f)$.

Conversely, assume that $\phi_n$ converges to $\phi$ with $\pi_+(\phi_n)=(u,f_n)$ and $\pi_+(\phi)=(u,f)$. Let $x\in T$. Then we have for $n$ large enough $\phi_n(x,0)=\phi(x,0)$. Since $f_n(x)$ (resp. $f(x)$) is the projection of $\phi_n(x,0)$ (resp. $\phi(x,0)$) to $T_u$, we get immediately that $f_n$ converges to $f$.
\end{proof}

It will be convenient for us to consider various fiber products. If $A,B,C$ are topological spaces, with continuous maps $p_A:A\to C$ and $p_B:B\to C$, the \emph{fiber product} $A\times_C B$ is defined as the set $\{(a,b)\in A\times B\mid p_A(a)=p_B(b)\}$. It is equipped with the topology induced from the product topology.

\begin{lemma}\label{ShomeoFiber}
The map $\widetilde\scrS/S\to \D_\pm^\op\times_ {\D_+}\overline\D_+$ which associates to $\phi$ the pair $((\phi^+,\phi^-),\pi_+(\phi))$ is a homeomorphism, where $\widetilde\scrS/S$ is endowed with the quotient topology.

Similarly, the map  $\widetilde\scrS/S\to \D_\pm^\op\times_ {\D_-}\overline\D_-$ which associates to $\phi$ the pair $((\phi^+,\phi^-),\pi_-(\phi))$ is a homeomorphism.

\end{lemma}

\begin{proof}
We argue for  $\D_\pm^\op\times_ {\D_+}\overline\D_+$ only, as the other case is similar. 
The map of the Lemma is bijective: every pair of opposite vertices is linked by geodesics forming a tree, and given a pair of opposite vertices, there is an element of $\widetilde\scrS$ realizing a given embedding of the tree, and that element is unique modulo $S$. Since both maps $\widetilde\scrS/S\to \D_\pm^\op$ and $\widetilde\scrS/S\to \overline \D_+$ are continuous, their product also is continuous.

We now have to check that this map is open. Again, by compactness, it suffices to prove that the image of $\widetilde\scrS_x$ in $\D_\pm^\op\times_{\D_+}\overline \D_+$ is an open set. But this set is exactly the set of pairs whose first coordinate is in $\sigma(\widetilde\scrS_x)$. Hence it is open.
%
%
\end{proof}

\begin{lemma}\label{FiberContinuous}
The map $\D_\pm^\op\times_{\D_+}\overline \D_+\to \overline \D_-$, which associates to an element $((u,v),f)$ with $f \colon T\to T_v$  the pair $(v,f')$ where $f' \colon T\to T_u$ is the map obtained by composing $f$ with the identification $T_v\to T_u$, is continuous.
\end{lemma}

\begin{proof}
The space $\D_\pm^\op\times_{\D_+}\overline \D_+$ is homeomorphic to $\widetilde\scrS/S$ by Lemma \ref{ShomeoFiber}. The map of the Lemma is transported by this homeomorphism to the  map $\pi_-:\widetilde\scrS/S\to \overline \D_-$, which is continuous.
\end{proof}

\begin{proposition}\label{prop:persp continuous}
Let $n \geq 2$ be an integer and  let $V_n(u,v)$ be the set of sequences 
$$(u_1,u_2,\dots,u_n)$$ 
of vertices which are consecutively opposite, and such that $u_1$ is opposite $u$ and $u_n$ is opposite $v$. The set $V_n(u,v)$ is equipped with the topology induced from the product topology. 
Let also $M(u, v)$ be the set of simplicial isomorphisms of $T_u$ to $T_v$, with the topology of pointwise convergence.

Then the map
$$ V_n(u,v) \to M(u, v) : (u_1,u_2,\dots,u_n) \mapsto [u;u_1;u_2;\dots;u_n;v]$$
 is continuous.
\end{proposition}

\begin{proof}
We prove the proposition in the special case where $u$ and $v$ are of type $+$ (so that $n$ is odd). The other cases are similar.  Let $V_n$ be the set of sequence of $n$ consecutively opposite vertices, starting from a vertex of type $+$. Note that $V_n$ is naturally identified with  the fiber product $\D_\pm^\op\times_{\D_-}\D_\pm^\op\times_{\D_+}\times\dots\times_{\D_-}\D_\pm^\op$ (with $n$ factors).

Let us form the fiber product $V_n\times_{\D_+}\overline\D_+$ associated with the map $V_n \to \D_+ : (u_1,\dots,u_n) \mapsto u_n$. We then consider the map $P_n \colon V_n\times_{\D_+}\overline\D_+\to \overline \D_+$  which associates to $\big((u_1,\dots,u_n), (u_n,f) \big)$  the map $f':T\to T_{u_1}$ obtained by composing the map $f$ with the perspectivity $[u_n;u_{n-1};\dots;u_1]$.
Applying Lemma \ref{FiberContinuous} (and its symmetric statement with $+$ and $-$ exchanged), and using the associativity of the fiber product, we see that the map $P_n$ is continuous.


For every isomorphism $f \colon T\to T_u$, we have: 
$$P_{n+2}\big((v,u_n,\dots,u_1, u),(u,f)\big)= (v, [u, u_1, \dots, u_n, v] \circ f).$$ 
Let us now choose  an isomorphism $f \colon T\to T_u$ that we keep fixed.  Given $(u_1,u_2,\dots,u_n) \in V_n(u,v) $, let us denote by $m \in M(u, v)$ the perspectivity $[u;u_1;u_2;\dots;u_n;v]$. The map $(u_1,u_2,\dots,u_n) \mapsto m$  can be factored as the composite of the maps
$$(u_1,u_2,\dots,u_n)  \mapsto \big((v,u_n,\dots,u_1, u),(u,f)\big) \underset{P_{n+2}}{\longmapsto} (v, m \circ f) \mapsto m \circ f \mapsto m.$$
 Since each individual map in that sequence is continuous,  the Proposition follows.
\end{proof}

\subsection{The space of restricted marked wall trees} \label{subsec:rmwt}

In the next section we will construct measures on the various spaces we consider here.
There is a natural measure on $\widetilde\scrS$ that we could construct.
However, we will not do it, as this measure fails to satisfy the ergodic properties that we desire.
Instead we will consider a certain subspace of $\widetilde\scrS$, to be denoted $\scrS$, on which the analogous measure will be better behaved.
This subsection is devoted to the definition and the study of this subspace and its properties.

For the next definition, recall from Section~\ref{susec:walltrees} that $\Aut(\Upsilon)$ acts on $\widetilde\scrS$ by precomposition and $\Gamma$ acts on $\widetilde \scrS$ by postcomposition. Moreover $S$ is a cyclic normal subgroup of $\Aut(\Upsilon)$  acting as integer translations along the line factor of $\Upsilon$.

\begin{definition}\label{def:equivalent}
Two embeddings $\phi,\psi\in \widetilde \scrS$ are \textbf{positively equivalent} if for every $x\in VT$, there exists $t_x \in \RR$ such that   $\phi(x, t)=\psi(x, t)$ for all $t > t_x$. 
The embeddings $\phi$ and $\psi$ are \textbf{negatively equivalent} if for every $x\in VT$, there exists $t_x \in \RR$ such that   $\phi(x, -t)=\psi(x, -t)$ for all $t > t_x$. 
We denote the positive and negative equivalences by $\sim_+$ and $\sim_-$ respectively. 
We denote  by $\sim$ the equivalence relation on $\widetilde\scrS$ generated by $\sim_+$, $\sim_-$ and by  the orbit equivalence induced by the action of $S\times \Gamma$.
\end{definition}


\begin{lemma}\label{lem:Aut(T)transitive}
The equivalence relation $\sim$ is $(\Aut(\Upsilon)^0\times \Gamma)$-invariant.
Moreover, every equivalence class is $(S\times \Gamma)$-invariant, and   $\Aut(\Upsilon)^0$ acts transitively on $\widetilde\scrS/\sim$.
\end{lemma}

\begin{proof}
By construction, every individual $\sim$-class is $(S\times\Gamma)$-invariant.
Clearly $\sim_+$ and $\sim_-$ are invariant under  $\Aut(\Upsilon)^0$. Moreover since  $S\times \Gamma$ is a normal subgroup of $\Aut(\Upsilon)^0 \times \Gamma$, the  $(S\times \Gamma)$-orbits are permuted by $\Aut(\Upsilon)^0\times \Gamma$.   It follows that $\sim$ is also $(\Aut(\Upsilon)^0\times \Gamma)$-invariant. In particular $\Aut(\Upsilon)^0$ acts on $\widetilde\scrS/\sim$. 

It remains to show that $\Aut(\Upsilon)^0$ is transitive on $\widetilde\scrS/\sim$. Let   $\phi_1, \phi_2 \in \widetilde \scrS$. We must find $\alpha \in \Aut(\Upsilon)^0$ such that $\phi_1 \circ \alpha \sim \phi_2$. Set $(\phi_i^+, \phi_i^-) = (u_i, v_i) \in \D_\pm^\op$ for $i=1,2$. We choose $u\in \D_-$ which is opposite   both $v_1,v_2$, and $v\in \D_+$ which is opposite   both $u,u_2$ (see Lemma~\ref{lem:twoop}). We next choose successively $\psi_1, \psi_2, \psi_3, \psi_4 \in \widetilde \scrS$ such that:
\begin{itemize}
\item  $\psi_1^+ = u$ and $\psi_1 \sim_- \phi_1$;
\item $\psi_2^- = v$ and $\psi_2 \sim_+ \psi_1$;
\item $\psi_3^+ =  u_2$ and $\psi_3 \sim_- \psi_2$;
\item  $\psi_4^- = v_2$ and $\psi_4 \sim_+ \psi_3$.

\end{itemize}
 Thus we have $\phi_1 \sim_- \psi_1 \sim_+ \psi_2 \sim_- \psi_3 \sim_+ \psi_4$, so that $\phi_1 \sim \psi_4$. Moreover $(\psi_4^+, \psi_4^-) = (u_2, v_2) = (\phi_2^+, \phi_2^-)$. Since $\Aut(\Upsilon)^0$ is transitive on each fiber of the map $ \widetilde \scrS \to \D_\pm^\op$ by Lemma~\ref{lem:S/A=D}, we infer that there exists $\alpha \in \Aut(\Upsilon)^0$ with $\psi_4\circ \alpha = \phi_2$. The conclusion follows since $\phi_1 \sim \psi_4$. 
\end{proof}

Fix $\gamma\in\Gamma$ and $u\in\Delta_-$. The map  $\gamma$ induces a map $T_u\to T_{\gamma u}$.
Let $(\gamma u,v_1,\dots,v_n)$ be a sequence of consecutively opposite vertices.
We denote by $[\gamma:u;v_1;\dots;v_n]$ the map $T_u\to T_{v_n}$ obtained by first applying $\gamma$ and then the perspectivity $[\gamma u;v_1;\dots;v_n]$.

\begin{lemma} \label{lem:pers-representatives}
Let $\phi,\phi'\in\widetilde\scrS$  and set $\pi_-(\phi)=(u_-,f)$, $\pi_+(\phi) = (u_+, f_+)$ and $\pi_-(\phi')=(u'_-,f')$.
Then the following assertions are equivalent. 
\begin{enumerate}[(i)]
\item  $\phi\sim\phi'$.
\item There exist an element $\gamma\in\Gamma$, an  integer $n$, and vertices $v_0, v_1,\ldots,v_n$ in $\Delta$ such that $(\gamma u_-, \gamma u_+ ) = (v_0, v_1)$, the pairs $(v_1,v_2)$, \dots, $(v_{n-1},v_n)$, $(v_n,u'_-)$ are all opposite, 
and $f'=[\gamma : u_-;v_1;\dots;v_n; u'_-]\circ f$.

\end{enumerate}
 \end{lemma}

\begin{proof}
Assume that  $\phi\sim\phi'$. 
By   definition, this means that there exists a sequence $(\phi=\phi_0,\phi_1,\dots,\phi_{n-1},\phi_n=\phi')$ such that $\phi_i$ and $\phi_{i+1}$ are   equivalent for $\sim_+$ or $\sim_-$ or in the same $(S\times \Gamma)$-orbit.
Since $\sim_-$ and $\sim_+$ are preserved by $S\times \Gamma$ (see Lemma~\ref{lem:Aut(T)transitive}) we may assume that $\phi_1=(s,\gamma)\phi$ and for $i>1$, $\phi_i$ and $\phi_{i+1}$ are either $\sim_+$ or $\sim_-$ equivalent.
Since $\sim_+$ and $\sim_-$ are transitive relations we can assume further that they alternate.
Thus (i) implies (ii). 

Conversely, let $\gamma\in\Gamma$ and let $v_0, v_1,\ldots,v_n$ be vertices of $\Delta$ such that $(\gamma u_-, \gamma u_+) = (v_0, v_1)$, and  $(v_1,v_2)$, \dots, $(v_{n-1},v_n)$ , $(v_n,u'_-)$ are all opposite pairs, and such that  $f'=[\gamma : u_-;v_1;\dots;v_n;u'_-]\circ f$. Since $u_-, \gamma u_-$ and $u'_-$ all belong to $\Delta_-$, it follows that $n$ must be odd. Set $\phi_0 =  \gamma \circ \phi$. For all  $i \in \{1, \dots, n-1\}$, choose successively $\phi_i \in \widetilde\scrS$ satisfying  $\big((\phi_i)_+, (\phi_i)_-\big) = (v_i, v_{i+1})$ and  $\phi_i \sim_+ \phi_{i-1}$ if $i$ is odd, while $\phi_i$ must satisfy for $\big((\phi_i)_-, (\phi_i)_+\big) = (v_i, v_{i+1})$ and $\phi_i \sim_- \phi_{i-1}$ if $i$ is even. Finally, choose $\phi_n \in \widetilde\scrS$ with $\{(\phi_n)_{+}, (\phi_n)_{-}\} = (v_n, u_-')$ and $\phi_n \sim_+ \phi_{n-1}$. By construction we have $\phi \sim \phi_0 \sim \dots \sim \phi_n$. It remains to observe that the condition $f'=[\gamma : u_-;v_1;\dots;v_n;u'_-]\circ f$ ensures that $\phi_n \sim_- \phi'$. Thus (ii) implies (i). 
\end{proof}

In the next lemma, we use the projection map $p_T \colon \Aut(\Upsilon) \to \Aut(T)$ from Section~\ref{susec:walltrees}. 

\begin{lemma} \label{lem:M'}
Fix an $\sim$-equivalence class $\scrS' \subset \widetilde\scrS$ and let $M'<\Aut(\Upsilon)^0$ be the stabilizer of $\scrS'$, so that $M'$ contains $S$. 
Fix $\phi\in \scrS'$ and set $\pi_-(\phi)=(u,f)\in \overline\Delta_-$ and $\pi_+(\phi)=(u',f')\in \overline\Delta_+$.
Let  also 
$$M'_\phi=f \big(p_T(M') \big)f^{-1}<\Aut(T_u).$$
Then
\begin{align*}
M'_\phi=\big\{[\gamma : u; & \; v_1;\dots;v_n;u]\mid  \  \gamma \in \Gamma, \ v_0, v_1, \dots, v_n \in \Delta_\pm,  \ (\gamma u, \gamma u') = (v_0, v_1),\\
&  \text{and the pairs } (v_1,v_2),\ldots(v_{n-1},v_n),(v_n,u)\text{ are all opposite}\big\}.
\end{align*} 
In particular, the projectivity group $\Pi(u)$ is contained in $M'_\phi$, and $p_T(M')$ is $3$-transitive on $\partial T$. 
\end{lemma}

\begin{proof}
Fix $m'\in M'$. We have $\phi\circ m'\sim \phi$ by the definition of $M'$. Moreover, recalling that $\pi_-(\phi) = (u, f)$, we have $\pi_-(\phi \circ m') = (u, f \circ \pi_T(m'))$.   By Lemma~\ref{lem:pers-representatives}, there exists an automorphism of $T_u$ of the form $[\gamma:u;v_1;\dots;v_n;u]$ for some $\gamma \in \Gamma$ and some  vertices $v_0, v_1, \dots, v_n $ of $ \Delta$ with $(\gamma u, \gamma u') = (v_0, v_1)$  and   $(v_1,v_2)$, \dots,  $(v_{n-1},v_n)$, $(v_n,u)$     all opposite pairs, such that $f \circ \pi_T(m') =[\gamma : u;v_1;\dots;v_n;u]\circ f$. 

Conversely, let  $\gamma \in \Gamma$ and  let  $v_0, v_1, \dots, v_n $ be vertices of $ \Delta$ such that $(\gamma u, \gamma u') = (v_0, v_1)$, the pairs   $(v_1,v_2)$, \dots,  $(v_{n-1},v_n)$, $(v_n,u)$ are    all opposite pairs, and define the automorphism  $\alpha = [\gamma : u; \; v_1;\dots;v_n;u] \in \Aut(T_u)$. Since $p_T \colon \Aut(\Upsilon)^0 \to \Aut(T)$ is surjective, there exists $m' \in \Aut(\Upsilon)^0 $ with $p_T(m')  = f^{-1} \circ \alpha \circ f$. We have  $\phi\circ m'\sim \phi$  by Lemma~\ref{lem:pers-representatives}.

That $M'_\phi$ is $3$-transitive on $\partial T_u$ follows from Lemma~\ref{lem:modeltree}. Therefore $p_T(M')$ is $3$-transitive on $\partial T$. 
\end{proof}

We wish to consider a single equivalence class in $\widetilde\scrS$.
A technical issue is that such an equivalence class will not be closed.
The obvious solution is to replace the equivalence relation $\sim$ by its closure (as a subset of $\widetilde\scrS\times\widetilde\scrS$).
Note that in general the closure of an equivalence relation need not be an equivalence relation.
However the situation is better when a topological group is acting continuously on the ambient space and transitively on the quotient space.

\begin{definition}
We let $\simeq$ be the closure of $\sim$ as a subset of $\widetilde\scrS\times\widetilde\scrS$.
\end{definition}

\begin{lemma} \label{lem:simeq}
The relation $\simeq$ is an equivalence relation and each $\simeq$-equivalence class is closed in $\widetilde\scrS$.
In fact, the $\simeq$-equivalence classes are the closures of the $\sim$-equivalence classes.
In particular, every $\simeq$-equivalence class is $S\times\Gamma$-invariant. The quotient space $\widetilde\scrS/\simeq$ is compact, and the natural $\Aut(\Upsilon)^0$-action on $\widetilde\scrS/\simeq$ is continuous. 
%
\end{lemma}

\begin{proof}
Recall from Lemma~\ref{lem:Aut(T)transitive} that $\Aut(\Upsilon)^0$ acts transitively on $\widetilde\scrS/\sim$. Therefore, the orbit map induces an equivariant bijection  $\widetilde\scrS/\!\sim  \ \to \Aut(\Upsilon)^0/M'$, where $M'$ is the subgroup of $\Aut(\Upsilon)^0$ stabilizing some $\sim$-equivalence class. Let   $M$ be the closure of $M'$ in $\Aut(\Upsilon)^0$. We have canonical maps
$$\widetilde \scrS \to \widetilde \scrS /\! \sim \ \to \Aut(\Upsilon)^0/M' \to \Aut(\Upsilon)^0/M.$$
We may now identify $\simeq$ with the fiber equivalence relation of the   map $\widetilde\scrS\to \Aut(\Upsilon)^0/M$. Since $S \leq M$ and $p_T(\Aut(\Upsilon)^0) = \Aut(T)$, it follows that $p_T(M) \cong M/S$ is a closed subgroup of $\Aut(T)$, and that $\Aut(\Upsilon)^0/M$ is homeomorphic to $p_T(\Aut(\Upsilon)^0)/p_T(M) = \Aut(T)/p_T(M)$. Since $p_T(M')$ is $3$-transitive on $\partial T$ by Lemma~\ref{lem:M'}, the same holds for $M$, and it follows from Theorem~\ref{thm:BuMo} that $p_T(\Aut(\Upsilon)^0)/p_T(M)$ is compact. All the required assertions now follow.
\end{proof}

The following crucial definition involves a non-canonical choice.

\begin{definition} \label{def:S}
We fix once and for all an $\simeq$-equivalence class in $\widetilde\scrS$ and denote it by $\scrS$.
We call the elements of $\scrS$ the \textbf{restricted marked wall trees}.
We let $M<\Aut(\Upsilon)^0$ be the stabilizer of $\scrS$. In particular we have $S \lhd M$. 
\end{definition}

\begin{lemma} \label{lem:scrs}
The space $\scrS$ is closed in $\widetilde\scrS$ and invariant under the action of $M\times \Gamma$.
The $\Gamma$-action on $\scrS$ is proper
and $\scrS/\Gamma$ is a compact $M$-space.
The $M$-action on $\scrS$ is proper and free, and 
the restriction of the map $\widetilde\scrS\to \D_\pm^\op$ to $\scrS$ yields a $\Gamma$-equivariant homeomorphisms $\scrS/M \overset{\sim}{\longrightarrow}\D_\pm^\op$ and $\scrS/S \big/ M/S \overset{\sim}{\longrightarrow}\D_\pm^\op$.
\end{lemma}

\begin{proof}
The first sentence holds by the construction of $\scrS$.
The second sentence follows by Lemma~\ref{lem:scrs} and Lemma~\ref{lem:tldcocompact}.
To see the third, use Lemma~\ref{lem:S/A=D} and observe that 
the embedding $\scrS \to \widetilde \scrS$ descends to a continuous   map $\scrS/M \to \widetilde \scrS/\Aut(\Upsilon)^0$ which is injective by the definition of $M$. Moreover, this map is surjective in view of \ref{lem:Aut(T)transitive}. Therefore we have a $\Gamma$-equivariant homeomorphism  $\scrS/M  \to \widetilde \scrS/\Aut(\Upsilon)^0 $. The last required assertions now follow from Lemma~\ref{lem:S/A=D} and the fact that $S$ is a closed normal subgroup of $M$.
\end{proof}

\begin{theorem}\label{thm:S}
The group $M$ is a unimodular group and $p_T \colon M \to \Aut(T)$ yields a continuous isomorphism of $M/S$ onto a closed subgroup of $\Aut(T)$ acting $3$-transitively on the set of ends $\partial T$. In particular $M$ has a closed cocompact normal subgroup $M^+$ containing $S$ such that $M^+/S$ is 
compactly generated,  topologically simple and non-discrete. 
Furthermore, if $X$ is not Bruhat--Tits, then $M^+/S$ has no non-trivial continuous linear representation over any local field.
\end{theorem}

\begin{proof}
By Lemma~\ref{lem:simeq}, the group $M$ is the closure of the group $M'$ and, by Lemma~\ref{lem:M'}, the group $p_T(M)$ contains a copy $\Pi$ of the projectivity group $\Pi(u)$ for some $u\in \Delta_-$. 
It follows by Lemma~\ref{lem:modeltree} that $M$ acts $3$-transitively on $\partial T$. As observed before, we have $S \leq M \leq \Aut(\Upsilon)^0$ and $p_T(\Aut(\Upsilon)^0) = \Aut(T)$, so that $p_T(M)$ is a closed subgroup of $\Aut(T)$ acting $3$-transitively on $\partial T$. 
The group $p_T(M)^+$ is described in Theorem \ref{thm:BuMo} where its various properties are recorded, as well as the unimodularity of $p_T(M)$. Since $S$ is discrete, we see that $M$ is unimodular. We set $M^+ = p_T^{-1}(p_T(M)^+)$. 
We are left to show that the non-linearity of $M^+/S \cong p_T(M)^+$ in case $X$ is not Bruhat--Tits.
To this end we consider the topologically simple group $\overline{\Pi}^+$ which is obtained by applying Theorem \ref{thm:BuMo} to the closure $\overline{\Pi}$ of the projectivity group $\Pi$. 
As $M/M^+$ is profinite and $\overline{\Pi}^+$ topologically simple we have $\overline{\Pi}^+<p_T(M)^+$ since $\Pi < p_T(M)$.
By Theorem~\ref{thm:charBT}, the group $\overline{\Pi}^+$ has no linear representations over local fields.
We deduce that $p_T(M)^+ \cong M^+/S $ shares this property, as it is topologically simple.
\end{proof}

\begin{definition}
We call the $M\times S$-space $\scrS/\Gamma$ the \textbf{singular Cartan flow}.
\end{definition}

\subsection{The broad picture} \label{subsec:broadpic}

In this subsection we pack together the various $\Gamma$-spaces constructed previously in the following commutative diagram of $\Gamma$-equivariant maps:

\begin{align} \label{eq:broadpic}
\begin{split}
\xymatrix{
X^{(0)}/\Gamma & \scrF/\Gamma \ar[l] & \scrS/\Gamma \ar[l] & & \\
X^{(0)} \ar[u] & \scrF \ar[dd] \ar[l] \ar[u] &  \scrS \ar[l] \ar[d] \ar[u] & & \\
& & \scrS/S \ar[d]\ar[dl]  \ar[r] \ar@/^2pc/[rr] & \overline\Delta_+ \ar[d] & \overline\Delta_- \ar[d] \\
\Ch(\D)  & \Delta^{2,\op} \ar[r] \ar[l]_{p_1,p_2} & \Delta_\pm^\op \ar[r] \ar@/^2pc/[rr] & \D_+ & \D_-
}
\end{split}
\end{align}
In order to keep our notation as light as possible we will not name most of the arrows in Diagram~(\ref{eq:broadpic}),\
but in what follows we will make clear what they are.

In the beginning of \S\ref{subsec:cartan-setup} we have fixed the model apartment $\Sigma$ and a positive Weyl chamber $Q^+$ emanating from $0\in\Sigma$.
We defined $\scrF\to X^{(0)}$ by
\[ \scrF=\{\text{simplicial embeddings}~\Sigma\to X\} \ni \phi \mapsto \phi(0) \in X^{(0)}.\]
We further defined the map $\scrF\to \D^{2,\op}$ by
\[ \scrF=\{\text{simplicial embeddings}~\Sigma\to X\} \ni \phi \mapsto ([\phi(Q^+)],[\phi(Q^-)]) \in \D^{2,\op},\]
where $[\phi(Q^+)]$ denotes the equivalence class of the section $\phi(Q^+)\subset X$ which is a point in $\Delta$, similarly for $[\phi(Q^-)]$ which is an opposite point.

In the beginning of \S\ref{susec:walltrees} we fixed the model marked wall tree  $\Upsilon$  and a base point $o\in VT$.
We now relate $\Sigma$ and $\Upsilon$. 
We denote by $\ell$ the unique wall through $0$ in $\Sigma$ such that $\ell\cap \overline{Q^+}=\{0\}$
and fix a simplicial embedding $\iota:\Sigma \hookrightarrow \Upsilon$ such that $\iota(\ell)$ is $\{o\}\times \RR \subset \Upsilon \subset T\times \RR$. Recall that $\scrS \subset \widetilde \scrS = \{\text{simplicial embeddings}~\Upsilon\to X\}$.
We define the map
\[\scrS\to \scrF :   \psi \mapsto \psi|_{\iota(\Sigma)}. \]
Note that the composition $\scrS\to \scrF\to X^{(0)}$ agrees with the restriction to $\scrS$ of the map~(\ref{eq:tldS}) given in \S\ref{susec:walltrees}.
The maps $\scrS\to \scrF\to X^{(0)}$ are clearly $\Gamma$-equivariant ($\Gamma$ acts on the target space $X$) and induce corresponding maps $\scrS/\Gamma \to \scrF/\Gamma \to X^{(0)}/\Gamma$. The commutativity of the upper part of Diagram~(\ref{eq:broadpic}) is obvious.

Note that the composed map $\scrS\to\scrF\to\D^{2,\op}$ is $S$-invariant and  thus factors through $\scrS/S$.
This explains the map $\scrS/S\to \D^{2,\op}$ appearing in Diagram~(\ref{eq:broadpic}).
We define the map $\D^{2,\op}\to \D_\pm^\op$ by associating with a pair of chambers $(C_1,C_2)\in \D^{2,\op}$ the unique pair of elements $\delta_+\in\D_+$ and $\delta_-\in \D_-$ such that $\delta_+$ and $\delta_-$ are in the boundary of the unique flat that contains both $C_+$ and $C_-$ but disjoint from their closures.
By composing $\scrS/S\to \D^{2,\op}\to \D_\pm^\op$ we also obtain the map $\scrS/S\to\D_\pm^\op$.
Note that the composed map $\scrS\to\scrS/S \to \D_\pm^\op$ is the restriction to $\scrS$ of the map~(\ref{eq:scrsdpm}) defined in \S\ref{susec:walltrees}, namely
\[  \phi \mapsto (\phi^+, \phi^-) = \left(\lim_{s\to +\infty} \phi(o,s),\lim_{s\to -\infty} \phi(o,s)\right) \in \D_\pm^\op. \]
Similarly we can restrict to $\scrS$ the maps $\widetilde\scrS\to  \overline\Delta_+, \overline\Delta_-$ appearing in Definition~\ref{def:bard} and obtain maps $\scrS\to  \overline\Delta_+, \overline\Delta_-$.
The latter, being $S$-invariant, factor through $\scrS/S$ and we obtain the maps $\scrS/S\to  \overline\Delta_+, \overline\Delta_-$ appearing in the diagram.
The maps $\overline\Delta_+\to \D_+$ and $\overline\Delta_-\to \D_-$ are the obvious ones and
the maps $\Ch(\D) \to \D_+$ and $\Ch(\D) \to\D_-$ are the
ones described in the beginning of \S\ref{susec:walltrees}. The commutativity of the diagram is clear from the definitions.
Finally, the maps $p_1,p_2\colon \D^{2,\op}\to \Ch(\D)$ are the first and second factor projections.

\subsection{Equivariance properties of Diagram~(\ref{eq:broadpic})}

Note that all the maps in Diagram~(\ref{eq:broadpic}) are $\Gamma$-equivariant.
The following
 lemmas describe $\scrF$ and $\D^{2,\op}$
as spaces of orbits of certain group actions on $\scrS$.

The group $W$ contains a subgroup $W_0=\{1,\tau\}$, where $\tau$ is the longest element of $W$, which exchanges the positive and the negative Weyl chamber of $\Sigma$. In particular the map $\scrF\to \D^{2,\op}$ is $W_0$-equivariant, where $\tau$ acts on $\D^{2,\op}$ by flipping the two coordinates.

Recall our fixed embedding $\iota:\Sigma  \hookrightarrow \Upsilon$ from the previous section. Recall also that $M$ acts on $\scrS$ by precomposition via its action on
$\Upsilon$. 

Now let $M'$ be the stabilizer of $\iota(\Sigma)$ in $M$. Note that $M'$ is also the stabilizer of $\partial (\iota(\Sigma))$, since a maximal flat is uniquely defined by its boundary. Let $M_0$ (resp. $M'_0$) be the pointwise fixator of $\iota(\Sigma))$ (resp. $\partial(\iota(\Sigma))$). By definition, the groups $M'/M_0$ and $M'/M'_0$ act on $\Sigma$ and $\partial \Sigma$, so that there is an embedding of the group $M'/M_0$ in $\Aut(\Sigma)=W\ltimes A$ and of $M'/M_0'$ in $\Aut(\partial \Sigma)=W$.

\begin{proposition}\label{prop:M0'}

Consider the space $\scrS/M_0$ (resp. $\scrS/M_0'$), endowed with the quotient topology and the action of $M'/M_0$ (resp. $M'/M_0'$). Then:
\begin{enumerate}[(i)]
\item The groups $M'$, $M_0$ and $M'_0$ are unimodular. 
\item The image of $M'/M_0$ (resp. $M'/M'_0$) in $W\ltimes A$ (resp. $W$) is equal to $W_0\ltimes A$ (resp. $W_0$).
\item There is $\Gamma\times M'/M_0$-equivariant homeomorphism  $\scrS/M_0\to \scrF$,
\item There is a $\Gamma\times M'/M'_0$-equivariant homeomorphism $\scrS/M'_0\to \D^{2,\op}$. 
\end{enumerate}
\end{proposition}

\begin{proof}
The image of $\Sigma$ under the projection $\Upsilon \subset T\times \RR\to T$ is a geodesic $L$ in $T$. Note that $p_T(M')$ is exactly the stabilizer of $L$, whereas $p_T(M_0)$ (resp. $p_T(M_0')$) is the pointwise fixator of $L$ (resp. of $\partial L$). 

Since $p_T(M')$ acts $3$-transitively on $\partial T$, it contains an element which flips the two endpoints of $L$. By our choice of $\iota$, this elements acts as $\tau$ on $\Sigma$. Furthermore, the two endpoints corresponding to the $\RR$ factor in $\Upsilon\subset T\times \RR$ are fixed by $M'$, since $M'\subset \Aut^0(\Upsilon)$. This implies that the image of $M'/M'_0$ is exactly $W_0$. 
It also follows that the linear part of $M'/M_0$ is equal to $W_0$. Let $A'=M'/M_0\cap A$. To show that $A'=A$, we first notice that $M'$ contains $S$, so that the stabilizer of $\ell$ in  $A'$ acts transitively on the vertices of $\ell$.  Using again the $3$-transitivity of $M$ on $\partial T$, we see that  $M'$ contains an element which acts as a translation of length~$1$ on $L$. Hence $A'$   acts transitively on the set of all lines in $\Sigma$ which are parallel to $\ell$. Therefore $A'$ is transitive on the vertices of $\Sigma$, so that $A'=A$ since $A$ acts sharply transitively on those vertices. Since $M'/M_0$ contains a reflection whose linear part is $\tau$, it follows that $M'/M_0=W_0\ltimes A$.

To prove that $M'$ is unimodular, it suffices to prove that $M'/S$ is unimodular, since $S$ is discrete. But $M'/S\simeq p_T(M')$ is the stabilizer of $L$ in $T$, which contains  a translation of $L$. In particular $p_T(M')$ contains a cocompact lattice, isomorphic to $\ZZ$. Therefore it is unimodular. Any open normal subgroup of a unimodular group is also unimodular. Thus $M_0$ and $M'_0$ are unimodular.

The map $\scrS\to \scrF$ defined (as above) by $\psi\mapsto \psi_{|\iota(\Sigma)}$ is clearly $M_0$-invariant and $M'$-equivariant, so it factors through a map $\theta \colon \scrS/M_0\to\scrF$, which is $M'/M_0$-equivariant, and continuous by definition of the quotient topology. The injectivity of $\theta$ is clear. Let us  prove its surjectivity.
Using Lemma \ref{lem:twoop}, we see  that for every $u,v\in \D^{\op}_{\pm}$,  $\scrS$ contains an element $\phi$ such that $(\phi_+,\phi_-)=(u,v)$. It follows that the image of the map $\theta$ contains every  marked flat  up to translation. Since $M'/M_0$ contains $A$, it is always possible to translate in order to get back to the origin, so that the map $\scrS\to \scrF$ is indeed surjective.
If $\scrS_x=\{\phi\in \scrS\mid \phi(o,0)=x\}$ then we see that $\scrS_x$ is a compact open subset of $\scrS$. By compactness, $\theta_{|\scrS_x/M'_0}$ is a homeomorphism on its image, which is $\scrF_x$, hence is again open. It follows that $\theta$ is indeed a homeomorphism.

The homeomorphism between $\scrS/M_0'$ and $\D^{2,\op}$ is obtained by composing the map $\theta$ above with the map $\scrF\to \D^{2,\op}$. This new map clearly factors through $M_0'$. The proof that it is an equivariant homeomorphism is similar as before.
\end{proof}

\begin{corollary}\label{cor:flip}
There exists a closed subgroup $N<M^+$ and a continuous surjective homomorphism $N \to \{1,\tau\}$, such that the maps $\scrS \to \D^{2,\op}$ and $\scrS/S \to \D^{2,\op}$ are equivariant with respect to this homomorphism.
\end{corollary}

\begin{proof}
As above, the image of $\Sigma$ under the projection $\Upsilon \subset T\times \RR\to T$ is a geodesic line $L$ in $T$. 
We put $N=M'\cap M^+$. By Theorem \ref{thm:S}, the group $p_T(M^+)$ is still $3$-transitive on $\partial T$, hence $p_T(N)$ contains an element which flips the two endpoints of $L$. This element acts as $\tau$ on $\scrF$, hence also on $\D^{2,\op}$.
\end{proof}


\section{Measuring the topology} \label{sec:meas}

In this section, we retain the setting and the notation of the previous one and explain how to construct natural measures on the various topological spaces considered there.
In \S\ref{subsec:d}, we   give the construction of visual measures on $\D$ and we   use it in \S\ref{subsec:d2} to construct a Radon measure on $\D^{2,\op}$.
In \S\ref{subsec:dpm}, we   use the above to construct a Radon measure on $\D_\pm^\op$.
Using the measure on $\D_\pm^\op$ we are   able to define a measure on $\scrS$ and by this on the various spaces appearing in Diagram~(\ref{eq:broadpic}). That process will be carried out in \S\ref{subsec:measgen}.


\subsection{The visual measures on $\Ch(\Delta)$} \label{subsec:d}

Our goal is to define  natural probability measures $\mu_x$ on $\Ch(\Delta)$, indexed by $x\in X$. We follow the construction of J.~Parkinson from \cite{Parkinson}.

Recall that $A$ is the group of simplicial translations of our model apartment $\Sigma$ and $Q^+\subset \Sigma$ is the positive Weyl chamber.
We denote by $A^+$ the subsemigroup of $A$ formed by elements translating $0$ to a point in $Q^+$ (including possibly points on the walls). This gives a partial order on $A$, defined by $t>t'$ if $t-t'\in A^+$.

Let $\alpha_1$ and $\alpha_2$ be the simple roots associated with the choice of the positive Weyl chamber.
In other words, they are the $\ZZ$-linear forms on $\Sigma$ taking integer values on vertices, whose kernel are walls delimiting $Q^+$ which are positive on $Q^+$ and achieve the value $1$.  We also denote by $A^{++}$ the set of $\lambda\in A$ such that $\alpha_1(t)>0$ and $\alpha_2(t)>0$.

Let $t_1$ and $t_2$ be the two generators of $A$ such that every $t\in A$ is written $t=\alpha_1(t)t_1+\alpha_2(t)t_2$. We denote
$\ell(t)=\alpha_1(t)+\alpha_2(t)$.

 For $t\in A^+$ and $x$ a vertex in $X$ we denote by $V_t(x)$ the set of vertices $y\in X$ such that, in an apartment $F\subset X$ containing $x$ and $y$,  we have $t(x)=y$.

There is a natural action of the Weyl group $W$, isomorphic to the symmetric group $S_3$, on $\Sigma$. The group $W$ is generated by the two reflections $s_1$ and $s_2$, which act on $\Sigma$ by reflections with respect to the walls corresponding to $\alpha_1$ and $\alpha_2$ respectively.

\begin{lemma}\label{lem:NLambda}

Let $t\in A^{++}$ and $x$ be a vertex in $X$. Then the cardinality $N_t$ of $V_{t}(x)$ does not depend on $x$ and we have
$$N_t= K q^{2\ell(t)},$$
for some constant $K$ depending on $q$.
\end{lemma}

\begin{proof}
See \cite[Corollary 2.2]{CartwrightMlotowski} or \cite[Theorem 5.15]{ParkinsonHecke}. 
\end{proof}

We will need a more precise calculation. Fix a chamber $C\in\Ch(\Delta)$ and a vertex $x\in X$ contained in some apartment.
Consider the collection $F$ of all apartments in $X$ containing $x$ and $C$.
Each element of $F$ could be seen as a marked flat, by choosing its unique type preserving identification with $\Sigma$ taking $0$ to $x$ and $Q^+$ to  $Q(x,C)$.
For given $t\in A^+$ and $w\in W$ we consider in each  apartment in $F$ the image of the element $w\circ t(0)\in \Sigma$ under the associated identification.
We denote by $Y_w^t$ the set of all vertices of $X$ obtained this way.

\begin{lemma}\label{Ywlambda}

Denoting $(i,j)=(\alpha_1(t),\alpha_2(t))$, we have:
\[ |Y_{e}^t|=1, \quad |Y_{s_1}^t|= K_1q^{j}, \quad |Y_{s_2}^t|=K_2q^{i}, \quad |Y_{s_1s_2}^t|=K_3q^{2i+j} \quad\mbox{and}\quad |Y_{s_2s_1}^t|=K_4 q^{i+2j}, \]
for some constants $K_1, \dots , K_4$ depending only on $q$.
 In particular, for every $w\neq w_0$, the quantity $|Y_w^t|/|V_t(y)|$ tends to $0$ as $i$ and $j$ tend (simultaneously) to $+\infty$.
\end{lemma}

\begin{proof}
We argue by induction on $\ell(t)=i+j$. The first step of the induction is absorbed by the constants.
 Let us do the calculation of $|Y_{s_1}^t|$, the other ones being similar. We refer to Figure \ref{fig:counting}.

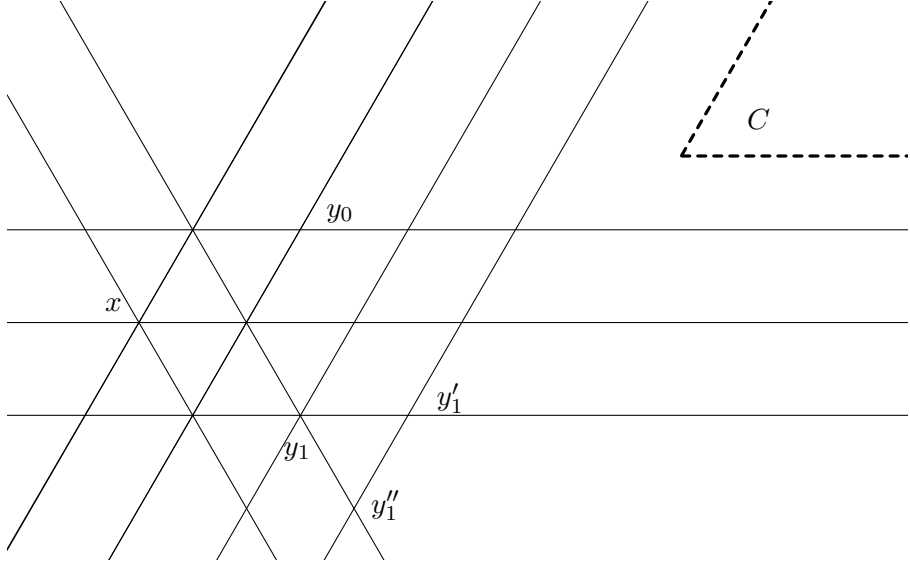
\begin{figure}[h]
\begin{center}
\begin{tikzpicture}[line cap=round,line join=round,>=triangle 45,x=0.5cm,y=0.5cm]
\clip(-4.31,-6.27) rectangle (19.39,8.51);
\draw [domain=-4.31:19.39] plot(\x,{(-0-0*\x)/1});
\draw [domain=-4.31:19.39] plot(\x,{(--2.07--2.46*\x)/1.42});
\draw [domain=-4.31:19.39] plot(\x,{(--2.46-0*\x)/1});
\draw [domain=-4.31:19.39] plot(\x,{(-1.73--0.87*\x)/-0.5});
\draw [domain=-4.31:19.39] plot(\x,{(-0.73-0.87*\x)/-0.5});
\draw [domain=-4.31:19.39] plot(\x,{(--1.73-0.87*\x)/-0.5});
\draw [domain=-4.31:19.39] plot(\x,{(--1.73-0.87*\x)/-0.5});
\draw [domain=-4.31:19.39] plot(\x,{(--0.73--0.87*\x)/-0.5});
\draw [domain=-4.31:19.39] plot(\x,{(-2.46-0*\x)/1});
\draw [domain=-4.31:19.39] plot(\x,{(--4.2-0.87*\x)/-0.5});
\draw (-2,0.9) node[anchor=north west] {$x$};
\draw (14.9,5.94) node[anchor=north west] {$C$};
\draw (3.83,3.4) node[anchor=north west] {$y_0$};
\draw (2.7,-2.9) node[anchor=north west] {$y_1$};
\draw (6.73,-1.33) node[anchor=north west] {$y'_1$};
\draw [domain=-4.31:19.39] plot(\x,{(--6.66-0.87*\x)/-0.5});
\draw (5,-4.18) node[anchor=north west] {$y''_1 $};
\draw [line width=1.2pt,dash pattern=on 3pt off 3pt,domain=13.444416691289469:19.39388933797909] plot(\x,{(-38.14--3.5*\x)/2.01});
\draw [line width=1.2pt,dash pattern=on 3pt off 3pt,domain=13.444416691289469:19.39388933797909] plot(\x,{(--19.39-0*\x)/4.39});
\end{tikzpicture}
\end{center}
\caption{Counting elements in $Y_{s_1}^t$}
\label{fig:counting}
\end{figure}

Fix $t=it_1+jt_2$. Let $t'=(i+1)t_1+jt_2$. Let $Y=Y_{s_1}^{t}$ and $Y'=Y_{s_1}^{t'}$. We consider the following map $\psi$: to a point $z'$ in $V_{t'}(x)$, we associate the unique point $z=\psi(z')$ in the intersection of $V_t(x)$ and the convex hull of $x$ and $z$. The map $\psi$ is surjective: indeed, if $z\in V_t(x)$, then $x$ and $z$ are in some common apartment, and it is possible to find $z'$ in this apartment such that $\psi(z')=z$. We also claim that the restriction of this map to $Y'$ is also injective.
Indeed, if $y_1\in Y$, then any apartment containing $y_1,x$ and $C$ also contains the convex hull of $y_1$ and $C$, and therefore already contains a point $y'_1\in Y'$. On the other hand, any point  $y'\in Y'$ such that $\psi(y')=y_1$ is contained in some apartment containing $y_1,x$ and $C$, and therefore must be equal to $y'_1$. This proves the claim.
Hence we have $|Y'|=|Y|=K_1q^{2j}$.

Now let $t''=it_1+(j+1)t_2$, and denote $Y''=Y_{s_1}^{t''}$. We still have a surjective map $Y''\to Y$, defined in the same way as the map $Y'\to Y$. However this map is no longer injective. To count the cardinality of the fibers, choose $y_1\in Y$. This gives a choice of $y'_1\in Y'$, defined as $y'_1=\psi^{-1}(y_1)$. There are $q$ chambers which have $y_1$ and $y'_1$ as their vertices, and are not in the convex hull of $C$ and $y_1$. The third vertex of each of these chambers is a possible choice for $y''_1\in Y''$. Hence $|Y''|=q|Y|=K_1q^{j+1}$.
\end{proof}

\begin{definition}
Let $x$ be a vertex in $X$. The \textbf{visual measure} based at $x\in X$, denoted $\mu_x$, is the unique measure on $\Ch(\Delta)$ such that for all $y\in V_t(x)$,
the set $\Omega_{x}(y)$ has measure $\frac{1}{N_t}$.
\end{definition}

The construction of this measure is justified more carefully in \cite{Parkinson} ; it follows from the fact that $\Ch(\Delta)$ can be viewed as the projective limit of sets $V_t(x)$ as $t$ grows in $A^+$. We note also that  $g.\mu_x=\mu_{gx}$. 

\begin{definition}

Let $C\in \Ch(\Delta)$. The \textbf{horofunction} based at $C$, denoted $h_C$ is a function $X\times X$ defined as follows. Let $z\in Q(x,C)\cap Q(y,C)$. Assume that $z\in V_s(x)\cap V_{t}(y)$. Then  $h_C(x,y) =s-t$.
\end{definition}

This definition does not depend on the particular choice of $z$ (see \cite[Theorem 3.4]{Parkinson}). 
The horofunction $h_C(\cdot,\cdot)$  satisfies the following cocycle relation \cite[Proposition 3.5]{Parkinson}:

\begin{lemma}\label{cocycle}
For all $x,y,z\in X$ and $C\in \Ch(\Delta)$, we have $h_C(x,y)=h_C(x,z)+h_C(z,y)$.
\end{lemma}

\begin{proposition}\label{prop:RN}

For $x,y$ vertices in $X$, the measure $\mu_x$ and $\mu_y$ are absolutely continuous relatively to each other. Furthermore, for all $C\in\Ch(\Delta)$, we have

$$\frac{\d\mu_x}{\d\mu_y}(C)=q^{2\ell( h_C(x,y))}$$
\end{proposition}

\begin{proof}
See \cite[Theorem 3.17]{Parkinson}.
\end{proof}

\begin{definition} \label{defmu}
We denote by $\mu$ the measure class of $\mu_x$.
\end{definition}

\begin{proposition}\label{prop:opposite}
The set $\Delta^{2,\op}$ is of conull measure (relatively to the measure class $\mu\otimes\mu$).
\end{proposition}

\begin{proof}
For every vertex $x\in X$, let $\Delta^2_x$ (resp. $\Delta^{2,\op}_x$) be the subset of $\Ch(\Delta)\times\Ch(\Delta)$ composed of pairs of chambers which are in the boundary of an apartment containing $x$ (resp. and are opposite).

Then $\Ch(\Delta)\times\Ch(\Delta)=\bigcup_{x\in X} \Delta^2_x$.
By countability it is sufficient to prove that for every $x$ we have $\mu_o^2(\Delta^{2,\op}_x)=\mu_o^2(\Delta^2_x)$, for a fixed vertex $o\in X$. Since $\mu_x$ and $\mu_o$ are in the same measure class, it is sufficient to prove that $\mu_x^2(\Delta^{2,\op}_x)=\mu_x^2(\Delta^2_x)$ for every vertex $x$.

By Fubini, it is enough to prove that for almost every chamber at infinity $C$, the set $\mathcal B$ of chambers $C'$ such that $(C,C')\in \Delta^{2,\op}_x$ is of full measure (in the set of chambers in a same apartment as $C$ and $x$).
Let $\mu_{x,C}$ be the restriction of the measure $\mu_x$ to the set $\mathcal B'$ of those chambers $C' \in \Ch(\Delta)$ with $(C,C')\in \Delta^2_x$. We have to prove that $\mu_{x,C}(\mathcal B)=\mu_{x,C}(\mathcal B')$.

Recall the sets $Y_w^t$ discussed in  Lemma~\ref{Ywlambda} and set $\Omega_x(Y^t_w)=\bigcup_{y\in Y_w^t}\Omega_x(y)$.
It is clear that $\mathcal B\subseteq \bigcap_{t\in A^+} \Omega_x(Y^t_{w_0})$.
We claim that this inclusion is actually an equality.
Indeed, if $C'\in \bigcap_{t\in A^+} \Omega_x(Y^t_{w_0})$, then every finite subset of $Q(x,C')\cup Q(x,C)$ is contained in some apartment, hence by Theorem \ref{aptfull} also $Q(x,C')\cup Q(x,C)$ is contained in an apartment, and then it is clear that $C$ and $C'$ are opposite in it.

Note that the intersection $\bigcap_{t\in A^+} \Omega_x(Y^t_{w_0})$ is decreasing, in the sense that for every $t,t'\in A^+$ such that $t'-t\in A^+$, we have $\Omega_x(Y^{t'}_{w_0})\subseteq \Omega_x(Y^t_{w_0})$. Hence it is enough to prove that $\mu_{x,C}(\Omega_x(Y^t_{w_0}))$ tends to $\mu_{x,C}(\mathcal B')$ when $t\to \infty$.
Equivalently, we need to show that that for every $w\neq w_0$, the limit of  $\mu_{x,C}(\Omega_x(Y^t_{w}))$ is $0$ as $t$ tends to infinity (simultaneously in both directions).
This follows from Lemma \ref{Ywlambda}, as $\mu_x(\Omega_x(Y^t_w))=|Y_w^t|/|V_t(y)|$.
\end{proof}

\subsection{The Radon measure on $\D^{2,\op}$} \label{subsec:d2}

In this section we construct a natural Radon measure on $\D^{2,\op}$, which 
is invariant by the actions of $\Gamma$. The existence of such a measure is well-known to experts (see for example \cite{pansu} for a possible definition); its construction is nevertheless usually different than what we propose here.

\begin{lemma}\label{lem:beta}
For $(C,C')\in \Delta^{2,\op}$ and $x$ a vertex in $X$,  the quantity
\[ \beta_x(C,C'):=h_C(x,z)+h_{C'}(x,z), \]
with $z$ a point in the apartment containing $C$ and $C'$, does not depend on $z$.

Furthermore,  for every vertex $y\in X$, we have
$$\beta_x(C,C')-\beta_y(C,C')=h_C(x,y)-h_{C'}(x,y).$$
\end{lemma}

\begin{proof}
Let us prove first that $\beta_x(C,C')$ is independent of $z$. Let $z$ and $z'$ be two points in the apartment $F$ containing $C$ and $C'$. Using the assumption that $C$ and $C'$ are opposite, we have that $h_C(z',z)=-h_{C'}(z',z)$.

Hence by Lemma \ref{cocycle} we have
\begin{align*}
\beta_x(C,C')&=h_C(x,z)+h_{C'}(x,z)\\
&= h_C(x,z')+h_C(z',z)+h_{C'}(x,z')+h_{C'}(z',z)\\
&=h_C(x,z')+h_{C'}(x,z').
\end{align*}

Now let us turn to the second equality. Fix $x,y\in X$. Then, using Lemma \ref{cocycle} again, we have

\begin{align*}
\beta_x(C,C')-\beta_y(C,C')&=h_C(x,z)+h_{C'}(x,z)-h_C(y,z)-h_{C'}(y,z)\\
&=h_C(x,z)+h_{C'}(x,z)+h_C(z,y)+h_{C'}(z,y)\\
&=h_C(x,y)+h_{C'}(z,y).\qedhere
\end{align*}
\end{proof}

For a fixed vertex $x$ in $X$ we define the measure $m_x$ on $\Ch(\Delta)\times \Ch(\Delta)$
by the formula
$$\d m_x(C,C')=q^{-2\ell(\beta_x(C,C'))}\d \mu_{x}(C) \d\mu_x(C').$$
By Lemma~\ref{lem:beta} and Proposition \ref{prop:RN}, using also the fact that $\ell(t+t')=\ell(t)+\ell(t')$ for every $t,t'\in A$,
we have
$$\frac{\d m_x}{\d m_y}(C,C')=q^{-2\ell( \beta_x(C,C')) +2\ell( \beta_y(C,C')) }\frac{\d \mu_x}{\d \mu_y}(C)\frac{\d \mu_x}{\d \mu_y}(C')=1.$$
Thus the measure $m_x$ is independent on the choice of $x$.
We get that this measure is $\Gamma$-invariant, as for all $g\in \Gamma$ we have $g\mu_x=\mu_{gx}$ and $\beta_{gx}(gC,gC')=\beta_x(C,C')$,
hence $gm_x=m_{gx}=m_x$.

Using Proposition~\ref{prop:opposite} we view $m_x$ as a measure on $\Delta^{2,\op}$
which, in view of its independence on $x$, we denote $m$.
Thus $m$ is a $\Gamma$-invariant measure on $\Delta^{2,\op}$.

Recall that for a vertex $x\in X$ we defined $\scrF_x$ to be the space simplicial embeddings $f \colon  \Sigma\to X$ satisfying $f(0)=x$ and let $\scrF'_x$ be the image of $\scrF_x$ under the projection  $\scrF\to\D^{2,\op}$. Note that $\scrF\to \D^{2,\op}$ is injective in restriction to $\scrF_x$, thus induces a homeomorphism $\scrF_x\to\scrF'_x$  by Lemma~\ref{lem:op=F}.

\begin{lemma}\label{lem:finite}
We have $m(\scrF'_x)<\infty$.
\end{lemma}

\begin{proof}
By definition, $m(\scrF'_x)=\int_{(C,C')\in {\scrF'_x}} q^{2\ell(\beta_x(C,C'))}\d\mu_x(C)\d \mu_x(C')$.

By Lemma \ref{lem:beta}, if $(C,C')\in \scrF'_x$, we have $\beta_x(C,C')=h_C(x,x)-h_{C'}(x,x)=0$. Hence $m(\scrF'_x)=(\mu_x\otimes\mu_x)(\scrF'_x)<\infty$.
\end{proof}

\begin{corollary} \label{cor:minvariant}
The measure $m$ is a $\Gamma$-invariant Radon measure on the locally compact space $\D^{2,\op}$ which is in the measure class of 
$\mu\times \mu$.
\end{corollary}

\begin{proof}
The only thing that does not immediately follow from the discussion above is the fact that $m$ is finite on compact sets.
This follows from Lemma~\ref{lem:finite}, as $\{\scrF'_x\mid x\in X^{(0)}\}$ is an open cover of $\D^{2,\op}$.
\end{proof}


\subsection{The Radon measure on $\D_\pm^\op$} \label{subsec:dpm}

The first part of this section is devoted to the construction of a Radon measure $m_\pm$ on $\D_\pm^\op$, see Definition~\ref{defn:mpm} and Lemma~\ref{lem:mpmindependent} below. The construction is similar to the construction of $m$ given in the previous section. The second part of this section is a preperation for the proof that the measures $m$ and $m_\pm$ are compatible which will be given in the next section, see Theorem~\ref{thm:compactiblemeasure}.

 For $u\in \D_+$ (resp. $v\in \D_-$), we define the horofunction $h_u$ (resp. $h_v$) by $h_u(x,y)=\alpha_1(h_C(x,y))$, for any chamber $C$ adjacent to $u$ (resp. $h_v(x,y)=\alpha_2(h_C(x,y))$ for any chamber $C$ adjacent to $v$). It is easy to check that this quantity does not depend on the choice of $C$. Let $\pr_+:\Ch(\Delta)\to \Delta_+$ be the natural projection. We define $\mu_{x,+}=(\pr_+)_ * \mu_x$.
Similarly, we define $\mu_{x,-}=(\pr_{-})_*(\mu_x)$, where $\pr_-:\Ch(\Delta)\to \Delta_-$. Recall also the notation $\Omega^+_x(y)=\pr_+(\Omega_x(y))$ and $\Omega^-_x(y)=\pr_-(\Omega_x(y))$.

\begin{lemma}\label{lem:mu+(Omega)}

Fix $s\in A^{++}$, and $z\in V_s(x)$. Then we have $\mu_{x,+}(\Omega_x^+(y))=\frac{1}{K_1q^{2\alpha_1(s)}}$ for some constant $K_1$ depending only on $q$.
Similarly, we have $\mu_{x,-}(\Omega^-_x(y))=\frac{1}{K_2q^{2\alpha_2(s)}}$ for some constant $K_2$.
\end{lemma}

\begin{proof}
We treat only the case of $\mu_{x,+}$, the other case being similar. Assume first that $s\neq 0$ satisfies $\alpha_2(s)=0$ (the result is also true in this case, even if $s\not\in A^{++}$). Let $y\in V_s(x)$. The point
 $y$ is on some geodesic ray from $x$ to a vertex of type $+$. Then $\Omega_x^+(y)$ is the set of $u$ such that $y\in [x, u)$. If $C\in \Ch(\Delta)$ is a chamber adjacent to $u$, then this ray belongs to $Q(x,C)$, and in particular $C\in \Omega_x(y)$. Hence $\pr_+^{-1}(\Omega^+_x(y))=\Omega_x(y)$. So $\mu_{x,+}(\Omega_x^+(y))=\mu_x(\Omega_x(y))=\frac{1}{N_s}$. In this case, by \cite[Theorem 5.15]{ParkinsonHecke} we have $N_s=K_1q^{2\alpha_1(s)}$ for some constant $K_1$, hence the result.

Now let us treat the case of  $s\in A^{++}$. Let $z\in V_s(x)$ and let $s_1=\alpha_1(s) t_1$. Let $y$ be the point in $V_{s_1}(x)$ which is in an apartment containing $x$ and $z$. Note that since $s\in A^{++}$, we have $s_1\neq 0$, hence $y\neq x$.

Let $u\in\Omega^+_x(y)$. In other words, $u$ is the endpoint of some geodesic starting by the segment $[x, y]$. We claim that there exists an apartment containing $u,x$ and $z$. Indeed, let $y_0=y,y_1,\dots,y_n,\dots$ be the consecutive vertices of the ray $[yu)$, and let $F_n$ be the combinatorial convex hull of $y_n$, $x$ and $z$. We prove by induction that, for every $n$, $F_n$ is in some sector based at $x$, and furthermore, seeing $\alpha_1$ as a linear form on this sector in the obvious way, we have $\alpha_1( y')\leq \alpha_1(y_n)$.

The induction basis being clear, let us show the induction step. Assume that $F_n$ is contained in some apartment $A_n$. Furthermore, the induction hypothesis implies that $F_n$ is contained in the half-space $H$ delimited by $y_n+\Ker(\alpha_1)$. There exists a simplex adjacent to $H$ which contains $y_{n+1}$. By Theorem \ref{aptfull}, the union of $H$ and this chamber is contained in an apartment. It follows that $F_{n+1}$ is also contained in an apartment. Now it is easy to see that it is actually contained in a sector and that $\alpha_1(y')\leq \alpha_1(y_{n+1})$ for every $y'\in F_{n+1}$.

This proves the claim. It follows that there exists some chamber $C\in \Omega_x(z)$ such that $u\in\pi_+(C)$, so that $\Omega_x^+(y)\subset \Omega_x^+(z)$. Since it is clear that $\Omega_x(z)\subset\Omega_x(y)$, we have that $\Omega_x^+(y)=\Omega_x^+(z)$, and the result follows.
\end{proof}

\begin{lemma}
Let $x,y\in X^{(0)}$.
 For all $u\in\D_+$ and $v\in \Delta_-$ we have
$$\frac{\d\mu_{x,+}}{\d\mu_{y,+}}(u)=q^{2h_u(x,y)}
\hspace{1cm} \text{and} \hspace{1cm}
\frac{\d\mu_{x,-}}{\d\mu_{y,-}}(v)=q^{2h_v(x,y)}.$$
\end{lemma}

\begin{proof}
Let $t\in A^+$,  and $z\in V_t(x)$.
 Let $C\in \Omega_x(z)$. By  \cite[Lemma 3.13]{Parkinson} we have, as soon as $t$ is large enough, $z\in V_{t-h_C(x,y)}(y)$ and $\Omega_x(z)=\Omega_y(z)$. Hence we get $\mu_{y,+}(\Omega^+_x(z))=\frac{1}{Kq^{2(\alpha_1(t)-h_u(x,y))}}$ while $\mu_{x,+}(\pi_+(\Omega_x(z)))=\frac{1}{Kq^{2\alpha_1(t)}}$. As $t$ is arbitrary large, this proves the result.
\end{proof}

\begin{definition} \label{defn:mpm}
Let $(u,v)\in \D_\pm^\op$, and $x\in X^{(0)}$. We define  $\beta_x(u,v)=h_u(x,z)+h_v(x,z)$ where $z$ is a point on a geodesic between $u$ and $v$. This does not depend on the choice of $z$.
We define the measure $m_\pm$ on $\D_\pm^\op$ by
$$\d m_\pm(u,v)=q^{2\beta_x(u,v)} \d\mu_{x,+}(u)\d \mu_{x,-}(v).$$
\end{definition}

\begin{lemma}\label{lem:mpmindependent}
The measure $m_\pm$ does not depend on the choice of $x$. It is a $\Gamma$-invariant Radon measure on $\D_\pm^\op$. Furthermore, the map $(\D^{2,\op},m)\to ( \D_\pm^\op,m_\pm)$ is measure class preserving.
\end{lemma}

\begin{proof}
The first part follows from similar calculations as in \S\ref{subsec:d2}.
For the latter part, observe that (for a given $x$) $m$ is in the measure class of 
$\mu_{x}\times  \mu_{x}$, $m_\pm$ is in the measure class of
$\mu_{x,+}(u)\times \mu_{x,-}$ and the map 
$(\D^{2,\op},\mu_{x}\times  \mu_{x})\to ( \D_\pm^\op,\mu_{x,+}(u)\times \mu_{x,-})$ is measure preserving.
\end{proof}

We now want to investigate more precisely the relationship between $m_\pm$ and $m$. Our goal is to obtain a preliminary relation between $\mu_x$ and $m_\pm$ (Proposition \ref{prop:desintmu}). We need first to introduce a few more notations.

 Let $\Delta_{u,v}$ be the set of chambers $C$ which are adjacent to both a chamber adjacent to $u$ and a chamber adjacent to $v$. The set of pairs of distinct chambers in $\Delta_{u,v}$ is the preimage of $(u,v)$ by the map $\D^{2,\op}\to \D_\pm^\op$.

Recall from \S\ref{subsec:paneltree} the construction of the panel tree $T_u$. The topology of the boundary of $T_u$ is generated by the sets $\Omega^T_x(y)$ of endpoints of geodesic rays starting from $x$ passing through $y$.
There is a natural projection $\pi_u \colon X\to T_u$, which associates to a point $x\in X$ the class of the geodesic ray $[x, u)$. If $C$ is a chamber adjacent to $u$, then $\pi_u(Q(x,C))$ is a half-line in $T_u$. The endpoint of this half-line does not depend on $x$. This defines a map from the set $\Ch(u)$ of chambers adjacent to $u$ to the boundary of $T_u$. This map is a homeomorphism.

\begin{lemma}
Let $p_{u,v}$ be the restriction of the projection map $\proj_u \colon \Ch(\D)\to \Ch(u)$  to $\D_{u,v}$. Then $p_{u,v}$ is a homeomorphism.
\end{lemma}

\begin{proof}
The inverse image by $\pi_u$ of a geodesic ray is a half-flat. The boundary of this flat (in $X$) will consist of three chambers, exactly one of them being in $\Delta_{u,v}$. This proves the bijectivity of $p_{u,v}$.

Let us fix a point $x\in X$, and let $x'=\pi_u(x)$. Fix also $y'\in T_u$ and $C\in p_{u,v}^{-1}(\Omega^T_x(y'))$. Then $\pi_u(Q(x,C))$ is a geodesic ray in $T_u$ passing through $y'$. So there exists $y\in Q(x,C)$ such that $\pi_u(y)=y'$.  Conversely, if $y\in X$ is such that $\pi_u(y)=y'$, let $C\in\Omega_x(y)\cap \D_{u,v}$. Then $y\in Q(x,C)$, so that $y'$ is on the geodesic ray from $x'$ to $p_{u,v}(C)$.
 Hence $p_{u,v}^{-1}(\Omega^T_{x'}(y'))$ is the union of all $\Omega_x(y)\cap \Delta_{u,v}$, where $y$ varies in $\pi_u^{-1}(y')$. So it is open in $\D_{u,v}$, which proves that $p_{u,v}$ is continuous.
 Since $\D_{u,v}$  is compact (as an intersection of closed subsets of $\Ch(\Delta)$), $p_{u,v}$ is a homeomorphism.
\end{proof}

We can define a measure on $\partial T_u$ in a similar way as we defined the measure on $\Ch(\Delta)$. For $x,y\in T_u$, let $\Omega^T_x(y)$ be the set of $\xi\in \partial T_u$ such that $y$ is on the geodesic ray $[x, \xi)$. We define the measure $\mu^T_x$ by the formula $\mu^T_x(\Omega_x^T(y))=\frac{1}{K' q^{d(x,y)}}$ valid for all  $x\neq y\in T_u$ and for some normalization constant $K'$ (which is actually equal to $\frac{q+1}{q}$).

This allows us to define a measure on $\D_{u,v}$, via the map $p_{u,v}$: for $x\in X^{(0)}$, we define the measure $\mu_{u,v;x}$ as $\mu_{u,v;x}=(p_{u,v}^{-1})_*\mu^T_{\pi_u(x)}$.

\begin{lemma}\label{desint m}
Let $y\in V_t(x)$.
The measure $\mu_{u,v;x}$ satisfies $\mu_{u,v;x}(\Omega_x(y))=\frac{1}{K'q^{\ell(t)}}$ whenever $\Omega_x(y)\cap \D_{u,v}\neq \varnothing$.
\end{lemma}

\begin{proof}
Assume that $\Omega_x(y)\cap \Delta_{u,v}$ is not empty. Let $C\in\Omega_x(y)\cap \Delta_{u,v}$, and let $C_1=\proj_u(C)$. By definition $C$ and $C_1$ share a vertex of type $-$, say $v'$. It follows that the geodesic ray $[xv')$ is contained in both $Q(x,C_1)$ and $Q(x,C)$.

Let $y'$ be the point on this geodesic ray such that $\pi_u(y)=\pi_u(y')$. We first claim  that
$$p_{u,v}(\Omega_x(y')\cap \D_{u,v})=\Omega^T_{\pi_u(x)}(\pi_u(y)).$$

Indeed, it is clear that $p_{u,v}(\Omega_x(y')\cap \D_{u,v})\subset \Omega^T_{\pi_u(x)}(\pi_u(y)).$ Let $C\in p_{u_,v}^{-1}( \Omega^T_{\pi_u(x)}(\pi_u(y)))$, and let $C'_1=\proj_u(C')$. As $C'_1$ is adjacent to $u$, the sector $Q(x,C'_1)$ contains the geodesic ray $[xu)$. Furthermore, since $\xi\in\Omega^T_{\pi_u(x)}(\pi_u(y))$, the sector $Q(x,C'_1)$ must contain some point $z$ with $\pi_u(y)=\pi_u(z)$, hence the geodesic ray $[zu)$. But $y'$ is in the convex hull of this geodesic ray and $x$. This proves that $y'\in Q(x,C'_1)$. Let $v''$ be the vertex of type $-$ of $C'_1$. It follows that $y'\in [x v'')$. Since $[xv'')\subset Q(x,C')$, we also have $y'\in Q(x,C')$, or equivalently, $C'\in\Omega_x(y')$, which proves the claim.

Now, since it is clear that $p_{u,v}(\Omega_x(y)\cap\Delta_{u,v})\subset \Omega^T_{\pi_u(x)}(\pi_u(y))$, it follows that $\Omega_x(y)\cap \Delta_{u,v}\subset \Omega_x(y')\cap \Delta_{u,v}$. Now let $C_2=\proj_v(C)$, $u'$ the vertex contained in both chambers $C$ and $C_2$, and $y''$ the point on $[xu')$ such that $\pi_v(y)=\pi_v(y'')$. By a symmetric argument, we have $\Omega_x(y)\cap \Delta_{u,v}\subset \Omega_x(y'')\cap\Delta_{u,v}$. The same argument, replacing $y$ with $y'$, leaves to the inclusion $\Omega_x(y')\cap \Delta_{u,v}\subset \Omega_x(y'')\cap\Delta_{u,v}$.

So we have $\Omega_x(y')\cap \D_{u,v}\subset \Omega_x(y')\cap\Omega_x(y'')\cap \D_{u,v}$. Since $y$ is in the convex hull of $y'$ and $y''$, we have $\Omega_x(y')\cap \Omega_x(y'')\subset \Omega_x(y)$. Therefore, we have the equality
$$p_{u,v}(\Omega_x(y)\cap \D_{u,v})=\Omega^T_{\pi_u(x)}(\pi_u(y)).$$

Hence, we get that $\mu_{u,v;x}(\Omega_x(y))=\frac{1}{K' q^{d(\pi_u(x),\pi_u(y)}}$. By definition, we have $\pi_u(x)=\pi_u(y')$. Let $t'$ be such that $y'\in V_{t'}(x)$. In the chamber $C$, it is easy to see that
$\ell(t)=\ell(t')$, and looking in the chamber $C_1$ we see that $d(\pi_u(x),\pi_u(y'))=\ell(t')$.
\end{proof}

Finally, we need a calculation similar to that of Lemma \ref{Ywlambda}. Let $x\in X^{(0)}$, $t\in A^+$, and fix $y\in V_t(x)$.
For $\eps=+$ or $-$, let $Z_\eps(y)$ be the set of $y'\in V_t(x)$ such that there exists some $C\in \Omega_x(y)$ and $C'\in \Omega_x(y)$ which is $\eps$-adjacent to $C$.

\begin{lemma}\label{Zeps(y)}
Assume that $t\in A^{++}$. Then there exist constants $K_+$ and $K_-$ such that $|Z_+(y)|=K_+ q^{\alpha_1(t)}$ and $|Z_- (y)|=K_- q^{\alpha_2(t)}$.
\end{lemma}

\begin{proof}
We treat the case of $Z_+$. We proceed by induction on $\max(\alpha_1(t),\alpha_2(t))$. Let $t'=t+t_1$ and $t''=t+t_2$.
As in Lemma \ref{Ywlambda}, we use the projections $p_1:V_{t'}(x)\to V_t(x)$ and $p_2:V_{t''}(x)\to V_t(x)$, which associates to a point $y$ the unique point in $V_t(x)$ which is in the convex hull of $x$ and $y$.

Let $y'\in V_{t'}(x)$ and $y''\in V_{t''}(x)$ be such that $p_1(y')=p_2(y'')$. Let $y=p_1(y')$. It is sufficient to prove that $Z_+(y')$ has the same cardinality as $Z_+(y)$ and that $|Z_+(y'')|=q |Z_+(y)|$.
Note that $p_1$ (resp. $p_2$) restricts to a map $p_1':Z_+(y')\to Z_+(y)$ (resp. $p_2':Z_+(y'')\to Z_+(y)$). We prove first that $p_1'$ is bijective.

 Let $y_1\in Z_+(y)$. By definition there exists $C\in \Omega_x(y)$ and $C_1\in \Omega_x(y_1)$ which share a vertex of type $+$, say $u$. Let $y_1'$ be the first vertex on the geodesic ray $[y_1u)$ after $y_1$. Then $y_1\in V_t(x)$ satisfies $p_1(y'_1)=y_1$. Furthermore, any $z_1\in p_1^{-1}(y_1)$ is on this geodesic ray, so that $z_1=y_1'$. Hence $p'_1$ is bijective, and we have $|Z_+(y')|=|Z_+(y)|=K_+ q^{\alpha_1(t)}= K_+ q^{\alpha_1(t')}$ by induction.

Consider now a simplex with $[y_1, y'_1]$ as an edge, and let $y''_1$ be the third vertex of this simplex. Since $\alpha_2(t)>0$, there is one choice of such a simplex for which $y''_1\not\in V_{t''}(x)$ (the one in the convex hull of $y'_1$ and $x$), and $q$ other choices. Each of these other choices is a point in $p_2^{-1}(y_1)$. So $|Z_+(y'')|=q|Z_+(y)|=K_+ q^{\alpha_1(t'')}$.
\end{proof}

\begin{proposition}\label{prop:desintmu}
For every $x\in X^{(0)}$ we have
$$\mu_x=\int \mu_{u,v;x} \d \mu_{x,+}(u) \d \mu_{x,-}(v).$$
\end{proposition}

\begin{proof}
Let $\widetilde\mu=\int \mu_{u,v;x} \d \mu_{x,+}(u) \d \mu_{x,-}(v)$.
 It suffices to prove that $\widetilde\mu(\Omega_x(y))=\frac{1}{Kq^{2\ell(t)}}$ for every $y\in V_t(x)$. In fact, it is sufficient to prove this for every $t$ large enough. So we can assume that $t\in A^{++}$.

  By Lemma \ref{desint m}, we have $\mu_{u,v;x}(\Omega_x(y))=\frac{1}{K'q^{\ell(t)}}$ whenever $\Omega_x(y)\cap \D_{u,v}\neq \varnothing$, and $0$ if not.  So we have to calculate the $\mu_{x,+}\times \mu_{x,-}$ measure of the set $D$ of $(u,v)$ such that  $\Omega_x(y)\cap \D_{u,v}\neq \varnothing$.

 We claim that
 $$D=\bigcup_{z_1\in Z_+(y)}\bigcup_{z_2\in Z_-(y)} (\Omega^+_x(z_2)\times \Omega^-_x(z_1))\cap \D_\pm^\op.$$
 
Let $(u,v)\in D$. Let $C\in \Omega_x(y)\cap \D_{u,v}$ and let $C_2=\proj_u(C)$. By definition $C_2$ and $C$ have a common vertex of type $-$. Let $z_2\in Q(x,C_2)\cap V_t(x)$. Then we have $z_2\in Z_-(y)$, and $u=\pi_+(C_2)\in \Omega^+_x(z_2)$. Similarly we find $z_1\in Z_+(y)$ such that $v\in \Omega_-(z_1)$.


  Conversely, let   $z_1\in Z_+(y)$ and $z_2\in Z_-(y)$.  By definition there exists a chamber $C_2$ in $\Omega_x(z_2)$ which is adjacent to some chamber in $\Omega_x(y)$. Let $u=\pi_+(C_2)$ and $v'=\pi_-(C_2)$. Let  $v\in \Omega_x^-(z_1)$ and assume that $u$ and $v$ are opposite. 
 Let $C=\proj_{v'}(v)$ and $u'=\pi_+(C)$.
   Note that $Q(x,C)$ is the sector whose boundary is the union of the geodesic rays from $x$ to $u'$ and from $x$ to $v'$. By construction $y$ is in the convex hull of these two geodesic rays. Hence $y\in Q(x,C)$, or in other words,
   $C\in \Omega_x(y)$, and  $C$ is also in $\D_{u,v}$, which proves that $(u,v)\in D$.

Thus the claim is proven and we deduce that
 $$(\mu_{x,+}\times\mu_{x,-})(D)=\sum_{z_1\in Z_+(y)}\sum_{z_2\in Z_-(y)} \mu_{x,+}(\Omega_x^+(z_2))\mu_{x,-}(\Omega_x^-(z_1)).$$
 From Lemma \ref{Zeps(y)} we know that $|Z_+(y)|=K_+q^{\alpha_1(t)}$ and $|Z_-(y)|=K_-q^{\alpha_2(t)}$.
 Furthermore, for every $z\in V_t(x)$, we have 
$\mu_{x,+}(\Omega^+_x(z)))=\frac{1}{Kq^{2\alpha_1(t)}}$ and $\mu_{x,-}(\Omega_x^-(z))=\frac{1}{Kq^{2\alpha_2(t)}}$. 
Hence we find that 
$$(\mu_{x,+}\times\mu_{x,-})(D)=K_+q^{\alpha_1(t)}\times K_-q^{\alpha_2(t)}\times \frac{1}{K_1' K_2'q^{2\alpha_2(t)+2\alpha_1(t)}}.$$
It follows that $\widetilde\mu(\Omega_x(y))=\frac{1}{K''q^{2\ell(t)}}$ for some constant $K''$. In other words $\widetilde\mu$ is a multiple of $\mu_x$. Since both are probability measures, we have $\widetilde\mu=\mu_x$.
\end{proof}

\subsection{Measuring Diagram~(\ref{eq:broadpic})} \label{subsec:measgen}

The following is a general principle.

\begin{lemma} \label{lem:MeasHaarFiber}
Let $H,\Lambda$ be second countable locally compact groups and let $U$ be a locally compact metrizable space on which $H\times \Lambda$ acts continuously by homeomorphisms.
Assume the $H$-action on $U$ is proper and free.
Assume further that $H$ is unimodular and $\nu$ is a 
$\Lambda$-invariant Radon measure on $H \backslash U$.
Then there exists an $H\times \Lambda$-invariant Radon measure $\zeta$
on $U$ and an $H$-equivariant measure preserving Borel isomorphism
\[ (U,\zeta) \simeq (H\backslash U,\nu) \times (H,\eta), \]  
where $\eta$ is a Haar measure on $H$ and $H$ acts on $H\backslash U \times H$ via its left action on the second coordinate.
In particular, the map $(U,\zeta) \to (H \backslash U,\nu)$ is measure class 
preserving. 

Furthermore, for any closed unimodular subgroup $H'<H$
and every unimodular closed normal subgroup $H'_0\lhd H'$ such that $H'/H'_0$ is unimodular, there exists an $H'/H'_0\times \Lambda$-invariant Radon measure $\zeta'$
on $H'_0\backslash U$ 
and an $H'/H'_0$-equivariant measure preserving Borel isomorphism
\[ (H'_0\backslash U,\zeta) \simeq (H\backslash U,\nu) \times (H_0'\backslash H,\eta'), \]  
where $\eta'$ is a right $H$-invariant Radon measure on $H'_0\backslash H$ and $H'/H'_0$ acts on $H\backslash U \times H'_0\backslash H$ via its left action on the second coordinate.
In particular, the maps $(U,\zeta) \to (H'_0 \backslash U,\zeta')$ and
$(H'_0 \backslash U,\zeta') \to (H \backslash U,\nu)$ are measure class 
preserving. 
\end{lemma}

\begin{proof}
In what follows we set $V=H\backslash U$ and 
fix a Haar measure $\eta$ on $H$.
Note that $\eta$ is both left and right invariant, by the unimodularity 
of $H$.

Given $\phi\in C_c(U)$ we let $\bar{\phi}\in C_c(U)$ be defined by the 
formula $\bar{\phi}(u)=\int_H \phi(hu) d\eta(h)$ 
and note that, by the right invariance of $\eta$, $\bar{\phi}$ is $H$-invariant.
We consider, as we may, $\bar{\phi}$ as an element of $C_c(V)$
and set $I(\phi)=\int_V \bar{\phi}(v)d\nu(v)$.
By Riesz representation theorem we obtain a Radon measure $\zeta$ on 
$U$ such that $\int_U \phi d\zeta=I(\phi)$
and note that $\zeta$ is $H$-invariant, as $I$ is, and that the map $(U,\zeta) \to (V,\nu)$ is measure class 
preserving. 

By the Federer--Morse selection theorem
(see \cite[Lemmas 3 and 4]{BorelSelection} for a nice presentation) we
obtain a Borel isomorphism $U \simeq V \times H$ which is
$H$-equivarant with respect to the left $H$ action on the second coordinate of $V\times H$
and commutes with the map $U\to V$ and the projection $V \times H\to V$.
Note that under this isomorphism, by its construction, the measure $\zeta$ is identified with $\nu \times \eta$.
Note also that $H$ acts on $V\times H$ via its right action on the second
coordinate (the corresponding $H$-action on $U$ is Borel, but need not
be continuous). This action is measure preserving, by the unimodularity
of $H$. The $\Lambda$-action on $U$ gives rise to a $\Lambda$-action on $V\times H$ which is given by a cocycle $c:\Lambda \times V\to H$,
$\lambda(v,h)=(\lambda v,hc(\lambda,v))$.
This action preserves $\nu\times \eta$. Indeed, for every $\phi\in L^1(V\times H,\nu\times \eta)$ and $\lambda\in \Lambda$, using Fubini and the right $H$-invariance of 
$\eta$ we get 
\[ \int_{V\times H} d(\nu\times \eta) \lambda\cdot \phi =\int_V d\nu(v)
\int_H d\eta(h) \lambda\cdot \phi(v,h)=
\int_V d\nu(v)
\int_H d\eta(h) \phi(v,hc(\lambda,v))=
\]
\[
\int_V d\nu(v)
\int_H d\eta(h) \phi(v,h)=
\int_{V\times H} d(\nu\times \eta)\phi. \]
We conclude that $\zeta$ is 
indeed $\Lambda$-invariant.

Assume now given a closed unimodular subgroup $H'<H$
and a closed normal subgroup $H'_0\lhd H'$.
As $H$ and $H'_0$ are unimodular, we get a right $H$-invariant Radon measure $\eta'$ on $H'_0\backslash H$.
We claim that $\eta'$ is also invariant under the left $H'/H'_0$-action on $H'_0\backslash H$.
Indeed, we may identify $H'_0\backslash H$ with the coset space
$(H\times H'/H'_0)/D$, where $D<H\times H'/H'_0$ is the image of $H'$ under
the obvious diagonal homomorphism, and use the fact that both $H\times H'/H'_0$ and $D\simeq H'$ are unimodular.
We denote $U'=H'_0\backslash U$ and identify it with 
$V \times H'_0\backslash H$. We endow $V \times H'_0\backslash H$ with the measure $\nu\times \eta'$ and denote by $\zeta'$ the corresponding measure on $U'$. Clearly, $\zeta'$ is $H'/H'_0$-invariant. Expressing the $\Lambda$-action on $V \times H'_0\backslash H$ using the cocycle $c$, we get as before that $\zeta'$ is also
$\Lambda$-invariant. 
\end{proof}

For the following theorem, recall that $W_0=\{1,\tau\}$ is the subgroup of $W$ generated by the longest element $\tau$.

\begin{theorem} \label{thm:compactiblemeasure}
There exist Radon measures $\zeta$ on $\scrS$ and $\zeta'$ on $\scrF$
such that $\zeta$ is  
$\Gamma\times M$-invariant 
and $\zeta'$ is $\Gamma\times (W_0\ltimes A)$-invariant
and such that the maps
$(\scrS,\zeta)\to (\scrF,\zeta')$
and
$(\scrF,\zeta')\to (\D^{2,\op},\mu\times \mu)$
appearing in Diagram~(\ref{eq:broadpic})
are measure class preserving.
\end{theorem}

\begin{proof}
The construction of $\zeta$ is obtained by specializing Lemma~\ref{lem:MeasHaarFiber} to $U=\scrS$, $H=M$
and $\Lambda = \Gamma$.
Note that by Lemma~\ref{lem:scrs}, the $H$ action on $U$ is proper and free,
and $U/H$ could be identified with $\D_\pm^\op$.
By Theorem~\ref{thm:S} $H$ is unimodular and by Lemma~\ref{lem:mpmindependent} the measure $\nu$ on $U/H$ corresponding to $m_\pm$ on $\D_\pm^\op$ is $\Lambda$-invariant.
We thus obtain a Radon measure $\zeta$ on $\scrS$ which is 
$\Gamma\times M$-invariant and such that $(\scrS,\zeta)\simeq (\D_\pm^\op \times M,m_\pm\times \eta)$,
where $\eta$ is the Haar measure on $M$.
We let $M',M_0,M_0'$  be as in Proposition~\ref{prop:M0'} and note that, by this proposition, these groups are unimodular, and the quotients $M'/M_0$ and $M'/M'_0$ are discrete, hence unimodular. 
Applying Lemma~\ref{lem:MeasHaarFiber} again we obtain Radon measures
$\zeta'$ and $\zeta''$ on $M_0\backslash\scrS$ and $M_0\backslash\scrS$
which are $\Gamma\times M'/M_0$-invariant and $\Gamma\times M'/M'_0$-invariant correspondingly and such that 
\[ (M_0\backslash \scrS,\zeta') \simeq (M\backslash \scrS,\nu) \times (M_0\backslash M,\eta') \quad \mbox{and} \quad  
(M'_0\backslash \scrS,\zeta'') \simeq (M\backslash \scrS,\nu) \times (M_0'\backslash M,\eta''),\]  
where $\eta'$ and $\eta''$ are the obvious corresponding measures.
In particular, as $M'_0<M_0$, we get that the maps 
$(\scrS,\zeta) \to (M'_0 \backslash \scrS,\zeta')$,
$(M_0 \backslash \scrS,\zeta') \to (M'_0 \backslash \scrS,\zeta'')$ and
$(M'_0 \backslash \scrS,\zeta'') \to (\D_\pm^\op,m_\pm)$ are all measure class 
preserving. 
Using Proposition \ref{prop:M0'} again
we identify $M'/M_0$ with $W_0\ltimes A$ and $M'/M'_0$ with $W_0$ and we identify $\scrF$ with $M_0 \backslash \scrS$ and $\D^{2,\op}$ with $M'_0 \backslash \scrS$ equivariantly, thus consider $\zeta'$ and $\zeta''$ as measures
on $\scrF$ and $\D^{2,\op}$ correspondingly.

We are left to show that
$\mu\times \mu$ is equivalent to $\zeta'$.
Fixing $x\in X^{(0)}$ and
recalling that the measures $m_\pm$ are equivalent to 
$\mu_{x,+}\times \mu_{x,-}$ on $\D_\pm^\op$,
we will argue to show that the fibered measures with respect to their
disintegrations of $\mu_x\times \mu_x$ and $\zeta'$ over $(\D_\pm^\op, \mu_+\times \mu_-)$ are equivalent.
As $\zeta'$ corresponds to the measure $m_\pm\times \eta''$  on $\D_\pm^\op \times M''_0\backslash M$,
in view of Proposition~\ref{prop:desintmu}, it suffices to prove that for almost every $(u,v)\in \D_\pm^\op$, the measure $\mu_{u,v;x} \times \mu_{u,v;x}$ is quasi invariant by the action of $M$ on the fiber over $(u,v)$. We will argue to show that.
Fix $(u,v)\in\D_\pm^\op$. 
By Theorem~\ref{thm:S} the action of $M$ on the boundary $\partial T$ of the model tree $T$ is $3$-transitive. The stabilizer of a pair of distinct points $(\xi,\xi')\in \partial T\times \partial T$ is conjugated to $M'_0=\Fix_M(\partial\Sigma)$. The measure $\eta''$ on 
$M'_0\backslash M$ could be identified with an $M$-invariant measure $m^T$ on the set  of distinct pairs of points in the boundary of $T$. This measure, as can be checked with similar computations as in \S\ref{subsec:d2}, is given by the formula $\d m^T(\xi,\xi')=q^{-\beta^T_x(\xi,\xi')}\d \mu^T_x(\xi)\d \mu^T_x(\xi')$ (where $\beta^T_x(\xi,\xi')=(\xi,\xi')_x$ is the usual Gromov product on $T$, and $\mu^T_x$ is the measure defined before Lemma~\ref{desint m}). Hence it is in the same class as $\mu^T_x\times \mu^T_x$, which is identified (when viewed via $p_{u,v}$ as a measure on $\Ch(\D)$) with $\mu_{u,v;x}\times \mu_{u,v;x}$.
\end{proof}


\begin{remark}
It is possible, by a more precise calculation, in the spirit of Lemma \ref{desint m}, to prove that the measure $\zeta''$ on $\D^{2,\op}$ which was constructed in the proof above is actually the same measure as $m$ discussed in \S\ref{subsec:d2}.
This also follows from Theorem~\ref{CartanFlow} below, as the ergodicity of the $\Gamma$-action on $\D^{2,\op}$ implies that  there exists at most one $\Gamma$-invariant Radon measure in the class $\mu\times\mu$ (up to a scaling factor), so using Corollary~\ref{cor:minvariant}, we get that $\zeta''$ coincides with $m$ up to a factor, which is easily checked to be 1.
It is also possible to show that the measure $\zeta'$ on $\scrF$ given in Theorem~\ref{thm:compactiblemeasure} is $W\ltimes A$-invariant.
For this one can use Lemma~\ref{lem:MeasHaarFiber} in the setting $U=\scrF$ and $H=W\ltimes A$ and take $\nu$ to be the measure corresponding to $m$ under the identification $H\backslash U\simeq \D^{2,\op}$.
One then needs to argue that the measure thus obtained on $\scrF$ (which is $W\ltimes A$-invariant by Lemma~\ref{lem:MeasHaarFiber}) coincides with the measure $\zeta'$. This is not hard, but we will not carry the details here, as the $W_0\ltimes A$-invariance of $\zeta'$, which is already guarenteed suffices for our purposes in the sequel.
\end{remark}

Note that $\Gamma$ acts properly on the spaces $\scrS$ and $\scrF$, as it acts properly on their factor, $X^{(0)}$.
Since $\scrF/\Gamma$ and $\scrS/\Gamma$ are compact in view of Lemma~\ref{lem:FmodGammaCompact} and Theorem~\ref{thm:S}, we obtain the following.

\begin{corollary} \label{cor:CompatibleMeas}
There exists an $M$-invariant probability measure on $\scrS/\Gamma$ and a $W_0\ltimes A$-invariant probability measure on $\scrF/\Gamma$ such that
$\scrS/\Gamma\to\scrF/\Gamma$ is measure preserving.
\end{corollary}

In the above we explained how to construct invariant Radon measures on some spaces.
Focusing now on measure classes merely, we observe that $\scrS$ is a source in Diagram~(\ref{eq:broadpic})
(for every other space in the diagram there is a map from $\scrS$ to the space), thus pushing the measure class from $\scrS$ to each other space and using Theorem~\ref{thm:compactiblemeasure} and the commutativity of the diagram we obtain the following corollary.

\begin{corollary} \label{cor:measuss}
Each of the spaces in Diagram~(\ref{eq:broadpic})  is endowed with a measure class such that all maps in the diagram are measure class preserving. Moreover, all measures considered on the various spaces so far on a given space from Diagram~(\ref{eq:broadpic}) belong to the same measure class.
\end{corollary}

\section{Ergodicity} \label{sec:ergodicity}

In this section we establish the ergodicity properties required by the proof of the main result.
For the sake of readability, we often omit the measures from our notations.
All measures are understood to be those defined in \S\ref{sec:meas}, see Corollary~\ref{cor:measuss}.
Similarly, unspecified maps are those discussed in \S\ref{subsec:broadpic}, see Diagram~(\ref{eq:broadpic}).

\subsection{Ergodicity the Cartan flow}

This subsection is devoted to the proof of the following theorem.

\begin{theorem}\label{CartanFlow}
The following two actions are ergodic:
\begin{itemize}
\item The $\Gamma$-action   on $\Delta^{2,op}$.
\item The $A$-action  on $\scrF/\Gamma$.
\end{itemize}
\end{theorem}

It is possible to use \cite{Cartwright} to prove that $(\Ch(\Delta),\mu)$ is the Poisson boundary of $\Gamma$ (endowed with a suitable probability measure), at least in the case when $\Gamma$ acts transitively on vertices in $X$. Using the fact that the diagonal action of a group is always ergodic on the square of its Poisson boundary, we deduce the ergodicity of the $\Gamma$ action on $\Ch(\Delta)\times \Ch(\Delta)$.
We choose to present another proof, which is based a version of the classical ``Hopf argument'' presented in Appendix~\ref{app-hopf}, and is valid without assuming that $\Gamma$ is vertex-transitive on $X$. 
In order to facilitate the readability of the proof, we summarize some of the constructions and discussions given in the previous section in the following diagram,
which should be compared with the diagram~(\ref{diad:hopf}) appearing in Theorem~\ref{hopf}.

\begin{align} \label{diag:cartan}
\xymatrix{
& \scrF \ar[dl]_\phi\ar[dr]^\psi& &\\
\scrF/\Gamma & & \scrF/A\simeq  \Delta^{2,\op} \ar[dl]_{p_1}\ar[dr]^{p_{2}}& \\
& \Ch(\D) &  & \Ch(\D)
}
\end{align}

\begin{proof}[Proof of Theorem~\ref{CartanFlow}]
Set $T=A$. Note that this is an abelian, hence amenable group.
Set also $\Lambda=\Gamma$ and $Z=\scrF$.
Note that $Z$ is a locally compact topological space endowed with  commuting actions of $T$ and $\Lambda$
and both actions (separately) $\Lambda \curvearrowright Z$ and $T\curvearrowright Z$ are proper.
Set $V=Z/\Lambda$ and $Y=Z/T$
and let $\phi \colon Z\to X$ and $\psi \colon Z\to Y$ be the quotient maps.
By Lemma~\ref{lem:FmodGammaCompact}, the space $V$ is compact.
By Lemma~\ref{lem:op=F}, we have $Y\simeq \Delta^{2,\op}$.
We let $W_1=W_2=\Ch(\Delta)$ and let $\pi_1 = p_1$ and $\pi_2 = p_2$ be the first and second coordinate projections $\Delta^{2,\op}\to \Ch(\Delta)$. 

Now, let $f\in \pi_1^*\big(L^\infty(\Ch(\Delta))\big)^\Gamma \cap \pi_2^*\big(L^\infty(\Ch(\Delta))\big)^\Gamma$.
Then by definition there exists $f_1,f_2\in L^\infty(\Ch(\Delta))$ such that for $m$-almost every $(x,y)$ we have $(x,y)=f_1(x)$ and $f(x,y)=f_2(y)$. Using the fact that $m$ is equivalent to $\mu\times \mu$, and Fubini, this means that there is a $y_0$ such that for almost every $x$ we have $f_1(x)=f_2(y_0)$. Hence $f(x,y)=f_2(y_0)$ for almost all $(x,y)$, and $f$ is essentially constant.

Recall that $Z=\scrF$ is the set of simplicial embeddings from $\Sigma$ to $X$ where a basis for the topology on $\scrF$ is given by the sets of all maps which are the same in restriction to a ball around $0$.
Let $t_1$ be a vector of $T=A$ contained in the interior of the positive Weyl chamber.
It is of course central, by the commutativity of $A$.
Fix $z,z'\in \scrF$ such that $\pi_2\psi(z)=\pi_2\psi(z')$. The latter equality means that the   maps $z$ and $z'$ coincide at infinity of the   negative Weyl chamber. Thus $z(\Sigma)$ and $z'(\Sigma)$ have a common negative subsector, and there exist $t,t'\in T$ such that $z|_{Q^--t}=z'|_{Q^--t'}$.
Equivalently we write $tz|_{Q^-}=t'z'|_{Q^-}$.
Since $t_1$ is in the interior of the positive Weyl chamber, we see that for every $\sigma\in\Sigma$, we have 
$\sigma-nt_1\in Q^-$ for all $n$ sufficiently large.
It follows that for every compact neighborhood $B$ of $0\in\Sigma$, for every $n$ large enough,
$B-nt_1\in Q^-$
and therefore, $tz|_{B-nt_1}=t'z'|_{B-nt_1}$.
Equivalently we write $t_1^ntz|_{B}=t_1^nt'z'|_{B}$.
It   follows that the two sequences $(t_1^nt\phi(z))_{n\in\mathbb{N}}$ and $(t_1^nt'\phi(z'))_{n\in\mathbb{N}}$ are proximal in $\scrF$, hence also in $V$.
We note also that the measure class  on $Z$ constructed in Section~\ref{subsec:measgen} satisfies the assumptions of Theorem~\ref{hopf} by Corollaries~\ref{cor:CompatibleMeas} and~\ref{cor:measuss}.
It follows from Theorem~\ref{hopf} that the $A$-action on $\scrF/\Gamma$ and the $\Gamma$-action on $\Delta^{2,\op}$
are ergodic.
\end{proof}

The following corollary is straightforward, as $\D_\pm^\op$ is a $\Gamma$-equivariant factor of $\D^{2,\op}$.

\begin{corollary} \label{cor:Dpmerg}
$\Gamma$ acts ergodically on $\D_\pm^\op$.
\end{corollary}


\subsection{Ergodicity of the singular Cartan flow}

Our goal in this subsection is to prove the following.


\begin{theorem}\label{ergodic}
The $\Gamma$-action  on $\scrS/S$ is ergodic.
\end{theorem}

In the sequel we denote by $Y$ the space $\scrS/S$. It is equipped with the measure class pushed from $\scrS$. Moreover  $M\times \Gamma$ acts on $Y$ and preserves   this measure class.
Note that the continuous surjective map $\scrS\to\D_\pm^\op$ factors through $Y$ by Lemma~\ref{lem:scrs}.
For $(u,v)\in\Delta_\pm^\op$ we let $Y(u,v)$ be the associated fiber under the map $Y\to \D_\pm^\op$.

Note that the space $Y(u,v)$ is in bijection with $M/S$. More precisely, it is the $M/S$-orbit of some (or any) $\phi\in Y(u,v)$, by Lemma~\ref{lem:scrs}.


Recall from Definition~\ref{def:bard} the definitions of the maps $\pi_+ \colon  \widetilde\scrS\to \overline\Delta_+$ and $\pi_- \colon \widetilde\scrS\to \overline\Delta_-$.
Restricting these maps to $\scrS$ and noting that they factor through $Y$,
we abuse the notation and consider these maps as $\pi_+ \colon Y\to\overline\Delta_+$ and $\pi_- \colon Y\to \overline\Delta_-$.
We endow these spaces with the measure classes pushed under $\pi_+,\pi_-$.

Let $p_+$ (resp. $p_-$) be the second coordinate of $\pi_+$ (resp. $\pi_-$).
In other words, $p_+(\phi)$ is the map $f \colon T\to T_{\phi_+}$ which is defined as $f(x)=\phi(x,t)$ for $t$ large enough.
The space $Y_u$ (resp. $Y_v$) is defined as $p_-(Y(u,v))$ (resp. $p_+(Y(u,v))$). It is a set of maps from $T$ to $T_u$, equipped with an action of $M/S$ (by precomposition).
By definition, we have equivariant bijections $p_-:Y(u,v) \to Y_u$ and $p_+:Y(u,v)\to Y_v$. By composition we have a bijection $Y_u\to Y_v$ for every $(u,v)\in\D_\pm^\op$.

Fix $\phi\in Y(u,v)$, and let $O_\phi \colon M/S \to Y(u,v)$ be the orbit map $a\mapsto a.\phi$. We define $\lambda(u,v)=(O_\phi)_*(\Haar(M/S))$. Similarly,  for $\phi'\in Y_u$, we define $\lambda_u$ as the image of the Haar measure on $M/S$ by the map $O_{\phi'}:a\mapsto a.\phi'$.

\begin{lemma}\label{lambda invariant}
The measures $\lambda(u,v)$ and $\lambda_u$ are independent on the choice of $\phi$ and $\phi'$. Furthermore, we have $(p_-)_*\lambda(u,v)=\lambda_u$.
\end{lemma}

\begin{proof}
Let us argue first that $\lambda(u,v)$ is independent of the choice of $\phi\in Y(u,v)$.
Let $\phi,\phi_1$ be elements of $Y(u,v)$.
We know that there exists an $a\in M/S$ such that $\phi_1=a.\phi$.
So, for any $m\in M/S$, we have $m\phi_1=ma\phi$. Let $R_a \colon M/S \to M/S$ be the right translation by $a$.
This means that $O_{\phi_1}=O_\phi \circ R_a$.
By Theorem~\ref{thm:S}, the quotient $M/S$ of $M$ by the discrete normal subgroup $S$ is unimodular, so that $(O_\phi)_*(\Haar(M/S))=(O_{\phi_1})_*(\Haar(M/S))$. The argument is similar for $\lambda_u$.

Now let us prove that $(p_-)_*\lambda(u,v)=\lambda_u$. In view of the previous argument, we can choose $\phi'=p_-(\phi)$.  It follows that $p_-=O_{\phi'}\circ O_{\phi}^{-1}$ is measure preserving.
\end{proof}

Given $n\geq 1$ and a sequence $(u_1,u_2,\dots,u_n)$   of    consecutively opposite vertices, with $u_1\in\Delta_-$, we define $Y(u_1,u_2,\dots,u_n)$ as the set of all sequence  $(\phi_1,\dots,\phi_{n-1})$ such that $\phi_1\in Y(u_1,u_2)$, and $\phi_i$ is the unique element of $Y(u_i,u_{i+1})$ such that $\phi_i\sim_\eps \phi_{i+1}$, where $\eps=+$ if $i$ is odd and $\eps=-$ if $i$ is even. Remark that the map  
$$Y(u_1,u_2,\dots,u_n) \to Y(u_1, u_2) : (\phi_1,\dots,\phi_{n-1}) \mapsto \phi_1$$ 
is bijective, so that $Y(u_1,u_2,\dots,u_n) $ is naturally identified with $Y(u_1,u_2)$. Similarly, the space $Y(u_1,u_2,\dots,u_n) $ is naturally identified with $Y(u_{n-1},u_n)$

In the sequel we will have to consider some full measure subsets of various spaces.  We fix a subset $D\subset Y$. Let $D(u,v)=Y(u,v)\cap D$. More generally, given a sequence 
 $(u_1,u_2,\dots,u_n)$   of consecutively opposite vertices, with $u_1\in\Delta_-$, we define 
 $$D(u_1,\dots,u_n)=\{ (\phi_1,\dots,\phi_{n-1})\in Y(u_1,\dots,u_n)\mid \; \forall i,\;\phi_i \in D \}$$
 
\begin{definition}
A pair $(u,v)$ of opposite vertices is \textbf{$D$-full} if the set $D\cap Y(u,v)$ is of full $\lambda_{u,v}$-measure in $Y(u,v)$.
A sequence $(u_1,u_2,\dots,u_n)$ of consecutively opposite vertices is \textbf{$D$-full} if the set $D(u_1,u_2,\dots,u_n)$ is of full $\lambda(u_1,u_2)$-measure in the set $Y(u_1,u_2,\dots,u_n)\simeq Y(u_1,u_2)$.
\end{definition}

\begin{lemma}\label{compAfull}
If $(u_1,\dots, u_n,v)$ is $D$-full, and $(v,v_1,\dots,v_m)$ is also $D$-full, then the sequence $(u_1,\dots,u_n,v,v_1,\dots,v_m)$ is $D$-full.
\end{lemma}

\begin{proof}
An element of $D(u_1,\dots,u_n,v)$ can be described by a sequence $\phi_1,\dots,\phi_{n-1},\phi_n$ as above, where in particular $\phi_n$ goes from $u_n$ to $v$. Since we assume $(u_1,\dots,u_n,v)$ to be $D$-full, the set of maps $\phi_n$ that can be the last element of such a sequence is of full measure in $Y(u_n,v)$. So (by Lemma \ref{lambda invariant}) its projection to $Y_v$ is also of full measure.
Similarly, the set $D(v,v_1,\dots,v_m)$ is described by sequences $\phi'_0,\phi'_1,\dots,\phi'_{m-1}$, and the set of possible $\phi_0$ is of full measure in $Y(v,v_1)$. So its projection is again of full measure in $Y_v$.

So we have two full measure sets of $Y_v$. It follows that their intersection is again of full measure in $Y_v$. If an element is in this intersection, then there are $\phi_n$ and $\phi'_0$ as above which are positively equivalent. Hence the sequence $(\phi_1,\dots,\phi_n,\phi'_0,\dots,\phi'_m)$ will define an element of $D(u_1,\dots,v,v_1,\dots,v_m)$. It follows that the sequence $(u_1,\dots,v_m)$ is $D$-full.
\end{proof}

Let $n\geq 2$. In what follows, we define $\mu_{x,-,n}$ as the measure  $\mu_{x,-}\otimes \mu_{x,+}\otimes\mu_{x,-}\otimes\dots$, where there are $n$ factors. Similarly we define $\mu_{x,+,n}$ as the measure  $\mu_{x,+}\otimes \mu_{x,-}\otimes\dots$ with $n$ factors. Finally, we define $\mu_{x,n}=\mu_{x,+,n}+\mu_{x,-,n}$, and denote $\mu_n$ the class of $\mu_{x,n}$. This measure is supported on the sets of sequences $(u_1,u_2,\dots,u_n)$, where $u_i$ is opposite to $u_{i+1}$.

\begin{lemma}\label{aeAfull}
Assume that $D$ has full measure in $Y$, and let $n\geq 2$. Then $\mu_n$-almost every sequence $(u_1,\dots,u_n)$ is $D$-full.
\end{lemma}

\begin{proof}
Recall that the set $Y$ is identified to $\Delta_\pm^\op \times M/S$, and the measure is in the same class as the product measure. Using Fubini, the fact that $D$ is of full measure means that for almost all $u,v$, the set of elements in $D(u,v)$ is of full measure in $M/S$. This proves the Lemma for $n=2$.

Using that fact, a straightforward induction using Lemma \ref{compAfull} shows that  almost every sequence is $D$-full.
\end{proof}

\begin{definition}
Let $u,v\in\Delta_\pm^\op$. We define the subgroup $M_D(u, v)$ of $M$ as the group generated by all automorphisms of the panel tree $T_u$ of the 
form $[\gamma:u;v_1=\gamma v;v_2;\dots;v_n,u]$, where $\gamma \in \Gamma$, the successive pairs $(\gamma u, \gamma v)$, $(v_1, v_2)$, \dots, $(v_n, u)$ are opposive, and the sequence $(\gamma u,v_1,\dots,v_n,u)$ is $D$-full. 
\end{definition}

We observe that, when $D=Y$, the group $M_D(u,v)$ is isomorphic to the group $M'_\phi$ described in Lemma~\ref{lem:M'}, where $\phi \in \scrS$ is such that $\phi_- = u$, and $\phi_+=v$.
By Lemma~\ref{lem:simeq}, the group $M/S$ is isomorphic to the closure of the group $M'_\phi$. Since for $D\subset Y$ we have $M_D(u,v)\subset M_Y(u,v)$ we get that for every $D$ and for every $(u,v)$, the group $M_D(u,v)$  is a subgroup of $M/S$.

\begin{lemma}\label{MAdense}
Assume that $D$ has full measure in $Y$. Then, for almost every $(u,v)\in\D_\pm^\op$, the group $M_D(u,v)$ is dense in $M/S$.
\end{lemma}

\begin{proof}

Using Lemma \ref{aeAfull} and Fubini, we deduce that for any fixed odd $n$,  for almost every $(u,v)$, the set of those $(v_2,\dots,v_n)$ such that $(u,v=v_1,v_2,\dots,v_n)$ is $D$-full, is of full measure. Similarly, the set of $v_n\in\Delta_+$ such that $(v_n,u)$ is $D$-full is of full measure.  
 A countable intersection of full measure sets is again of full measure, so that for almost all $(u,v)$, for all $n$, the set of those $(v_2,\dots,v_n)$ such that $(u,v,v_2\dots,v_n,u)$ is $D$-full is of full measure.
 
 A similar argument proves that, for every $\gamma\in\Gamma$, for almost every $(u,v)$,  the sequence $(\gamma u,v_1=\gamma v,v_2,\dots,v_n,u)$ is $D$-full. Taking the intersection over all $n\in\NN$ and $\gamma\in\Gamma$, we get a full measure subset of $(u,v)\in\D_\pm^\op$ such that for every $\gamma$ and $n$, for almost every $(v_2,\dots,v_n)$, the sequence $(\gamma u,v_1=\gamma v,v_2,\dots,v_n,u)$ is $D$-full.

So let us fix some $(u,v)$ in this full measure set, and prove that $M_D(u,v)$ is dense in $M/S$. For any fixed $n$, the set of $(v_2,\dots,v_n)$ such that $(u,v,\dots,v_n,u)$ is $D$-full is of full measure. The measure $\mu_n$ has full support, so this set is dense in $\Delta_+\times\Delta_-\times\dots\times\Delta_+$. By Proposition~\ref{prop:persp continuous}, it follows that the set of projectivities of length $n+2$ which are in $M_D(u,v)$ is dense in all projectivities of length $n+2$ starting by $[u;v]$ and ending at $u$. Applying the same argument with the sequence $(\gamma u,\gamma v,v_2,\dots,v_n,u)$ we see that the closure of $M_D(u,v)$ contains every automorphism of the form $[\gamma : u;\gamma v,v_2,\dots,v_n]$. 

 Hence, for almost every $(u,v)$, the closure of $M_D(u,v)$ contains $M'_\phi$ by Lemma~\ref{lem:M'}. Since $M/S$ is isomorphic to the closure of $M'_\phi$ by Lemma~\ref{lem:simeq},  
it follows that $M_D(u,v)$ is dense in $M/S$, for almost every $(u,v)$. 
\end{proof}

Finally we are ready to prove the main theorem of this section.
In order to facilitate the readability of the proof, we summarize the relevant constructions and discussions given in the previous section in the following diagram,
which should be compared with the diagram~(\ref{diad:hopf}) appearing in Theorem~\ref{hopf}.

\begin{align} \label{diag:singcartan}
\xymatrix{
& \scrS \ar[dl]_\phi\ar[dr]^\psi& &\\
\scrS/\Gamma & & \scrS/S   \ar[dl]_{\pi_+}\ar[dr]^{\pi_{-}}& \\
&  \overline \Delta_+ &  & \overline \Delta_-
}
\end{align}

\begin{proof}[Proof of Theorem \ref{ergodic}]
We want to apply Theorem~\ref{hopf}, putting $\Lambda = \Gamma$, $T= S$, $Z=\scrS$, $V=\scrS/\Gamma$, $Y=\scrS/S$ and $W_+=\overline \Delta_+$, $W_-=\overline \Delta_-$ (with the maps $\pi_-$ and $\pi_+$ as described above).

Let us start by checking the proximality condition.
Let $z,z'\in\scrS$ which have the same image in $W_+$ (the argument is similar for $W_-$).
It follows from the definition of the map to $W_+$ that, up to replacing $z$ by some translate $sz$ with $s\in S$, we can  assume that $z$ and $z'$ are positively equivalent (see Definition \ref{def:equivalent}).
In other words, for every $x \in VT$, there exists $t_x$ such that $z(x, t)=z'(x, t)$ for every $t>t_x$. Let $B_k$ be the $k$-ball centered at the origin $(o, 0)$ in the model wall tree $\Upsilon \subset T \times \RR$. Therefore we may find $t_0$ depending on $k$ such that $z(x, t)=z'(x, t)$ for all $x$ with $(x, 0) \in B_k$, and all $t > t_0$. 

Using the  multiplicative notation for elements of $S$ and denoting by $s$ the positive generator of $S$, we infer that  $s^\alpha z(x, t)= s^\beta z'(x,t )$ for   $\alpha = \beta = t_0$ and all $x$ with $(x, 0) \in B_k$, and all $t > 0$. 
Therefore,   for all sufficiently large $m$ (depending on the value of  $k$), we have $s^m s^\alpha z|_{B_k}=s^ms^\beta z'|_{B_k}$.
This proves that the sequences $(s^ms^\alpha z)_{m\in\NN}$ and $(s^ms^\beta z')_{m\in\NN}$ are proximal in $\scrS$.

In view of Corollary~\ref{cor:measuss}, the only condition left to check in order to apply Theorem~\ref{hopf} and conclude the proof is that every function class in $L^\infty(Y)$ which is $\Gamma$-invariant and is a pullback of functions on $W_+$ and $W_-$ is essentially constant.
We will do it by showing that every such function class is in a pull back under the factor map $Y\to \D_\pm^\op$ of a $\Gamma$-invariant function class on $\D_\pm^\op$.
Indeed, by Corollary~\ref{cor:Dpmerg}, $\Gamma$ acts ergodically on the latter space, hence every such function class must be constant.

We now fix a function class $f$ which is $\Gamma$-invariant and is a pullback of functions on $W_+$ and $W_-$
and argue to show that $f$ is constant on a.e. fiber of $Y\to \D_\pm^\op$,
that is for almost every $(u,v)\in\D_\pm^\op$, $f$ is essentially constant on $Y(u,v)$.
Note that by Fubini Theorem $f$ is well defined for a.e. such $(u,v)$ and denote its restriction to $Y(u,v)$ by $f_{u,v}$.
As before, we identify $Y(u,v)$ with $M/S$ and view $f_{u,v}$ as a function on $M/S$, a.e. defined with respect to the Haar measure.

We let $D$ be the subset of $Y$ on which the two pullbacks of $f_+$ and $f_-$ agree.
Then $D$ is a $\Gamma$-invariant full measure subset of $Y$.
We view $f$ as an actual function defined on $D$.
Since $f$ is the pullback of $f_+$ and $f_-$, we see that if $\phi\sim_+\phi'$ and $\phi,\phi'\in D$, then $f(\phi)=f(\phi')$.
Let  $(u,v_1,\dots,v_n,u)$ be a $D$-full sequence, and $m=[u,v_1,\dots,v_n,u]$ the associated projectivity.
Then the set of all $\phi\in D(u,v)$ such that $\phi\circ m\in D(u,v)$ is of full measure.
If $\phi$ is an element of this set, then $f_{u,v}(m\phi)=f_{u,v}(\phi)$.
So $f_{u,v}$ is invariant by $m$.
Similarly, using also the $\Gamma$-invariance of $D$, if $\gamma\in\Gamma$ and $(\gamma u,v_1,\dots,v_n,u)$ is $D$-full, then $f_{u,v}$ is essentially $[\gamma:u;v_1;\dots;v_n;u]$-invariant.
It follows that $f_{u,v}$ is $M_D(u,v)$ invariant.
For almost every $(u,v)$, we know by Lemma \ref{MAdense}  that $M_D(u,v)$ is dense in $M/S$.

Now, $f_{u,v}$, is a function in $L^{\infty}(Y(u,v))\simeq L^{\infty}(M/S)$ which is invariant by $M_D(u,v)$, and $M_D(u,v)$ is dense in $M/S$. Let us prove that $f_{u,v}$ is actually $M$-invariant. Let $g\in M$, and let $(g_n)$ be a sequence of elements of $M_D(u,v)$ converging to $g$. Let $\varphi$ be a continuous function on $M/S$. Then
\begin{align*}
 \int_{M/S} f_{u,v}(x)\varphi(x) dx &= \int_{M/S} (g_n.f_{u,v})(x)\varphi(x) dx
 = \int_{M/S} f_{u,v}(x)\varphi(g_n x) dx\\
 &\xrightarrow[n\to +\infty]{} \int_{M/S}f_{u,v}(x)\varphi(gx) dx=\int_{M/S} (g.f_{u,v})(x)\varphi(x) dx
\end{align*}
where all the integrals are taken with respect to the Haar measure on $M/S$. Since this is valid for every continuous $\varphi$, it follows that $f_{u,v}=g.f_{u,v}$.

It follows that $f_{u,v}$ is invariant by $M/S$, hence constant.
\end{proof}



\section{Non-linearity of $\Gamma$} \label{sec:nonlinear}

In this section we finally prove Theorem~\ref{thm:nonlinear-local}.
We will do so in \S\ref{subsec:pfofmain}.
Our proof relies on the Bader--Furman  theory of algebraic representations that is recalled in \S\ref{sub:BF}.

\subsection{Algebraic Representations}\label{sub:BF}

In this section we review Bader--Furman's theory of gates and algebraic representations, about  which the  main reference is \cite{BF-Superrigid}.

Let $\Gamma$ be a countable group  and  $k$ be a local field. We fix a representation $\rho \colon \Gamma\to \bfG(k)$, where $\bfG$ is a $k$-simple $k$-algebraic group, and assume that its image is Zariski-dense.

Let $Y$ be a standard Borel space, equipped with a measure $\mu$. We assume that $\Gamma$ has a measure-class preserving action  on $Y$. Furthermore, a standing assumption is that the $\Gamma$-action is moreover ergodic.

\begin{definition}
An \textbf{algebraic representation} of $Y$ is a pair $(\bfV,\phi)$, where
\begin{itemize}
\item $\bf V$ is a $k$-algebraic variety endowed with a $k$-algebraic action of $\bfG$,
\item $\phi\colon Y\to {\bf V}(k)$ is a measurable map such that $\phi(\gamma y)=\rho(\gamma) y$ for every $\gamma\in \Gamma$ and a.e. $y\in Y$.
\end{itemize}

A \textbf{morphism} from a representation $({\bf U},\phi_U)$ to $({\bf V},\phi_V)$ is a $k$-algebraic $\bf G$-equivariant map $\psi\colon {\bf U}\to {\bf V}$ such that $\phi_U\circ \psi=\phi_V$.
\end{definition}

We sometimes abuse  notation and speak of the algebraic representation $\bfV$ instead of $(\bfV,\phi)$.

\begin{theorem} \label{thm:locallyclosed}
Let $\bf V$ be a $k$-algebraic variety endowed with a $k$-algebraic action of $\bfG$.
Then the orbits of the $\bfG(k)$-action on $\bfV(k)$ are locally closed for the $k$-topology on $\bfV(k)$.
\end{theorem}

\begin{proof}
This is proved in \cite[Proposition 3.3.1]{Zimmerbook} under the assumption that the characteristic of $k$ is 0.
The general case is proved \cite{BernsteinZelevinski}. For a complete discussion see \cite[Proposition 2.2]{BDL-tameness}.
\end{proof}

Let $({\bf V},\phi)$ be an algebraic representation of $Y$.
It follows from Theorem~\ref{thm:locallyclosed} and the ergodicity assumption that the pushforward of $\mu$ is supported on a unique $\bfG$-orbit in $\bf V$, hence on a variety which is a quotient of $\bfG$.

Now consider the set of all representations of $Y$ into quotients of $\bf G$. More precisely, consider the set of all algebraic subgroups ${\bf H}<{\bf G}$ such that there exists an algebraic representation $({\bf G}/{\bf H},\phi)$. By Noetherianity, there exists a minimal element in this set of subgroups.

The above argument is the first step in the proof of the following theorem, which can be found in \cite[Theorem 5.3]{BF-Superrigid}.

\begin{theorem}\label{thm:gate}
There exists an algebraic subgroup $\bfH<\bfG$, and an algebraic representation $\phi\colon Y\to \bfG/\bfH$, such that for every algebraic representation $\bf V$ of $Y$, there exists a morphism of algebraic representations ${\bf G}/{\bf H}\to \bf V$.
\end{theorem}

\begin{definition}
The algebraic representation $({\bf G}/{\bf H},\phi)$ is called the \textbf{gate} of $Y$.
\end{definition}

The morphism in Theorem \ref{thm:gate} is in fact uniquely defined. From this, we deduce the following.

\begin{proposition}\label{Yoneda}
Let $({\bf G}/{\bf H},\phi)$ be the gate of $Y$. Let $M$ be a Polish group with a Borel action on $Y$ which preserves the measure class of $\mu$ and commutes with the action of $\Gamma$. There exists a continuous homomorphism $M\to N_{\bf G}({\bf H})/{\bf H}(k)$ which turns the gate map $Y\to {\bf G}/{\bf H}$ into a $M$-equivariant map.
\end{proposition}

\begin{proof}
The existence of the homomorphism $\psi\colon M\to N_{\bf G}({\bf H})/{\bf H}(k)$ is proved in \cite[Theorem 6.1]{BF-Superrigid}. The only thing left to check is its continuity.

To simplify the notations we let $V={\bf G}/{\bf H}(k)$, $L=N_{\bf G}({\bf H})/{\bf H}(k)$
and $U=\LL(Y,V)$, that is $U$ is the set of classes of measurable functions from $Y$ to $V$, identified up to a.e. equality, with the topology of convergence in measure.
We endow $U$ with the action of $L$ by post-composition, and the action of $M$  by precomposition.
By the fact that $L$ acts freely on $V$, we get that $L$ acts freely on $U$ as well.
Using Theorem~\ref{thm:locallyclosed}, \cite[Proposition~3.3.1]{Zimmerbook} gives that the $L$-action on $U$ has locally closed orbits\footnote{Actually, in 
\cite[Proposition~3.3.1]{Zimmerbook} it is assumed that $k$ is a local field of zero characteristic, as it relies on \cite[Theorem~3.1.3]{Zimmerbook}, but upon replacing \cite[Theorem~3.1.3]{Zimmerbook} with  Theorem~\ref{thm:locallyclosed} the proof applies verbatim here as well.}.
It follows that the $L$-orbit map $L\to U$, $l\mapsto l\circ \phi$ is a homeomorphism onto its image, $L\phi$. 
We let $\alpha:L\phi\to L$ be its inverse.

By the fact that the map $M\times Y \to V$, $(m,y)\mapsto \phi(my)$ is a.e. defined and measurable, we get that the associated map $\beta:M\to U$, $m\mapsto \phi\circ m$ is a.e. defined and measurable, see  \cite[Chapter VII, Lemma 1.3]{Margulis}.
By the $M$-equivariance of $\phi$, $\phi\circ m=\psi(m)\circ \phi$ and we conclude that $\psi$ agrees a.e. with $\alpha\circ \beta$ which is a.e. defined and measurable.
It follows that $\psi$ is measurable. 
By \cite[Lemma 2.1]{Rosendal} we conclude that $\psi$ is a continuous homomorphism.
\end{proof}

\begin{proposition} \label{prop:gatemorph}
Assume $Y\to Z$ is a factor map.
Let $\phi_Y \colon Y\to \mathbf{G}/\mathbf{H}_Y(k)$ and $\phi_Z \colon Z\to \mathbf{G}/\mathbf{H}_Z(k)$ be the corresponding gate maps.
Then there exists a canonical $k$-defined algebraic $\mathbf{G}$-map $\psi \colon \mathbf{G}/\mathbf{H}_Y\to \mathbf{G}/\mathbf{H}_Z$ such that $\phi_Z=\psi\circ\phi_Y$.
\end{proposition}

\begin{proof}
The existence of $\psi$ follows by considering the composition $Y\to Z\to \mathbf{G}/\mathbf{H}_Z(k)$ as a representation of $Y$ and the fact that $\phi_Y$ is an initial object.
%
\end{proof}


\subsection{Proof of Theorem~\ref{thm:nonlinear-local}} \label{subsec:pfofmain}

In this section we fix a  local field $k$, a connected adjoint $k$-simple $k$-algebraic group $\mathbf{G}$ and a group homomorphism $\Gamma\to \mathbf{G}(k)$ with an unbounded and Zariski dense image.
We will argue to show that $X$ is Bruhat--Tits.

We will consider the $\Gamma$-spaces $\Ch(\Delta)$, $\Delta^{2,\op}$ and $\scrS/S$.	
These spaces are endowed with natural measure classes as discussed in \S\ref{sec:meas}
(see in particular \S\ref{subsec:d} and Corollary~\ref{cor:measuss}).
We will fix these measure classes, but make them implicit in our notation.
By Theorem~\ref{CartanFlow} and Theorem~\ref{ergodic} all these spaces are $\Gamma$-ergodic.
Thus we may consider their algebraic gates as discussed in \S\ref{sub:BF}.
We consider the map $\scrS/S\to \D^{2,\op}$ discussed in~\S\ref{subsec:broadpic} and its $N\to\{1,\tau\}$-equivariancy discussed in Corollary~\ref{cor:flip}.
We also consider the projections $p_1,p_2\colon \D^{2,\op}\to \Ch(\D)$.
We invoke the discussion of \S\ref{sub:BF} and summarize our situation in the following diagram:

\begin{align} \label{eq:main}
\begin{split}
\xymatrix{N<M^+\ar@{>>}[]!<-4ex,0ex>;[d]!<0ex,0ex>  \ar@/^2.5pc/@{.>}[rrrrr]^{} &\ar@(ul,dl) & \scrS/S \ar[r] \ar[d] & {\bf G}/{\bf H}_ {\scrS/S}(k) \ar@{>>}[d] &
\ar@(ur,dr) & N_{\bf G}({\bf H}_{ \scrS/S})/{\bf H}_{ \scrS/S}(k) \\
\{1,\tau\} \ar@/^2.5pc/@{.>}[rrrrr]^{} &\ar@(ul,dl) & \Delta^{2,op} \ar[r]\ar@< 2.5pt>[d]_{p_1~} \ar@< -2.5pt>[d]^{~~p_2} & {\bf G}/{\bf H}_{\Delta^{2,op}}(k) \ar@{>>}@< 2.5pt>[d] \ar@{>>}@< -2.5pt>[d] &  \ar@(ur,dr) &  N_{\bf G}({\bf H}_{\Delta^{2,op}})/{\bf H}_{\Delta^{2,op}}(k)   \\
& & \Ch(\Delta) \ar[r] & {\bf G}/{\bf H}_{\Ch(\Delta)}(k) &  & }
\end{split}
\end{align}

Here the spaces on the right-hand side are the gates of the corresponding spaces on the left-hand side and the maps between them are the ones guaranteed by Proposition~\ref{prop:gatemorph}. The dotted arrows are continuous group homomorphisms as guaranteed by Proposition~\ref{Yoneda}.

We break the proof of Theorem~\ref{thm:nonlinear-local} into the following four steps:
\begin{enumerate}
\item ${\bf H}_{\Ch(\Delta)},{\bf H}_{\Delta^{2,op}}$ and ${\bf H}_{\scrS/S}$ are proper subgroups of $\bf G$.
\item The homomorphism $\{1,\tau\}\to N_{\bf G}({\bf H}_{\Delta^{2,op}})/{\bf H}_{\Delta^{2,op}}(k)$ is non-trivial.
\item The homomorphism $M^+\to  N_{\bf G}({\bf H}_{\scrS/S})/{\bf H}_{\scrS/S}(k)$ is non-trivial.
\item $X$ is Bruhat--Tits.
\end{enumerate}

\begin{proof}[Proof of Step 1:]
Clearly it is enough to show that ${\bf H}_{\Ch(\Delta)} \neq\bfG$,
and by the universal property of the gate, it suffices to prove that there exists an algebraic representation ${\Ch(\Delta)} \to \bfG/\bfH(k)$ with $\bfH\neq\bfG$.
By the main result of \cite{Lecureux_amen} (see also the earlier reference \cite[\S4.2]{RobertsonSteger} for the special case of a lattice $\Gamma$ acting sharply transitively on the vertices of $X$), the action of $\Gamma$ on ${\Ch(\Delta)} $ is topologically amenable. It follows that the $\Gamma$-action on $({\Ch(\Delta)} ,\mu)$ is also amenable in the sense of Zimmer \cite[Proposition 3.3.5]{AnantharamanRenault}.
By \cite[Theorem 6.1]{BDL-tameness}, either there exists an algebraic representation of ${\Ch(\Delta)} \to \bfG/\bfH(k)$, with $\bfH$ a proper subgroup of $\bfG$, or there exists a $\Gamma$-equivariant map to a metric space $Z$ with a proper isometric action of $\bfG(k)$.
The second option is excluded by \cite[Theorem~1.1]{Glasner-Weiss}, as $({\Ch(\Delta)} \times{\Ch(\Delta)} ,\mu\times\mu)$ is $\Gamma$-ergodic.
\end{proof}

\begin{proof}[Proof of Step 2:]
If the image of $\tau$ in $N_{\bf G}({\bf H}_{\Delta^{2,op}})/{\bf H}_{\Delta^{2,op}}(k)$ were trivial then, by the commutativity of the diagram~(\ref{eq:main}),
the two maps ${\Ch(\Delta)} \times {\Ch(\Delta)}  \overset{p_1}\to {\Ch(\Delta)} \to {\bf G}/{\bf H}_{\Ch(\Delta)} (k)$ and ${\Ch(\Delta)} \times {\Ch(\Delta)}  \overset{p_2}\to {\Ch(\Delta)} \to {\bf G}/{\bf H}_{\Ch(\Delta)} (k)$ would coincide.
This  happens if and only if these maps are essentially constant, as the two factors ${\Ch(\Delta)} \times {\Ch(\Delta)}  \overset{p_1}\to {\Ch(\Delta)} $ and ${\Ch(\Delta)} \times {\Ch(\Delta)}  \overset{p_2}\to {\Ch(\Delta)} $ are obviously independent.
But the Zariski closure of the support of the image of $ {\Ch(\Delta)} $ in ${\bf G}/{\bf H}_{\Ch(\Delta)} (k)$ is $\Gamma$-invariant, hence $\bfG$-invariant, as $\rho(\Gamma)$ is Zariski-dense, hence it is all of the homogeneous space $\bfG/\bfH_{\Ch(\Delta)} $. This contradicts the fact that $\bfH_{\Ch(\Delta)} $ is a proper subgroup of $\bfG$ established in the previous step.
\end{proof}

\begin{proof}[Proof of Step 3:]
By Theorem~\ref{thm:S} the group $M^+/S$ is topologically simple.
By Proposition~\ref{Yoneda} the homomorphism $M^+\to N_{\bf G}({\bf H}_{\scrS/S})/{\bf H}_{\scrS/S}(k)$ is continuous, and by construction its kernel contains $S$ since it factors through the $M^+$-action on $\scrS/S$.
Thus, in order to prove its non-triviality it is enough to exhibit a single element $n\in M^+$ having a non-trivial image.
By Corollary~\ref{cor:flip} the homomorphism $N\to \{1,\tau\}$ is surjective. We let $n\in N<M^+$ be a preimage of $\tau$
and observe that if its image in $N_{\bf G}({\bf H}_{\scrS/S})/{\bf H}_{\scrS/S}(k)$ were trivial then   the image of $\tau$ in
$N_{\bf G}({\bf H}_{\Delta^{2,op}})/{\bf H}_{\Delta^{2,op}}(k)$ would also be trivial by the commutativity of the diagram~(\ref{eq:main}).
\end{proof}

\begin{proof}[Proof of Step 4:]
The representation $M^+/S\to  N_{\bf G}({\bf H}_{\scrS/S})/{\bf H}_{\scrS/S}(k)$
is a non-trivial continuous linear representation of $M^+/S$ over a local field.
We conclude by Theorem~\ref{thm:S}
that $X$ is Bruhat--Tits.
\end{proof}

\subsection{Proof of Theorem~\ref{nonlinear}}

Let $X$ be a locally finite $\widetilde A_2$-building and $\Gamma$ be a lattice in $\Aut(X)$.
Assume that $X$ is not isomorphic to the building associated to $\mathrm{PGL}_3(D)$, with $D$ a finite dimensional division algebra over a local field.
By \cite{pansu}, the group $\Gamma$ satisfies Kazhdan's property (T).
Fix a finite index subgroup $\Gamma_1<\Gamma$ which acts by type-preserving automorphisms.
Let $\Gamma_0<\Gamma_1$ be any further finite index subgroup.
Observe that $\Gamma_0$ is also a cocompact lattice in $\Aut(X)$.
Applying Theorem~\ref{thm:nonlinear-local} to $\Gamma_0$
we obtain that it has no unbounded Zariski dense representations into $\mathbf{G}(k)$, where $k$ is a local field and $\bf G$ is a connected adjoint $k$-simple $k$-algebraic group.
Therefore, the required conclusion follows from Corollary~\ref{cor:ringlinear}.\qed

\section{Galois lattices are non-linear}\label{app:Galois}

\subsection{The automorphism group of a simple algebraic group}

Given a locally compact group $G$, we denote by $\Aut(G)$ the group of bi-continuous automorphisms of $G$, endowed with the \textbf{Braconnier topology} (also sometimes called  the \textbf{Birkhoff topology}). In general $\Aut(G)$ need not be locally compact, but it will be so for the specific groups $G$ that we shall consider. Given a locally compact field $k$, we also denote by $\Aut(k)$ the group of bi-continuous automorphisms of $k$. If $k$ is not isomorphic to $\mathbf C$, any automorphism of $k$ is automatically continuous: this is clear if $k$ is discrete or if $k \simeq \mathbf R$ since $\Aut(\mathbf R)$ is trivial, and follows otherwise from the uniqueness of the discrete valuation with respect to which $k$ is complete. Moreover, if $k$ is non-discrete, then $\Aut(k)$ is infinite if and only if $k$ has positive characteristic.

Let now $\mathbf G$ be a  $k$-simple adjoint algebraic group over a local field $k$, with $k$-rank~$\geq 1$. The goal of this section is to show that if a lattice $\Gamma$ in the locally compact group $\Aut(\mathbf G(k))$ has infinite image in $\Out(\mathbf G(k))$ and if $\mathbf G$ has $k$-rank~$\geq 2$, then every finite-dimensional linear representation of $\Gamma$ has finite image. 
We first recall some useful facts on semi-simple groups and their automorphisms.

The first point to keep in mind is that given a group $\mathbf G$ over $k$ as above, there is a semi-simple absolutely   simple adjoint group $\mathbf H$ defined over a finite separable extension $k'$ of $k$, so that $\mathbf G$ is the Weil restriction of $\mathbf H$ from $k'$ to $k$ (see \cite{BoTi65}*{\S6.21(ii)}). In particular the locally compact groups $\mathbf H(k')$ and  $\mathbf G(k)$ are isomorphic, so there is no loss  of generality in assuming that $\mathbf G$ is absolutely   simple. After that reduction, a description of the automorphism group of $\mathbf G(k)$ is provided by the work of Borel and Tits \cite{BoTi}. We record the following points for the sake of future references.

\begin{proposition}\label{prop:Aut}
	Let $\mathbf G$ be an absolutely simple adjoint algebraic group over a local field  $k$, with $k$-rank~$\geq 1$. Let $X$ denote the Bruhat--Tits building of $\mathbf G(k)$, and $\mathbf G(k)^+$ denote  the subgroup generated by the unipotent radicals of $k$-parabolic subgroups. The following assertions hold.
	\begin{enumerate}[(i)]
		\item We have $\Aut(\mathbf G(k)) \cong \Aut(\mathbf G(k)^+)$. Moreover, as an abstract group $\Aut(\mathbf G(k))$ is an  extension of the group $\Aut(\mathbf G)(k)$ of algebraic automorphisms by the   group $\Aut_{\mathbf G}(k)$ consisting of those automorphisms $\alpha$ of $k$ such that the groups $\mathbf G$ and ${}^\alpha \mathbf G$ are $k$-isomorphic.
		
		\item The normal subgroup $\Inn(\mathbf G(k)) \cong  \mathbf G(k)$ is   of finite index in $\Aut(\mathbf G)(k)$. 
		
		\item If $\mathbf G$ is $k$-split, then $\Aut(\mathbf G(k)) \cong \mathbf G(k) \rtimes (A \times \Aut(k))$, where $A$ is the automorphism group of the Dynkin diagram of $\mathbf G$.
		
		\item $\Aut(\mathbf G(k))$ is a compactly generated  locally compact second countable group, the subgroup $\Inn(\mathbf G(k)) \cong  \mathbf G(k)$ is closed and cocompact. In particular $\Aut(\mathbf G)(k)$ is closed and cocompact in $\Aut(\mathbf G(k))$. The intersection of all   non-trivial closed normal subgroups of $\Aut(\mathbf G(k))$ is non-trivial and coincides with the natural image of $\mathbf G(k)^+$ in $\Aut(\mathbf G(k))$ (which is injective).
		
		\item  $\Aut(\mathbf G(k))$ acts continuously, properly and faithfully on $X$. This action induces an isomorphism of topological groups $\Aut(\mathbf G(k)) \cong N_{\Aut(X)}(\mathbf G(k))$, where $\mathbf G(k)$ is viewed as a closed subgroup of $\Aut(X)$.
		
		\item The map $\Aut(\mathbf G(k)) \to \Aut(X)$ is surjective if and only if the $k$-rank of $\mathbf G$ is~$\geq 2$.
	\end{enumerate}
\end{proposition}

\begin{proof}
	(i) We recall from \cite{BoTi}*{Prop.~6.14} that $\mathbf G(k)^+$ is a closed subgroup.  Moreover every abstract automorphism of $\mathbf G(k)$ (resp. $\mathbf G(k)^+$) is continuous by \cite{BoTi}*{Prop.~9.8}. The identifications
	$$\Aut(\mathbf G(k)) \cong \Aut(\mathbf G(k)^+)
	\hspace{1cm} \text{and} \hspace{1cm}
	\Aut(\mathbf G(k)) / \Aut(\mathbf G)(k) \cong \Aut_{\mathbf G}(k)$$
	as abstract groups follow from \cite{BoTi}*{Cor.~8.13}. The conjugation action of $\mathbf G(k)$ induces an embedding $\mathbf G(k) \to \Aut(\mathbf G)(k)$ whose image in $\Aut(\mathbf G(k))$ is $ \Inn(\mathbf G(k))$, and is thus normal.
	
	\medskip \noindent
	(ii) The index of  $ \Inn(\mathbf G(k)) \cong \mathbf G(k)$ in $\Aut(\mathbf G)(k)$ is finite in view of \cite{Conrad}*{Th.~7.1.9}. 

	\medskip \noindent
	(iii) If $\mathbf G$ is $k$-split, then it is defined over the prime field of $k$, so that  $\Aut_{\mathbf G}(k)=  \Aut(k)$. The result then follows from (i) and \cite{Springer}*{\S16.3}.
	
	\medskip \noindent
	(iv) The natural embedding of $\mathbf G(k)$ into $\Aut(\mathbf G(k))$ is continuous by definition of the Braconnier topology. A proof that $\Aut(\mathbf G(k))$ is locally compact and second countable may be found in \cite{CapStu}*{Prop.~4.10} (formally, the latter statement applies to $\Aut(\mathbf G(k)^+)$; the latter coincides with   $\Aut(\mathbf G(k))$ by (i)). Any continuous homomorphism of $\mathbf G(k)^+$ to a locally compact group has closed image (see \cite{CapStu}*{Prop.~2.1}). Since $\mathbf G(k)/\mathbf G(k)^+$ is compact by \cite{BoTi}*{Prop.~6.14}, the same statement holds for $\mathbf G(k)$. Hence $\Inn(\mathbf G(k))$ is closed in $\Aut(\mathbf G(k))$. Since $\mathbf G(k)^+$ is compactly generated (see e.g. \cite{CapStu}*{Th.~2.2}), so is $\mathbf G(k)$, and all remaining assertions of (iii) will follow once we establish that $\Aut(\mathbf G(k))/\Inn(\mathbf G(k))$ is compact.  We postpone the proof of this to the next paragraph.
	
	\medskip \noindent
	(v) Let $p$ be the residue characteristic of $k$. The chamber-stabilizers in $X$ are precisely the normalizers of the maximal pro-$p$ subgroups of $\mathbf G(k)$ (see the discussion in \cite{TitsCorvallis}*{\S3.7}). Moreover two distinct chambers are adjacent if and only if the index of the intersection of their stabilizers is minimal~$>1$. Those observations imply that $\Aut(\mathbf G(k))$ naturally acts by automorphisms on the building $X$. By definition of the Braconnier topology, the normalizer of any compact open subgroup of $\mathbf G(k)$ is open in $\Aut(\mathbf G(k))$. This implies that every vertex-stabilizer for the $\Aut(\mathbf G(k))$-action on $X$ is open. Thus the latter action is continuous. To see that it is faithful, observe that the kernel $K$ of this action normalizes every parahoric subgroup, and also every parabolic subgroup of $\mathbf G(k)$. In particular it also normalizes the unipotent radical $U$ of every minimal parabolic $P$. Since the latter unipotent radical $U$ acts freely and transitively on the parabolic subgroups opposite $P$, it follows that $K$ acts trivially on $U$. This holds for any $P$ and $U$, so that $K$ acts trivially on $\mathbf G(k)^+$. Therefore $K$ is trivial by (i).  Thus we have obtained a continuous injective homomorphism $\Aut(\mathbf G(k)) \to \Aut(X)$. It takes values in the normalizer $N_{\Aut(X)}(\mathbf G(k))$ (where $\mathbf G(k)$ has been identified with its image in $\Aut(X)$). The conjugation action of $N_{\Aut(X)}(\mathbf G(k))$ on $\mathbf G(k)$ also yields a continuous map $N_{\Aut(X)}(\mathbf G(k)) \to \Aut(\mathbf G(k))$ which is injective since any automorphism of $X$ commuting with $\mathbf G(k)$ has constant displacement function (because $\mathbf G(k)$ is transitive on the chambers), and must thus be trivial. It follows that the map $\Aut(\mathbf G(k)) \to N_{\Aut(X)}(\mathbf G(k)) $ is bijective. Since $\mathbf G(k)$ is closed in $\Aut(X)$, it follows that $N_{\Aut(X)}(\mathbf G(k))$ is closed. Therefore the continuous isomorphism $\Aut(\mathbf G(k)) \to N_{\Aut(X)}(\mathbf G(k)) $  is an isomorphism of topological groups by Baire, since $\Aut(\mathbf G(k))$ is second countable. This implies that the $\Aut(\mathbf G(k))$-action on $X$ is proper. Since $\mathbf G(k)$ acts cocompactly on $X$, it is a cocompact subgroup of $\Aut(X)$, and hence of $N_{\Aut(X)}(\mathbf G(k)) $. The compactness of $\Aut(\mathbf G(k))/\Inn(\mathbf G(k)) \cong  N_{\Aut(X)}(\mathbf G(k))/\mathbf G(k)$ follows.

	\medskip \noindent
	(vi) If $G$ has $k$-rank~$1$, then $X$ is a semi-regular tree. The group $\Aut(X)$ is then much larger than  $\Aut(\mathbf G(k))$. Indeed the group $\Aut(X)^+$ of type-preserving automorphisms of $X$ is simple (\cite{TitsTrees}) and of index~$\leq 2$ in $\Aut(X)$. If the map $\Aut(\mathbf G(k)) \to \Aut(X)$ were surjective, it would be an isomorphism of topological groups (by Baire) so that it would induce an isomorphism of $\mathbf G(k)^+$ onto $\Aut(X)^+$ by (iii). That $\Aut(X)^+$  is non-linear follows from \cite[Corollary~R]{CRW}. This is a contradiction.
	
	If $G$ has $k$-rank $\geq 2$, then  $\Aut(X)$ coincides with the full automorphism group of the spherical building at infinity $\Delta$, see \cite[Main Result 2]{Struyve} (alternatively, in the case at hand, this fact can be deduced by combining  \cite[Th.~26.37 and Prop.~26.40]{WeissBook} with  \cite{TitsLN}*{Cor.~5.9}).  The fact that every element of $\Aut(\Delta)$ comes from an element of $\Aut(\mathbf G(k))$ is due to J.~Tits: we refer to  \cite{TitsLN}*{Cor.~5.9} for a general result ensuring that this holds over arbitrary fields, with the notable exceptions of perfect fields of characteristic~$2$ or $3$. Those exceptions do not occur under the hypotheses of the Proposition, since local fields of positive characteristic are all imperfect. 
\end{proof}

\begin{remark}
	By Proposition~\ref{prop:Aut}(iii), the short exact sequence
	$$1 \to \Aut(\mathbf G)(k) \to \Aut(\mathbf G(k)) \to \Aut_{\mathbf G}(k) \to 1$$
	splits as soon as $\mathbf G$ is $k$-split. For non-split groups, the splitting of the short exact sequence may fail: explicit examples have  been found by T.~Stulemeijer \cite[Chapter~3]{Stulemeijer_thesis}. {
We believe that  the index of   $\Aut_{\mathbf G}(k)$ in $\Aut(k)$ is always finite. To see that, it suffices to show that the natural permutation action of $\Aut(k)$ on the set of $k$-forms of $\mathbf G$ is continuous. Notice that this claim trivially holds when  $k$ has characteristic~$0$, since in that case $\Aut(k)$ is finite. }
\end{remark}

\subsection{Galois lattices}

The goal of this section is to prove Theorem~\ref{thm:nonlinearGalois} from the introduction. We reproduce its statement here for the reader's convenience.

\begin{theorem}\label{thm:Galois}
	Let $\mathbf G$ be a $k$-simple algebraic group defined over a local field $k$, of $k$-rank~$\geq 2$. Let $\Gamma$ be a lattice in $\Aut(\mathbf G(k))$.
	
	If $\Gamma$ has infinite image in $\Out(\mathbf G(k))$, then for any commutative unital ring $R$ and any   $n\geq 1$, any homomorphism $\Gamma \to \GL_n(R)$ has finite image.
\end{theorem}

Since $\Out(\mathbf G(k))$ is virtually a subgroup of the Galois group $\Aut(k)$ by Proposition~\ref{prop:Aut}, a lattice as in the theorem with infinite image in $\Out(\mathbf G(k))$ will be called a \textbf{Galois lattice}. The existence of irreducible Galois lattices in higher rank is an open problem (which is explicitly mentioned as Problem 19 in the problem list \cite{FarbMozesThomas}; see also \cite[Annexe A, Problem 1]{CorHab}). In the rank~$1$ case, cocompact Galois lattices indeed exist, as shown by the following.

\begin{proposition}\label{prop:RankOneGaloisLattice}
	Let $k$ be a local field of characteristic~$p >0$, with residue field of order $q$. Given any $q$-generated residually $p$-group $\Lambda$, there is a uniform lattice in $G = \SL_2(k) \rtimes \Aut(k)$  whose projection to $\Aut(k) \cong G/ \SL_2(k)$ is isomorphic to $\Lambda$.
\end{proposition}

\begin{proof}
	Let $T$ be the Bruhat--Tits tree of $G$, which is regular of degree $q+1$. We denote by $\Aut(T)^+$ the index $2$ subgroup of $\Aut(T)$ consisting of the type-preserving automorphisms. Notice that $G$ acts by type-preserving automorphisms on $T$.
	
	Fix a base edge $\{v, w\} \in E(T)$. Let $e_1, \dots, e_q \in E(T)$ (resp. $f_1, \dots, f_q \in E(T))$ denote the edges different from $\{v, w\}$ and containing $v$ (resp. $w$).  We shall use the following.
	
	\begin{claim*}
		Let $g_1, \dots, g_q \in \Aut(T)^+$ be such that $g_i(e_i) = f_i$. Then $\Gamma = \langle g_1, \dots, g_q \rangle$ is a discrete free subgroup of $\Aut(T)^+$ acting cocompactly on $T$. In particular $\Gamma $ is a uniform lattice in $\Aut(T)$.
	\end{claim*}
	
	Since $g_i$ is type-preserving and maps the edge $e_i$ to an edge at distance~$1$, it must be a translation. The fact that $\Gamma$ is free and discrete then follows from a  ping-pong argument. For the cocompactness, consider an arbitrary edge $e \in E(T)$. If $e$ does not contain $v$ or $ w$, then the geodesic joining $e$ to the base edge $\{v, w\}$ contains $e_i$ or $f_i$ for some $i$, and the edge $g_i(e)$ or $g_i^{-1}(e)$ is strictly closer to $\{v, w\}$ than $e$. This proves by induction that the set of edges containing $v$ or $w$ contains a fundamental domain for the $\Gamma$-action. Thus the  claim holds.  
	
	We apply the claim as follows. Since $\SL_2(k)$ is edge-transitive on $T$, we may find $\gamma_i \in \SL_2(k)$ mapping $e_i$ to $f_i$. We next recall that $\Aut(k)$ is  compact and acts continuously on $T$ (by Proposition~\ref{prop:Aut}), so that the subgroup of $\Aut(k)$ fixing all edges containing $v$ or $w$ is open of finite index. We next invoke the fact, due to Rachel Camina \cite{Camina}, that $\Aut(k)$, and hence also every open subgroup of $\Aut(k)$, contains a copy of every countably based pro-$p$ group. We may thus choose $s_1, \dots, s_q \in \Aut(k)$ so that $s_i$ fixes every edge of $T$ containing $v$ or $w$ and that $\langle s_1, \dots, s_q\rangle$ is isomorphic to $\Lambda$.  Finally we put $g_i = \gamma_i s_i$. The claim ensures that $\Gamma = \langle g_1, \dots, g_q \rangle$ is a uniform lattice in  $G = \SL_2(k) \rtimes \Aut(k)$. By construction, its image in $\Aut(k)$ is $\langle s_1, \dots, s_q\rangle \cong \Lambda$. (A related argument which also provides a lifting of free groups to Galois groups appears in   \cite{TitsLN}*{\S11.14}.)
\end{proof}

Using Proposition~\ref{prop:Aut}, we can already show how Theorem~\ref{thm:Galois} implies Corollary~\ref{cor:Galois} from the introduction. 

\begin{proof}[Proof of Corollary~\ref{cor:Galois}]
	The sufficiency of the condition is clear. For the reverse implication, notice first by \cite{BoTi}*{Cor.~6.5 and Rem.~6.15} that the quotient $\mathbf G(k)^+/Z$ is isomorphic to $\mathbf G_{\mathrm{ad}}(k)^+$, where $\mathbf G_{\mathrm{ad}} = \mathbf G/Z(\mathbf G)$ is adjoint. Hence,  by Proposition~\ref{prop:Aut} we have $\Aut(X) \cong \Aut(\mathbf G(k)^+/Z) \cong  \Aut(\mathbf G_{\mathrm{ad}}(k))$. This allows us to invoke Theorem~\ref{thm:Galois}. The desired conclusion follows.
\end{proof}

\subsection{Proof of non-linearity}

This section is devoted to the proof of Theorem~\ref{thm:Galois}.

The following notation will be used throughout. The characteristic subgroup of $\mathbf G(k)$ generated by the unipotent radicals of $k$-parabolic subgroups is denoted by  $\mathbf G(k)^+$. The image of $\mathbf G(k)^+$ in $\Aut(\mathbf G(k))$, which is a closed normal subgroup, is denoted by $G^+$.  We also denote by $X$ the Bruhat--Tits building of $\mathbf G(k)$ and by $\Delta$ the spherical building at infinity of $X$. We may identify  $\Aut(\mathbf G(k))$ with $\Aut(X)$ in view of Proposition~\ref{prop:Aut}.

Upon replacing   $\Gamma$ by a finite index subgroup, we can assume  that the $\Gamma$-action on $X$ is  type-preserving and that $\Gamma/\Gamma \cap \Inn(\mathbf G(k))$ maps injectively to $\Aut_{\mathbf G}(k)$ under the map $\Aut(\mathbf G(k)) \to \Aut_{\mathbf G}(k)$ afforded by Proposition~\ref{prop:Aut} (see (i) and (ii) from the latter proposition).
Now we set
$$G = \overline{G^+ \Gamma}.$$
The quotient $G/G^+$ is compact by  Proposition~\ref{prop:Aut}.

We assume  that   $\Gamma$ has a linear representation over a commutative unital ring with infinite image. Our goal is thus to show that  $G/G^+$ is finite.

Let $A \subset X$ be an apartment and $\Fix_G(A)$ be the pointwise stabilizer of $A$ in $G$. Thus $\Fix_G(A)$ is compact, and we have
$$G = G^+ \Fix_G(A)$$
since $G^+$ is transitive on the set of apartments of $X$ and since $\Stab_{G^+}(A)$ maps onto the affine Weyl group $\Stab_G(A)/\Fix_G(A)$. Fix a chamber $C$ at infinity of $A$ and let $\pi_1, \dots, \pi_m$ be the panels of $C$. Recall that the stabilizer $\Stab_G(\pi_j)$ acts on a locally finite tree $\mathcal T_j$, which is the panel tree associated with $\pi_j$ (see \cite{WeissBook}*{11.18}). The set of ends of  $\mathcal T_j$ is in canonical one-to-one correspondence with the set of chambers of $\Delta$ incident with the panel $\pi_j$ (the corresponding statement in the case of $\widetilde A_2$-buildings was recalled in Proposition~\ref{bndpaneltree}).


Let $\pi'_j$ be the panel at infinity of $A$ which is opposite $\pi_j$ and set
$$L_j = \Stab_G(\pi_j, \pi'_j)
\hspace{1cm} \text{and} \hspace{1cm}
T_j  =  \Ker \varphi_j,$$
where
$$\varphi_j \colon L_j \to \Aut(\mathcal T_j)$$
is the natural action whose existence has just been recalled. Let also $C'$ be the chamber at infinity of $A$ which is opposite $C$ and set
$$T = \Stab_G(C, C').$$

The following result relies crucially on the hypothesis that the $k$-rank of $\mathbf G$ is~$\geq 2$.

\begin{lemma}\label{lem:T_jNonCompact}
	$T_j \cap G^+$ is not compact.
\end{lemma}

\begin{proof}
	Indeed $T_j \cap G^+$ contains a split torus of codimension~$1$ in $G^+$, and is thus non-compact by hypothesis on the rank.
\end{proof}

\begin{lemma}\label{lem:ergodic}
	The $\Gamma$-actions on $G/T_j$ and on $G/T$ are ergodic.
\end{lemma}

\begin{proof}
	By Moore's inversion trick \cite{Zimmerbook}*{Cor.~2.2.3}, it suffices to show that $T_j$ acts ergodically on $G/\Gamma$, hence that $T_j \cap G^+$ acts ergodically on $G/\Gamma$. In view of Lemma~\ref{lem:T_jNonCompact}, the Howe--Moore property  for $G^+$ (see \cite{HoweMoore}) reduces the problem to showing that $G^+$ acts ergodically on $G/\Gamma$. Now we apply Moore's inversion again, and see that the conclusion will follow if we show that $\Gamma$ is ergodic on $G/G^+$. By definition $\Gamma$ has dense image in the compact quotient $G/G^+$; the desired ergodicity is now a consequence of the Lebesgue Density Theorem (see \cite{AbertNikolov}*{Lem.~7} for a detailed proof).
	
	Since $T_j \leq T$, the space $G/T$ is a $\Gamma$-quotient of $G/T_j$, and is thus ergodic as well.
\end{proof}

\begin{lemma}\label{lem:ClosedImage}
	The  group $M_j = L_j/T_j$ maps continuously, injectively onto  a closed subgroup of $\Aut(\mathcal T_j)$ which is $2$-transitive on $\partial \mathcal T_j$.  In particular $M_j$ has a smallest non-trivial closed normal subgroup $M_j^+$ which is cocompact, compactly generated, topologically simple and non-discrete.
	
	Moreover $M_j^+$ is isomorphic to the quotient of the derived group of the rank~$1$ Levi factor $L_j \cap G^+ = \Stab_{G^+}(\pi_j, \pi'_j)$ by its center.
\end{lemma}

\begin{proof}
	The fact that the homomorphism $L_j/T_j \to \Aut(\mathcal T_j)$ is continuous and injective is clear from the definition.
	We have $G = G^+ \Fix_G(A)$. Moreover $\Fix_G(A) \leq \Stab_G(\pi_j, \pi'_j) = L_j$. It follows that
	$$L_j = (L_j \cap G^+) \Fix_G(A).$$
	Since $\Fix_G(A)$ is compact, the closedness of the image of $L_j$ follows from that of $L_j \cap G^+$. Notice that $L_j \cap G^+ = \Stab_{G^+}(\pi_j, \pi'_j)$ is a Levi factor of the parabolic subgroup  $\Stab_{G^+}(\pi_j)$. By construction, the semi-simple part of that Levi factor is of relative rank~$1$, and $\mathcal T_j$ is isomorphic to its Bruhat--Tits tree. The assertions of closedness and of boundary-$2$-transitivity follow, recalling that a semi-simple group over a local field acts properly on its Bruhat--Tits building (see Proposition~\ref{prop:Aut}(v)). The remaining assertions follow from Theorem~\ref{thm:BuMo}.
\end{proof}

\begin{lemma}\label{lem:GammaFiniteAb}
	$\Gamma$ has Kazhdan's property (T).
\end{lemma}

\begin{proof}
	Indeed, the group $\mathbf G(k)$ has property~(T) by \cite{BHV}*{Th.~1.6.1}. Now, combining  \cite{BHV}*{Th.~1.7.1} and Proposition~\ref{prop:Aut}(iv)), we see that this property is inherited by the locally compact group $\Aut(X)$, and then by the lattice   $\Gamma$.
\end{proof}

In the following step, we invoke again Bader--Furman's theory outlined in Section~\ref{sub:BF}.

\begin{lemma}\label{lem:LinearLevi}
	There exists $j \in \{1, \dots, m\}$ such that  group $M_j = L_j/T_j$  has a continuous faithful linear representation over a local field.
\end{lemma}

\begin{proof}
	By Lemma~\ref{lem:GammaFiniteAb}, we may apply Corollary~\ref{cor:ringlinear}. This ensures   that the given representation of $\Gamma$ may be assumed to be a representation $\rho \colon \Gamma \to \mathbf J(K)$, where  $\mathbf J$ is a simple algebraic group over a local field $K$, and $\rho(\Gamma)$ is unbounded and Zariski-dense. We may then apply the Bader--Furman theory with the ergodic $\Gamma$-space $Y_j = G/T_j$. The group $ N_G(T_j)/T_j$ has a natural action on $Y$ which commutes with the $\Gamma$-action. In particular, the same holds for the action of $M_j = L_j/T_j$ since $T_j \leq L_j \leq  N_G(T_j)$. It then follows from Proposition~\ref{Yoneda} that $M_j$ has a continuous representation over the same local field $K$. In order to prove that this representation is faithful, it suffices by  Lemma~\ref{lem:ClosedImage} to show that its restriction to $M_j^+$ is non-trivial.
	
	To this end, we proceed as in the proof of Theorem~\ref{thm:nonlinear-local}. 
	
	In a first step, we observe that the gate of the ergodic $\Gamma$-space $Y_j = G/T_j$ from Lemma~\ref{lem:ergodic} is non-trivial. Indeed, the $\Gamma$-action on $G/P$, where $P = \Stab_G(C)$, is amenable because $P$ is amenable (see \cite{Zimmerbook}*{Prop.~4.3.2}). Moreover, it is doubly ergodic by Lemma~\ref{lem:ergodic} since $G/P \times G/P$ is isomorphic to $G/T$ as a $\Gamma$-space. It then follows from \cite[Theorem 6.1]{BDL-tameness} and \cite[Theorem~1.1]{Glasner-Weiss} that the gate $\mathbf J/\mathbf H_P$ of the $\Gamma$-space $G/P$ is non-trivial. Therefore, the gate $\mathbf J/\mathbf H_Z$ of its extensions $Z = G/T$, and the gate $\mathbf J/\mathbf H_j$ of $Y_j = G/T_j$, are non-trivial either.

	In a second step, we observe that the longest element $w_0$ of the spherical Weyl group of $N_G(T)/T$ acts as the flip on $G/P \times G/P$, and therefore has a non-trivial image under the representation $N_G(T)/T \to N_{\mathbf J}(\mathbf H_Z)/\mathbf H_Z(K)$ afforded by Proposition~\ref{Yoneda}. This can be established with a similar argument as in Step~2 of the proof of  Theorem~\ref{thm:nonlinear-local}, as follows. If the image of the flip were trivial in $N_{\mathbf J}(\mathbf H_Z)/\mathbf H_Z(K)$, then the two maps $G/P \times G/P \to G/P \to \mathbf J/\mathbf H_P$ associated to to first and second projections $G/P \times G/P \to G/P$ would coincide, so that the map $G/P \to \mathbf J/\mathbf H_P$ would have to be essentially constant. Since that image is $\rho(\Gamma)$-invariant, hence $\mathbf J$-invariant by Zariski-density, we conclude that the gate $\mathbf J/\mathbf H_P$ of $G/P$ is trivial, a contradiction. 
	
	We now conclude the proof of the lemma as follows. For each $j \in \{1, \dots, m\}$, choose an element $\tau_j \in L_j^+$ (where $L_j^+$ denotes the preimage of $M_j^+$ in $L_j$) acting as a reflection of $A$ fixing the panel $\pi_j$. Then  the natural images of $\{\tau_1, \dots, \tau_m\}$ in the spherical Weyl group $N_G(T)/T$ is a generating set of that Weyl group. Therefore, if the image of $\tau_j$ under the continuous representation $M_j = L_j/T_j  \to N_{\mathbf J}(\mathbf H_j)/\mathbf H_j(K)$ afforded by Proposition~\ref{Yoneda} were trivial for all $j$, then so would be the image of $w_0$ under the representation $N_G(T)/T \to N_{\mathbf J}(\mathbf H_Z)/\mathbf H_Z(K)$ as a consequence of Proposition~\ref{prop:gatemorph}. This would contradict the conclusion of the second step. Therefore we find some  $j\in \{1, \dots, m\}$ such that the image of $\tau_j$ under the representation $M_j    \to N_{\mathbf J}(\mathbf H_j)/\mathbf H_j(K)$ is non-trivial. In particular the image of $M_j^+$ is non-trivial. Therefore the restriction of that representation to $M_j$ is faithful in view of Lemma~\ref{lem:ClosedImage}.
\end{proof}

\begin{lemma}\label{lem:vAbelianImage:1}
	The quotient $M_j/M_j^+$ is virtually abelian, where $j \in \{1, \dots, m\}$ is the element afforded by Lemma~\ref{lem:LinearLevi}.
\end{lemma}

\begin{proof}
	This follows from Proposition~\ref{prop:aux}(iv), whose hypotheses are satisfied in view of Lemmas~\ref{lem:ClosedImage} and~\ref{lem:LinearLevi}.
\end{proof}


\begin{lemma}\label{lem:FieldAutom}
	If $h \in \Fix_G(A)$ has non-trivial image in the Galois group $\Aut_{\mathbf G}(k)$ under the map $\Aut(\mathbf G(k)) \to \Aut_{\mathbf G}(k)$, then $\varphi_i(h) \not \in M_i^+$ for all  $i \in \{1, \dots, m\}$.
\end{lemma}
\begin{proof}
	This follows from how the proof of the Borel--Tits theorem works in \cite{BoTi}*{\S 8}. Here are more details.  Since $h $  acts trivially on the apartment $A$, it normalizes a maximal $k$-split torus $\mathbf S$ of $\mathbf G(k)$ and as well as  each of the root groups associated with the $k$-roots of $\mathbf S$. In \S\S 8.2 and 8.3 of \cite{BoTi}, the unique field automorphism involved in the automorphism $h$ of $\mathbf G(k)$ is constructed by considering the induced action of $h$ on a specific $3$-dimensional $k$-split reductive subgroup $L$ of a rank~$1$ Levi subgroup $Q$ of $\mathbf G(k)$ associated to a $k$-root of $\mathbf S$. Those arguments in \cite{BoTi} imply that   the $h$-action on $L$ involves a field automorphism if and only if   the $h$-action on $\mathbf G(k)$ does. For the same reason (viewing now $h$ as an automorphism of the Levi factor $Q$), we see that the $h$-action on $L$ involves a field automorphism if and only if   the $h$-action on $Q$ does. It is then showed in  \S\S 8.2  of \cite{BoTi} that the $h$-action on the rank~$1$ Levi factors associated to all $k$-roots of $\mathbf S$ induce field automorphisms that differ from one another by a power of the Frobenius automorphism. Since local fields of positive characteristic are not perfect, the Frobenius map is not a field automorphism, so in our setting the $h$-action on each rank~$1$ Levi factor involves the same field automorphism.

	Recalling from Lemma~\ref{lem:ClosedImage} that $M_i^+$ is the quotient of a rank~$1$ Levi factor of $G^+$ divided by its center, we deduce from the discussion above that if $\varphi_i(h) \in M_i^+$, then $h$ does not involve any non-trivial field automorphism, so that its image in $\Aut_{\mathbf G}(k)$ is trivial.
\end{proof}

\begin{lemma}\label{lem:Galois}
	Let $\widetilde G = \Inn( \mathbf G(k)) \cap G$.  Then the quotient $\Fix_G(A) /  \Fix_{\widetilde G}(A)$ is virtually abelian.
\end{lemma}
\begin{proof}
	%
	Let $h, h' \in \Fix_G(A) $ such that $h^{-1} h' \not \in \widetilde G$. Recall from the beginning of the proof that $\Gamma / \Gamma \cap \widetilde G$ maps injectively to   the Galois group. Therefore, we see that $h$ and $h'$ act as distinct field automorphisms on the group $G^+$. Thus they have distinct images in $M_i/M_i^+$ for all $i$ by Lemma~\ref{lem:FieldAutom}. This means that the kernel of the composed homomorphism  $\Fix_G(A) \overset{\varphi_i }\to M_i \to M_i/M_i^+$ is contained in $\widetilde G$. In particular $\Fix_G(A)/ \widetilde G \cap \Fix_G(A)$ maps injectively to a quotient of $M_i/M_i^+$. Applying that assertion to the index $i=j$ afforded by Lemma~\ref{lem:vAbelianImage:1}, we obtain the required conclusion.
\end{proof}

\begin{lemma}\label{lem:final}
	The quotient $G/G^+$ is finite.
\end{lemma}

\begin{proof}
	The group $\widetilde G/G^+$ is virtually abelian by \cite{BoTi}*{Prop.~6.14}. Therefore $ \Fix_{\widetilde G}(A)/ \Fix_{G^+}(A)$ is virtually abelian, too. In view of Lemma~\ref{lem:Galois}, we infer that $\Fix_G(A)/ \Fix_{G^+}(A)$ is virtually metabelian. On the other hand, we have $G/G^+ = G^+ \Fix_G(A) / G^+  \cong \Fix_G(A)/  \Fix_{G^+}(A)$, so that $G/G^+$ is virtually metabelian. Since   $\Gamma$ has property (T) by Lemma~\ref{lem:GammaFiniteAb} and dense image in $G/G^+$ by construction, we conclude that $G/G^+$ is indeed finite.
\end{proof}

Lemma~\ref{lem:final} concludes the proof of Theorem~\ref{thm:Galois}.

\section{Lattices in  exotic $\widetilde A_2$-buildings of arbitrary large order}\label{app:ExoticLargeOrder}

The goal of this section is to provide an explicit construction of an infinite family of lattices in exotic $\widetilde A_2$-buildings of arbitrarily large order. This relies on a construction,  originally due to M.~Ronan \cite{Ronan} and W.~Kantor \cite{Kantor84}, relating panel-regular $\widetilde A_2$-lattices with finite difference sets. That construction was recently concretized by means of explicit group presentations by J.~Essert \cite{Essert}, see Theorem~\ref{thm:Essert} below. Using those presentations together with Theorem~\ref{nonlinear}, we shall see that the exoticity of a panel-regular lattice in a building of order $q_0$ typically implies the exoticity of many such buildings whose orders are suitable powers of $q_0$. A concrete illustration is described in Corollary~\ref{cor:ExistenceExotic}.  Actually, our discussion will  show  that many   of the $\widetilde A_2$-buildings admitting a  panel-regular lattice are exotic. 

\subsection{Difference sets}

A \textbf{difference set} in a group $A$ is a subset $D \subset A$ such that any non-trivial element $a \in A$ can be written in a unique way as a product $a = d_1 d_2^{-1}$ with $d_1, d_2 \in D$. If $A$ is finite of order $n$, the existence of a difference set $D$ implies that $n = q^2 + q +1$, where $q = |D|-1$. A difference set $D \subset A$ naturally yields a projective plane $\mathscr P(A, D)$ whose points are the elements of $A$ and whose lines are the translates of $D$ in $A$. Conversely, given a finite projective plane $\mathscr P$ and a group $A$ acting freely and transitively on the points of $\mathscr P$, one obtains a difference set $D$ in $A$ by virtue of the following classical fact.

\begin{lemma}\label{lem:DiffSet}
	Let $\mathscr P$ be a finite projective plane and $A \leq \Aut(\mathscr P)$   a group acting freely and transitively on the points of $\mathscr P$. Then $A$ acts freely and transitively on the lines of $\mathscr P$. Moreover, for any point $p$ and any line $L$, the set $D = \{a \in A \mid a(p) \in L\}$ is a difference set in $A$.
	
	Moreover, the Desarguesian plane  $\mathscr P = \mathscr P(2, \FF_q)$ of order $q$ admits a cyclic group $A$ acting freely and transitively: namely, the canonical image of the  multiplicative group $\FF_{q^3}^\times$ in the group $\PGL_3(\FF_q)$, which is isomorphic to $\FF_{q^3}^\times/\FF_q^\times$.
\end{lemma}
\begin{proof}
	See Assertion~4.2.7 and the discussion on pp.~105--106 in \cite{Dembowski}.
\end{proof}

The cyclic group from Lemma~\ref{lem:DiffSet} was first pointed out by J.~Singer \cite{Singer}; it is called a \textbf{Singer group}.

We shall be interested in pairs   of finite groups $A_0, A$ and difference sets $D_0 \subset A_0$, $D \subset A$ together with an injective homomorphism $A_0 \to A$ mapping $D_0$ into $D$. The following shows that such configurations may be constructed using Singer groups.

\begin{lemma}\label{lem:SubDiffSets}
	Let $p$ be a prime, let $q_0 $ be a positive power of $p$ and $q = q_0^e$   a positive power of $q_0$. Assume that $q_0 \not \equiv 1 \mod 3$ and that  $e \not \equiv 0 \mod 3$. Then $n_0 = q_0^2 + q_0 +1$ divides $n = q^2 + q+ 1$, and there  exist  difference sets $D_0 \subset \ZZ/n_0\ZZ$ and $D \subset \ZZ/n\ZZ$ such that $\{0, 1\} \subset D_0$ and $\{dn/n_0 \mid d \in D_0\} \subseteq D$.
	
	If $p = q_0=2$, then we may further arrange that $D_0 = \{0, 1, 3\}$.
\end{lemma}
\begin{proof}
	Write $e = 3f+i$, with $i=1$ or $2$. Since $q_0^{3f}-1$ is a multiple of $n_0$, we have $\gcd(n_0, q-1) = \gcd(n_0, q_0^{3f}(q_0^i-1))$. Since $q_0$ is a prime power, we have $\gcd(n_0, q_0^{3f}) = 1$. Moreover $\gcd(n_0, q_0+1)=1$, and the hypothesis that $q_0 \not \equiv 1 \mod 3$ implies that  $ \gcd(n_0, q_0-1)=1$. Therefore, we infer that $\gcd(n_0, q-1) = 1$. This implies that  the canonical embedding   of cyclic groups $\FF_{q_0^3}^\times \to \FF_{q^3}^\times$ descends to an injective homomorphism $\varphi \colon \FF_{q_0^3}^\times /\FF_{q_0}^\times \to \FF_{q^3}^\times/\FF_q^\times $. In particular $n_0$ divides $n$.

	The inclusion $\FF_{q_0} \subset \FF_q$ yields an inclusion of  projective planes $\iota \colon \mathscr P(2, \FF_{q_0}) \to \mathscr P(2, \FF_q)$.  Identifying the Singer groups $A =\FF_{q^3}^\times/\FF_q^\times $ and $B =  \FF_{q_0^3}^\times /\FF_{q_0}^\times$   with their natural images in  $\PGL_3(\FF_q)$ and $\PGL_3(\FF_{q_0})$, we infer that the inclusion  $\iota \colon \mathscr P(2, \FF_{q_0}) \to \mathscr P(2, \FF_q)$ is equivariant with respect to $\varphi$. Let $a$ be a generator of $A$ and $b$   a generator of $B$ such that $\varphi(b) = a^{n/n_0}$.
	
	Let now $x_0$ a point of $ \mathscr P(2, \FF_{q_0})$ and $L_0$ be a line of  $ \mathscr P(2, \FF_{q_0})$ containing $x_0$ and $b(x_0)$.  Then
	$$D_0 = \{d \in \ZZ/n_0\ZZ \mid b^d(x_0) \in L_0\}$$
	is a difference set in $\ZZ/n\ZZ$ containing $0$ and $1$ by Lemma~\ref{lem:DiffSet}.
	Let also $x = \iota(x_0)$ and $L$ be a line of $ \mathscr P(2, \FF_{q})$ containing $\iota(L_0)$.  Then
	$$D = \{d \in \ZZ/n\ZZ \mid a^d(x) \in L\}$$
	is a difference set in $\ZZ/n\ZZ$ by Lemma~\ref{lem:DiffSet}. Since $\varphi (b) =   a^{n/n_0} $, we have $\{dn/n_0 \mid d \in D_0\} \subseteq D$ as required.
	
	Assume finally that $p=q_0=2$, so that $n_0 = 7$.
	Observe that the only difference sets in $\ZZ/7\ZZ$ containing $0$ and $1$ are $\{0, 1, 3\}$ and $\{0, 1, 5\}$. Since $q^2 +q + 1$ is not divisible by $5$, the element $a^5$ is also a generator of $A$. Therefore, in case we obtain $D_0 = \{0,1, 5\}$, we redo the construction with $a$ replaced by $a^5$ and $b$ by $b^5$. This yields $D_0 = \{0,1, 3\}$.
\end{proof}

\subsection{Panel-regular lattices and their presentations}

A type-preserving automorphism group of an $\widetilde A_2$-building $X$ acting freely and transitively on the set of panels of each type is called  a \textbf{panel-regular lattice}.
M.~Ronan \cite[Prop.~3.2 and Theorem~3.3]{Ronan} and W.~Kantor \cite[Theorem~C.3.1 and Cor.~C.3.2]{Kantor84} described a tight relation between difference sets in groups of order $n = q^2 + q +1$ and panel-regular lattices in $\widetilde A_2$-buildings of order~$q$.  That relation was made more explicit by J.~Essert \cite{Essert}, who characterized panel-regular lattices  by means of an explicit and elegant presentation.  

\begin{theorem}[J. Essert]\label{thm:Essert}
	Let $q > 1$ be an integer, $n = q^2 + q +1$ and $D$ be a difference set containing~$0$  in the cyclic group $\ZZ/n\ZZ$. For all permutations $\pi_1, \pi_2 \in \Sym(D)$ fixing $0$, the group
	$$\Gamma = \langle \sigma_0, \sigma_1, \sigma_2 \mid \sigma_0^n, \sigma_1^n, \sigma_2^n,  \ \sigma_0^d \sigma_1^{\pi_1(d)} \sigma_2^{\pi_2(d)}\  \forall d \in D\rangle$$
	is a panel-regular lattice in an $\widetilde A_2$-building $X$ of order $q$.
	Moreover there is a chamber with vertices $v_0, v_1, v_2$ such that the stabilizer $\Gamma_{v_i}$ coincides with the cyclic group $\langle \sigma_i \rangle$ for $i = 0, 1, 2$.
\end{theorem}

\begin{proof}
	See \cite[Theorem~5.8]{Essert}.
	The last assertion on the geometric action of the cyclic groups $\langle \sigma_i \rangle$ follows from \cite[Corollary~5.7]{Essert}.
\end{proof}

We shall need the following basic fact on panel-regular lattices.

\begin{lemma}\label{lem:TorsionElt}
	Let $X$ be an $\widetilde A_2$-building of order $q$ and $\Gamma \leq \Aut(X)$ be a panel-regular lattice. Let $v, v'$ be vertices of $X$  that have the same type and are distance~$2$ apart in the $1$-skeleton. Then the set $\{\gamma \in \Gamma \mid \gamma(v) = v'\}$ contains exactly~$2$ elements of finite order, and $q^2+q-1$ elements of infinite order.
	
	In particular, if $\Gamma$ is the panel-regular lattice  given by the presentation from Theorem~\ref{thm:Essert}, then for any  non-zero $d \in D$ and any $e \in  \ZZ/n\ZZ$, the element $\gamma = \sigma_0^d \sigma_1^e $ is of finite order if and only if  $e \in \{ 0,\pi_1(d)\}$.
\end{lemma}
\begin{proof}
	Since the $\Gamma$-action on $X$ is type-preserving, an element $\gamma \in \Gamma$ of finite order must fix a vertex $x$. Since the $\langle \gamma \rangle$-action on the link of $x$ is free on panels of each type, we infer that for each $y \in X$ different from $x$, we have the bound $\angle_x(y, \gamma y)  \geq 2\pi/3$ on the Alexandrov angle (in the sense of CAT($0$) geometry) as soon as $\gamma$ is non-trivial. This implies that if $\gamma$ maps a vertex $v$ to a vertex $v'$ at distance~$2$ in the $1$-skeleton, the only possible fixed vertex of $\gamma$ is one of the two vertices, say $v_1$ and $v_2$, that have distance~$1$ to both $v$ and $v'$. An element $\gamma$ fixing $v_i$ and mapping $v$ to $v'$ maps the edge $[v_i, v]$ to $[v_i, v']$. Since edges correspond to panels, the panel-regularity of $\Gamma$ implies that there is exactly one element of $\Gamma$  fixing $v_i$ and mapping $v$ to $v'$.
	
	Let us now consider  $\Gamma = \langle \sigma_0, \sigma_1, \sigma_2 \mid \sigma_0^n, \sigma_1^n, \sigma_2^n,  \ \sigma_0^d \sigma_1^{\pi_1(d)} \sigma_2^{\pi_2(d)}\  \forall d \in D\rangle$ as in Theorem~\ref{thm:Essert}, where $n = q^2 + q +1$. Let $v_0$ and $v_1$ be the adjacent vertices fixed by $\sigma_0$ and $\sigma_1$ respectively, afforded  by Theorem~\ref{thm:Essert}. Let  $d, e \in \ZZ/n\ZZ$. The element $\gamma = \sigma_0^d \sigma_1^e  \in \Gamma$ maps $v_1$ to a vertex $v'$ adjacent  to $v_0$, which is  at distance~$2$ from $v_0$ in the $1$-skeleton provided $d \neq 0$. Since we have $\{g \in \Gamma \mid g(v_1) = v'\} = \gamma \Gamma_{v_1} = \{\sigma_0^d\sigma_1^j \mid j \in \ZZ/n\ZZ\}$, we deduce from the first assertion that if $d\neq 0$, there are exactly two  elements  $j\in \ZZ/n\ZZ$ such that $\sigma_0^d\sigma_1^j$ is of finite order. If in  addition $d \in D$, those   must be $j=0$ or $j = \pi_1(d)$ since $ \sigma_0^d \sigma_1^{\pi_1(d)}  = \sigma_2^{-\pi_2(d)}$.
\end{proof}

\subsection{Maps between panel-regular lattices}

A groundbreaking example of an exotic $\widetilde A_2$-building with a vertex-transitive lattice and non-Desarguesian residues has recently been constructed in \cite{Radu}. That building is of order $9$, and contains exotic sub-buildings of order~$3$. This feature suggests to study the existence of maps between $\widetilde A_2$-lattices in buildings of different orders. In the context of panel-regular lattices, we record the following.

\begin{proposition}\label{prop:LatticeMorphism}
	Let $q, n, D$ and $q_0, n_0, D_0$ be two triples as in Theorem~\ref{thm:Essert}. Let also   $\tau_1, \tau_2 \in \Sym(D_0)$ and $\pi_1, \pi_2 \in \Sym(D)$ be permutations all fixing~$0$, and consider the groups
	$$\Gamma_0 = \langle s_0, s_1, s_2 \mid s_0^{n_0}, s_1^{n_0}, s^{n_0},  \  s_0^d s_1^{\tau_1(d)} s_2^{\tau_2(d)}\  \forall d \in D_0\rangle$$
	and
	$$\Gamma = \langle \sigma_0, \sigma_1, \sigma_2 \mid \sigma_0^n, \sigma_1^n, \sigma_2^n,  \ \sigma_0^d \sigma_1^{\pi_1(d)} \sigma_2^{\pi_2(d)}\  \forall d \in D\rangle,$$
	and the associated buildings $X_0$ and $X$ afforded by Theorem~\ref{thm:Essert}.
	
	Assume that the following conditions hold.
	\begin{itemize}
		\item $n_0$ divides $n$.
		
		\item $\{dn/n_0 \mid d \in D_0\} \subseteq D$.
		
		\item $\pi_i(dn/n_0) = \tau_i(d)n/n_0$ for all $d \in D_0$ and $i \in \{1, 2\}$.
	\end{itemize}
	Then the assignments $s_i \mapsto \sigma_i^{n/n_0}$ for $i=0, 1, 2$ extend to a homomorphism $\Gamma_0 \to \Gamma$ with infinite image.
\end{proposition}
\begin{proof}
	Set $\varphi(s_i) = \sigma_i^{n/n_0}$ for $i=0, 1, 2$. Since $n_0$ divides $n$, we have $\varphi(s_i)^{n_0} = 1$ for all $i$. Moreover, given $d \in D_0$, we have
	$$\begin{array}{rcl}
	\varphi(s_0)^d \varphi(s_1)^{\tau_1(d)} \varphi(s_2)^{\tau_2(d)} & =  &
	\sigma_0^{dn/n_0} \sigma_1^{\tau_1(d)n/n_0} \sigma_2^{\tau_2(d)n/n_0}\\
	& = & \sigma_0^{dn/n_0} \sigma_1^{\pi_1(dn/n_0)} \sigma_2^{\pi_2(dn/n_0)}\\
	& = & 1
	\end{array}$$
	since $dn/n_0$ is an element of $D$ by hypothesis. Thus the elements $\varphi(s_i)$   satisfy the defining relations of $\Gamma_0$, so that $\varphi $ indeed extends to a homomorphism $\varphi \colon \Gamma_0 \to \Gamma$. To prove  that $\varphi(\Gamma_0)$ is infinite, we pick any non-zero $d \in D_0$. Then $s_0^d$ and $s_0^d s_1^{\tau_1(d)}$ are both of finite order. By virtue of Lemma~\ref{lem:TorsionElt}, we deduce that for any  $e \in \ZZ/n_0\ZZ \setminus \{0, \tau_1(d)\}$,   the element
	$\varphi(s_0^d s_1^{e}) = \sigma_0^{dn/n_0} \sigma_1^{en/n_0}$ is of infinite order.
\end{proof}

\begin{remark}\label{rem:EmbeddingBuildings}
	As mentioned in the introduction,  an unpublished work of Yehuda Shalom and Tim Steger ensures that every proper normal subgroup of a lattice in an $\widetilde{A}_2$-building is of finite index. In particular, the homomorphism $\Gamma_0 \to \Gamma$ constructed in Proposition~\ref{prop:LatticeMorphism} must be injective. Using \cite[Corollary~5.7]{Essert}, this implies that the embedding $\Gamma_0 \to \Gamma$ is accompanied by a $\Gamma_0$-equivariant isometric embedding $X_0 \to X$. This noteworthy geometric feature will however not be needed in the subsequent discussion.
\end{remark}

Invoking Theorem~\ref{nonlinear}, we deduce the following exoticity criterion for $X$.

\begin{proposition}\label{prop:ExoticityCriterion}
	In the setting of Proposition~\ref{prop:LatticeMorphism}, if $X$ is   Bruhat--Tits, then at least one of the following assertions holds:
	\begin{enumerate}[(I)]
		\item $X_0$ is also Bruhat--Tits.
		
		\item $\Gamma$ is a Galois lattice; moreover $\Gamma_0$ and $\Gamma$ both have an infinite quotient which is \{residually $p$\}-by-cyclic, where $p$ is the residue characteristic of the defining local field of $X$.
		
	\end{enumerate}
\end{proposition}
\begin{proof}
	Assume that $X$ is    Bruhat--Tits, and let $\mathbf G$ be the  simple adjoint algebraic group over a local field $k$ to which $X$ is associated. By Proposition~\ref{prop:Aut}(i), (v) and (vi), we have $\Aut(X) = \Aut(\mathbf G(k))$ and the group  $\Aut(\mathbf G(k))$ decomposes as an extension of $\Aut(\mathbf G)(k)$ by $\Aut_{\mathbf G}(k)$, which is a subgroup of the full Galois group $\Aut(k)$. Moreover, if $k$ has characteristic~$0$ then   $\Aut(k)$ is finite, while if $k$ has characteristic~$p$, then $\Aut(k)$ is an extension of a  pro-$p$ group by a finite cyclic group  (see Theorem~1 and the remark following it in \cite{Camina}). Therefore, composing the homomorphism $\varphi \colon \Gamma_0 \to \Gamma \leq \Aut(X)= \Aut(\mathbf G(k))$  with the quotient map $ \Aut(\mathbf G(k)) \to \Aut_{\mathbf G}(k)$, we obtain a homomorphism whose image is   \{residually $p$\}-by-cyclic. If that image is infinite, then $\Gamma$ is a Galois lattice and $\Gamma_0$ and $\Gamma$ have an infinite   \{residually $p$\}-by-cyclic quotient. Otherwise, using again Proposition~\ref{prop:Aut}(i),  the group $\Gamma_0$ has a finite index subgroup whose image under $\varphi$ is contained in the group  $\mathbf G(k) \leq \Aut(X)$. Thus $\varphi$ induces in particular a linear representation with infinite image of a finite index subgroup of $\Gamma_0$. By Theorem~\ref{nonlinear}, this implies that $X_0$ is Bruhat--Tits.
\end{proof}

\begin{remark}\label{rem:residually-2}
	As mentioned in Remark~\ref{rem:EmbeddingBuildings}, every proper quotient of $\Gamma$ and $\Gamma_0$ is finite, so that  in Case (II) of Proposition~\ref{prop:ExoticityCriterion}, both $\Gamma_0$ and $\Gamma$ must be \{residually $p$\}-by-cyclic. In particular, each of their finite index subgroups has a non-trivial abelianization. That condition will in particular exclude all the panel-regular lattices that are virtually perfect (see for example the lattice $\Gamma_0$ in Lemma~\ref{lem:G_0} below; many more perfect panel-regular lattices can be constructed using Theorem~\ref{thm:Essert}).
	
	On the other hand, a panel-regular lattice can be \{residually $p$\}-by-cyclic. A concrete example is provided by the group
	$$\Gamma_2 =  \langle s_0, s_1, s_2 \mid s_0^7, s_1^7, s_2^7, s_0 s_1 s_2, s_0^3 s_1^3 s_2^3\rangle$$
	associated with the difference set $\{0, 1, 3\}$  in $\ZZ/7\ZZ$. As pointed out by J.~Essert \cite[Remark following Th.~5.8]{Essert}, that group is an index~$3$ subgroup of an arithmetic lattice $U$ in the group $G = \SL_3\big(\FF_2(\!(t)\!) \big)$ considered by K\"ohler--Meixner--Wester \cite{KMW2}. The lattice $U$ is in fact contained in  $\SL_3\big(\FF_2[t, (t^2+t+1)^{-1}]\big)$ (see  \cite[Theorem~1]{KMW2}), and it is pointed out in \cite[\S3]{KMW2}    that the image of $\Gamma_2$ under the congruence quotient induced by the evaluation   $t \mapsto 0$   is cyclic of order~$7$. By Platonov's theorem \cite{Platonov}, the kernel of that congruence quotient is residually~$2$. Thus $\Gamma_2$ is \{residually~$2$\}-by-\{cyclic of order~$7$\}.
\end{remark}

\subsection{Examples in characteristic~$2$}

We shall now specify a choice of $\Gamma_0$ which is suitable for an application of Proposition~\ref{prop:ExoticityCriterion}.

\begin{lemma}\label{lem:G_0}
	Consider the  set $D_0 = \{0, 1, 3\}$, which  is a difference set in $\ZZ/7\ZZ$, and the panel-regular lattice
	$$\Gamma_0 = \langle s_0, s_1, s_2 \mid s_0^7, s_1^7, s_2^7, s_0 s_1 s_2^3, s_0^3 s_1^3 s_2\rangle$$
	and the associated  $\widetilde A_2$-building $X_0$ of order~$2$ afforded by Theorem~\ref{thm:Essert}. Then:
	\begin{enumerate}[(i)]
		\item $X_0$ is not Bruhat--Tits.
		
		\item The abelianization $\Gamma_0/[\Gamma_0, \Gamma_0]$ has order~$7$, and the derived group  $[\Gamma_0, \Gamma_0]$ is perfect.
		
		\item The assignments $s_0 \mapsto s_0^2$, $s_1 \mapsto s_1^2$, $s_2 \mapsto s_0 s_1$ extend to an automorphism $\alpha \in \Aut(\Gamma_0)$ of order~$3$. The extension $\Gamma_0 \rtimes \langle \alpha \rangle$ acts freely and transitively on the set of chambers of $X_0$ (and coincides with the group $G_4$ from \cite[Th.~2.5]{Ronan}).
	\end{enumerate}
\end{lemma}

\begin{proof}
	By \cite[Theorem~2.5]{Ronan}, there are exactly~$4$ possible groups acting freely and transitively on the chambers of an $\widetilde A_2$-building, denoted there by $G_1, G_2, G_3$ and $G_4$. (Those groups and their structure are also discussed in \cite{Tits_cours85},\cite{Tits_Andrews},  \cite{Tits_balls}, \cite[\S C.3]{Kantor84} and \cite{KMW}.) Using the explicit presentations from \cite[Theorem~C.3.6]{Kantor84}, one sees that $G_4$ is isomorphic to $\Gamma_0 \rtimes \langle \alpha \rangle$. This proves (iii).
	Assertion (i) follows from \cite[Cor.~1]{Tits_balls}, and assertion (ii)  from \cite[Prop.~3]{KMW}.
\end{proof}

\begin{proof}[Proof of Corollary~\ref{cor:ExistenceExotic}]
Let $q$ be a power of $2$ that is not a power of $8$, and let $n=q^2+q+1$. By Lemma~\ref{lem:SubDiffSets}, there is a difference set $D \subset \ZZ/n\ZZ$   containing $\{0, n/7, 3n/7\}$.
	
	We pick permutations $\pi_1, \pi_2 \in \Sym(D)$ fixing $0$ such that $\pi_1(n/7) = n/7$, $\pi_1(3n/7) = 3n/7$, $\pi_2(n/7) = 3n/7$ and $\pi_2(3n/7) = n/7$. Let then $\Gamma$ be the panel-regular lattice associated with $D$, $\pi_1$ and $\pi_2$ via Theorem~\ref{thm:Essert}. Its building $X$ has order $q$. Moreover, Proposition~\ref{prop:LatticeMorphism} provides a homomorphism $\Gamma_0 \to \Gamma$, where $\Gamma_0$ is the panel-regular lattice from Lemma~\ref{lem:G_0}. Since $\Gamma_0$ has a perfect subgroup of index~$7$, for any prime $p$, every \{residually $p$\}-by-cyclic quotient of $\Gamma_0$ has order at most~$7$. Therefore, Proposition~\ref{prop:ExoticityCriterion} ensures that $X$ is not Bruhat--Tits.
\end{proof}

\appendix

\section{On the linearity of finitely generated groups}\label{app:RingLinear}

In this appendix we present various characterizations of the existence of a finite-dimensional linear representation with non-amenable image for general  finitely generated groups. The proof is based on a trick invented by Tits in his proof of the so called Tits alternative, and on its strengthening due to Breuillard and Gelander.

\begin{theorem} \label{thm:ringlinear}
For a finitely generated group $\Gamma$ the following conditions are equivalent.
\begin{enumerate}[(1)]

\item There exists a commutative ring $R$ with unity,  a finitely generated projective $R$-module $M$ and a group homomorphism $\Gamma\to \Aut_R(M)$ whose image  is not virtually solvable.

\item There exists a commutative ring $R$ with unity,  an integer  $n>0$ and group homomorphism $\Gamma\to \GL_n(R)$ whose image  is not virtually solvable.

\item There exists a field $K$, an integer  $n>0$ and group homomorphism $\Gamma\to \GL_n(K)$ whose image  is not virtually solvable.

\item There exists a finite index subgroup $\Gamma_0<\Gamma$, a field $K$, a connected adjoint $K$-simple algebraic group $\mathbf{G}$ defined over $K$ and and group homomorphism $\Gamma_0\to \mathbf{G}(K)$ whose image is Zariski dense.

\item There exists a finite index subgroup $\Gamma_0<\Gamma$, a local field $k$, a connected adjoint $k$-simple algebraic group $\mathbf{G}$ defined over $k$ and group homomorphism $\Gamma_0\to \mathbf{G}(k)$ whose image is  Zariski dense and unbounded (i.e. with non-compact closure).

\end{enumerate}
\end{theorem}

\begin{proof}
The implications $(5) \Rightarrow (4)\Rightarrow (3)\Rightarrow (2)\Rightarrow (1)$ are all clear (recalling that the Zariski closure of a solvable group is solvable, see Corollary~1 in \S2.4 of Chapter~1 in \cite{Borel}). We will prove their converses.

\medskip \noindent
$(1)\Rightarrow (2):$
We fix a commutative unital ring  $R$, a finitely generated projective $R$-module $M$ and a group homomorphism $\phi \colon \Gamma\to \Aut_R(M)$ whose image  is not virtually solvable.
We let $N$ be an $R$-module such that $M\oplus N\simeq R^n$ for some positive integer $n$
and consider the map $\phi':\Gamma\to \Aut(M\oplus N)\simeq \GL_n(R)$ given by $\phi'(\gamma)(m,n)=(\phi(\gamma)m,n)$.

\medskip \noindent
$(2)\Rightarrow (3):$
We fix a commutative unital ring  $R$, a positive integer  $n$ and group homomorphism $\phi \colon\Gamma\to \GL_n(R)$ whose image is not virtually solvable.
Since $\Gamma$ is finitely generate we assume as we may that $R$ is finitely generated, hence Noetherian by Hilbert's Basis theorem (see \cite[\S7.10]{Jacobson}).
We let $\mathfrak{n}$ denote the \textbf{nil radical} of $R$, i.e. the set of nilpotent elements of $R$.
Since  $R$ is  Noetherian, every ideal is finitely generated. Therefore there is a natural $d$ such that $\mathfrak{n}^d=\{0\}$. By a theorem of Krull \cite[Th.~7.1]{Jacobson}, the nil radical $\mathfrak n$ is the intersection of all prime ideals. Moreover,  a theorem of Noether ensures that the set of {minimal} prime ideals is finite (see Exercise 1.2 on p.~47 in \cite{Eisenbud}). In particular, denoting the minimal prime ideals by $\mathfrak p_1, \dots, \mathfrak p_m$, we infer that $\mathfrak n = \bigcap_{i=1}^m \mathfrak p_i$.

For every integer $e \geq 0$ we denote by $G(e)$ the kernel of the natural homomorphism $GL_n(R)\to GL_n(R/\mathfrak{n}^e)$.
Observe that for all $e,f \geq 0$, we have $[G(e),G(f)] \leq G(e+f)$. Moreover we have $G(e)\leq G(f)$ if $e \geq f$.
It follows that for $e\geq 1$, the quotient $G(1)/G(e)$ is nilpotent.
In particular, the group $G(1)\simeq G(1)/G(d)$ is nilpotent.
Let us denote by $K_i$ the fraction field of $R/\mathfrak{p}_i$ and by $\phi_i \colon \Gamma\to \GL_n(K_i)$ the homomorphism obtained by composing $\phi$ with the natural homomorphism $\GL_n(R)\to\GL_n(R/\mathfrak{p_i})\to\GL_n(K_i)$.
Since the image of $\Gamma$ in $\GL_n(R)$ is not virtually solvable and the kernel of $\GL_n(R)\to \prod_{i =1}^m \GL_n(K_i)$, which equals $G(1)$, is nilpotent, we conclude that for some $i$, the image of $\Gamma$ under $\phi_i$ is not virtually solvable.

\medskip \noindent
$(3)\Rightarrow (4):$
We fix a field $K$, an integer $n > 0$ and group homomorphism $\phi:\Gamma\to \GL_n(K)$ whose image  is not virtually solvable.
We let $\mathbf{H}$ be the Zariski closure of $\phi(\Gamma)$. By \cite[AG, Theorem~14.4]{Borel} $\mathbf{H}$  is defined over $K$ and $\phi(\Gamma)<\mathbf{H}(K)$.
We denote by $\mathbf{H}^0$ the identity connected component in $\mathbf{H}$
and observe that $\mathbf{H}^0(K)<\mathbf{H}(K)$ is of finite index.
We set $\Gamma_0=\phi^{-1}(\mathbf{H}^0(K))$ and deduce that $\Gamma_0<\Gamma$ is of finite index.
Note that the Zariski closure of $\phi(\Gamma_0)$ is $\mathbf{H}^0$, as it is contained in the connected group $\mathbf{H}^0$ and it is of finite index in $\mathbf{H}$.
We let $\mathbf{R}$ be the solvable radical of $\mathbf{H}^0$ and set $\mathbf{L}=\mathbf{H}^0/\mathbf{R}$.
Since $\phi(\Gamma)$ is not virtually solvable, the group $\phi(\Gamma_0)$ is not solvable, thus $\mathbf{R}\neq\mathbf{H}^0$ and $\mathbf{L}$ is non-trivial. Upon replacing $\mathbf L$ by its quotient modulo its center, which is finite, we may further assume that $\mathbf L$ is centerless. Thus $\mathbf{L}$ is a connected, adjoint, semisimple $K$-group and the image of $\phi(\Gamma_0)$ under $\mathbf{H}^0(K)\to\mathbf{L}(K)$  is Zariski dense in $\mathbf{L}$.
The adjoint group $\mathbf{L}$ is $K$-isomorphic to a direct product of connected, adjoint, $K$-simple $K$-groups.
We let $\mathbf{G}$ be one of these $K$-simple factors.
We let $\phi_0 \colon \Gamma_0\to \mathbf{G}(K)$ be the composition of $\phi$ with the natural map $\mathbf{H}^0(K)\to \mathbf{L}(K)\to \mathbf{G}(K)$
and conclude that $\phi_0(\Gamma_0)$ is Zariski dense in the connected adjoint $K$-simple $K$-algebraic group $\mathbf{G}$.

\medskip \noindent
$(4)\Rightarrow (5):$
We fix a finite index subgroup $\Gamma_0<\Gamma$, a field $K$, a connected adjoint $K$-simple $K$-algebraic group $\mathbf{G}$ and and group homomorphism $\Gamma_0\to \mathbf{G}(K)$ with a Zariski dense image.
We choose an injective $K$-morphism $\mathbf{G}\to\GL_n$ for some $n$ and let $\psi \colon \Gamma_0\to\GL_n(K)$ be the composition of $\phi$ with the homomorphism $\mathbf{G}(K)\to \GL_n(K)$.
We let $I$ be set of matrix coefficients of $\psi(\Gamma_0)$. Since $\Gamma_0$ is Zariski dense in $\mathbf{G}(K)$, the field $K$ is infinite and the set $I$ must be infinite as well. Let moreover $R$ be the ring generated by $I$ and $F$ be its fraction field.
Since $\phi(\Gamma_0)$ is Zariski dense in $\mathbf{G}$ and is contained in $\mathbf{G}(F)$ we conclude by \cite[AG, Theorem 14.4]{Borel} that $\mathbf{G}$ is defined over $F$.
The group $\Gamma$ is finitely generated, hence so is its finite index subgroup $\Gamma_0$.
It follows that $R$ is a finitely generated domain. By \cite[Lemma 2.1]{BreuillardGelander-Tits} there exists a local field $k$ and an embedding $F\hookrightarrow k$ such that the image of $I$ in $k$ is unbounded.
Since $\mathbf{G}$ is defined over $F$ we get that it is also defined over $k$, via this embedding.
The composition of $\phi \colon\Gamma_0\to \mathbf{G}(F)$ with the homomorphism $\mathbf{G}(F)\to\mathbf{G}(k)$ induced by the embedding $F\hookrightarrow k$ is again denoted by $\phi$.
We already know that $\phi(\Gamma_0)$ is Zariski dense in $\mathbf{G}$.
Since the image of $I$ is unbounded in $k$, we see that $\psi(\Gamma_0)$ is unbounded in $\GL_n(k)$, so that $\phi(\Gamma_0)$ is unbounded in $\mathbf{G}(k)$.
\end{proof}

\begin{remark}
The equivalence between $(2)$ and $(3)$ in Theorem~\ref{thm:ringlinear} remains true  if one replaces the phrases ``virtually solvable'' by either
``solvable'', ``finite'' or ``trivial''. The proof for the ``solvable'' case is the same, mutatis-mutandis.
For the ``finite'' and ``trivial'' cases we need one extra ingredient: the argument above fails only if the image of $\phi$ is virtually nilpotent, but every finitely generated virtually nilpotent group has a faithful representation over $\mathbf{Q}$.
\end{remark}

The following is a convenient tool for proving non-linearity of groups.

\begin{corollary} \label{cor:ringlinear}
Let $\Gamma$ be a finitely generated group such that every finite index subgroup has a finite abelization (e.g a group with property (T)).
Assume that there exists a finite index subgroup $\Gamma_1<\Gamma$ such that there is no homomorphism $\Gamma_0\to \mathbf{G}(k)$ with an unbounded and Zariski dense image,
where $\Gamma_0<\Gamma_1$ is a finite index subgroup, $k$ a local field and $\mathbf{G}$ a connected adjoint $k$-simple $k$-algebraic group.
Then every homomorphism $\Gamma\to \GL_n(R)$, where $R$ is a commutative ring with unity, has a finite image.
\end{corollary}

\begin{proof}
By assumption, the group $\Gamma_1$ does not satisfy condition $(5)$ of Theorem~\ref{thm:ringlinear}, thus it also does not satisfy condition $(2)$ of that theorem. 
It follows that also $\Gamma$ does not satisfy condition $(2)$.
Fixing a commutative ring with unity $R$,  an integer  $n>0$ and group homomorphism $\phi:\Gamma\to \GL_n(R)$, we get that $\phi(\Gamma)$ is finite,
otherwise we could find a finite index subgroup with an infinite
abelinization, as $\phi(\Gamma)$ is virtually solvable.
\end{proof}


\section{The Hopf argument} \label{app-hopf}

The Hopf argument is a familiar ergodicity criterion, see   \cite{Hopf}  and  \cite{Coudene}. 
In our paper we need to use it in a form which is slightly more abstract than the standard one.
This appendix is devoted to the formulation and proof of the Hopf argument in our needed form.
In fact, as no further complexity is introduced, we allowed ourselves a little more general form than actually needed:
for all applications we are aware of the group $T$ considered below is abelian and the collection $I$ consists of two elements.

As the formulation of the theorem below is not much simpler than its proof,
we advice the reader to have in mind the following classical setting:
$M$ is a hyperbolic surface whose universal cover is identified with the hyperbolic plane $\mathbf{H}^2$, $V=T^1(M)$ and $T=\mathbf{R}$ acts on it by the geodesic flow, $Z=T^1(\mathbf{H}^2)$ is endowed with the geodesic flow action of $T$ and the deck transformation action of $\Lambda=\pi_1(M)$,
$W_1=W_2$ is the ideal boundary of $\mathbf{H}^2$, $Y=W_1\times W_2$ and $\psi \colon Z\to Y$ is given by the asymptotic hitting points of the geodesic flow.

Two sequences $(a_n)$ and $(b_n)$ in a compact space $V$ are called \textbf{proximal}  if for every neighborhood $U$ of the diagonal in $V\times V$, we have $(a_n,b_n)\in U$ for all but finitely many $n$'s.

\begin{theorem}[The Hopf argument] \label{hopf}
Let $T$ be an amenable group and $\Lambda$ be any group.
Let $Z$ be a locally compact topological space endowed with commuting actions of $T$ and $\Lambda$.
Assume both actions (separately) $\Lambda \curvearrowright Z$ and $T\curvearrowright Z$ are proper, let $V=Z/\Lambda$ and $Y=Z/T$
and let $\phi \colon Z\to V$ and $\psi \colon Z\to Y$ be the quotient maps.
Assume $V$ is compact.
Note that there are natural actions $T\curvearrowright V$ and $\Lambda\curvearrowright Y$,
the map $\phi$ is $\Lambda$-invariant and $T$-equivariant and the map $\psi$ is $T$-invariant and $\Lambda$-equivariant.
Let $\{W_i\}_{i\in I}$ be a collection of $\Lambda$-spaces and for each $i\in I$ let $\pi_i \colon Y \to W_i$ be a $\Lambda$-equivariant map, as described in the following diagram:
\begin{align} \label{diad:hopf}
 \xymatrix{
& Z \ar[dl]_\phi\ar[dr]^\psi& &\\
V & & Y \ar[dl]_{\pi_i}\ar[dr]^{\pi_{j}}& \\
&  W_i& \cdots & W_j
}
\end{align}
Assume for each $i\in I$ there exists a central element $t_i\in T$ such that
for all  $z,z'\in Z$ satisfying $\pi_i\psi(z)=\pi_i\psi(z')$ there exist $t,t'\in T$ such that the two sequences $(t_i^nt\phi(z))_{n\in\mathbb{N}}$ and $(t_i^nt'\phi(z'))_{n\in\mathbb{N}}$ are proximal in $V$.

Let $m_Z$ be a $T\times \Lambda$-invariant measure class on $Z$ and denote its pushes to the various factors in the diagram by $m_V,m_Y$ and $m_{W_i}$.
Assume further:
\begin{itemize}
\item the measure class $m_V$ contains a $T$-invariant probability measure.
\item Any function in  $\bigcap_{i\in I} \pi_i^*(L^\infty(W_i,m_{W_i}))^\Lambda$ is essentially constant.

\end{itemize}
Then the two actions $T \curvearrowright (V,m_V)$ and $\Lambda\curvearrowright (Y,m_Y)$ are ergodic.
\end{theorem}

\begin{proof}
First note that
\[ L^\infty(V,m_V)^T \simeq L^\infty(Z,m_Z)^{T\times \Lambda}\simeq   L^\infty(Y,m_Y)^\Lambda, \]
thus the ergodicity of $T \curvearrowright (V,m_V)$ is equivalent to that of $\Lambda\curvearrowright (Y,m_Y)$.
We will prove the former.

Since the measure class $m_V$ contains a $T$-invariant probability measure (which by an abuse of notation we keep calling $m_V$), it is enough to prove that every $T$-invariant function class in $L^2(V,m_V)$ is essentially constant,
or equivalently, that the image of the orthogonal projection map $p_T \colon L^2(V,m_V)\to L^2(V,m_V)^T$ consists of the essentially constant functions.
To this end, since $C(V)$ is dense in $L^2(V,m_V)$, it is enough to show that for every $f\in C(V)$, $p_T(f)$ is essentially constant.
We now fix $f_0\in C(V)$ and argue to show that $p_T(f_0)$ is essentially constant.

Fix a F\o lner sequence $F_n$ in the amenable group $T$ and define, for every bounded measurable function $f\in B(V)$ and for every $x\in V$,
\[ \bar{f}(x)=\limsup_{n\to \infty}\frac{1}{|F_n|}\int_{F_n} f(sx)ds, \]
where the integration is with respect to a fixed left Haar measure on $T$ and $|F_n|$ denotes the measure of $F_n$.
We stress the fact that $\bar{f}$ is an actual function and it is everywhere $T$-invariant, while on the contrary $p_T(f)$ is a class of functions.
By the Mean Ergodic Theorem for amenable groups \cite{greenleaf}, $\bar{f}$ is a function in the class $p_T(f)$.

Fix $i\in I$ and let $t_i\in T$ be a central element with the property described in the formulation of the theorem.
Let $p_i \colon L^2(V,m_V)\to L^2(V,m_V)^{t_i}$ be the orthogonal projection onto the space of $t_i$-invariant functions
and note that by the centrality of $t_i$, $p_Tp_i=p_T$.
Define for every $x\in V$
\[ f_i(x)=\limsup_{N\to \infty}\frac{1}{N}\sum_{n=1}^N f_0(t_i^nx). \]
We obtain a bounded measurable function $f_i\in B(V)$.
By Birkhoff Ergodic Theorem $f_i$ is  in the class $p_i(f_0)$.

Fix $z,z'\in Z$ satisfying $\pi_i\psi(z)=\pi_i\psi(z')$.
By the defining property of $t_i$ there exist $t,t'\in T$ such that the two sequences $(t_i^nt\phi(z))_{n\in\mathbb{N}}$ and $(t_i^nt'\phi(z'))_{n\in\mathbb{N}}$ are proximal in $X$.
It follows that for every $s\in T$ the sequences $(t_i^nst\phi(z))_{n\in\mathbb{N}}$ and $(t_i^nst'\phi(z'))_{n\in\mathbb{N}}$ are proximal
(we have used here the centrality of $t_i$ again)
and by the (uniform) continuity of $f_0$, it is easy to see that $f_i(st\phi(z))=f_i(st'\phi(z'))$.
We deduce that $\bar{f}_i(t\phi(z))=\bar{f}_i(t'\phi(z'))$.
Since $\bar{f}_i$ is $T$-invariant we get $\bar{f}_i(\phi(z))=\bar{f}_i(\phi(z'))$.
Note that $\bar{f}_i$ is in the class of $p_Tp_i(f_0)=p_T(f_0)$.
In particular, its class is independent of $i\in I$.

Under the obvious isomorphisms
\[ B(V)^T \simeq B(Z)^{T\times \Lambda}\simeq  B(Y)^\Lambda \]
$\bar{f}_i$ corresponds to a function $g_i\in  B(Y)^\Lambda$. By the previous argument, if $y=\psi(z)$ and $y'=\psi(z')$ satisfy $\pi_i(y)=\pi_i(y')$ then $g_i(y)=g_i(y')$. Hence $g_i$
is a pull back of a function in $B(W_i)$.
Since the maps $\phi$ and $\psi$ are measure class preserving, we obtain that the class of $g_i$ is independent of $i$.
Denote this class by $g$ and get $g\in \bigcap_{i\in I} \pi_i^*(L^\infty(W_i,m_{W_i}))^\Lambda$.
Then by assumption $g$ is essentially constant.
We conclude that $p_T(f_0)$ is essentially constant, and the proof is complete.
\end{proof}

\bibliographystyle{alpha}
\bibliography{biblio}

\end{document}